%% file: Main.tex
\documentclass[leqno, 10pt]{article}
\usepackage{amsmath,amsthm,amssymb,amscd,stmaryrd}
\usepackage{ascmac}
\usepackage[mathscr]{eucal}
\usepackage[all,cmtip]{xy}
\usepackage{tikz-cd}
\usetikzlibrary{decorations.pathmorphing}
\usetikzlibrary{decorations.markings}
\usepackage{enumerate}
\usepackage{graphicx}
\usepackage{xcolor}
\usepackage{microtype}
\usepackage{verbatim}
\usepackage{version}
\usepackage{color}
\usepackage{float}
\usepackage{mathtools}
\usepackage[font=small,skip=0pt]{caption}
\usepackage[colorlinks=true,
            linkcolor=blue,   
            citecolor=blue,   
            urlcolor=blue     
           ]{hyperref}
\usepackage{nicematrix}
\usepackage{hhline}
\usepackage{enumitem}
\usepackage[twoside]{fancyhdr}

\usepackage{amsfonts}
\usepackage{geometry}
\usepackage{mathrsfs}
\usepackage{cite}

\usepackage{tocloft,titling}

\theoremstyle{plain}
\newtheorem{Theorem}[equation]{Theorem}
\newtheorem{Corollary}[equation]{Corollary}
\newtheorem{Lemma}[equation]{Lemma}
\newtheorem{Proposition}[equation]{Proposition}
\newtheorem*{Theorem*}{Theorem}

\theoremstyle{definition}
\newtheorem{Definition}[equation]{Definition}
\newtheorem{Example}[equation]{Example}

\theoremstyle{remark}
\newtheorem{Remark}[equation]{Remark}
\numberwithin{equation}{section}

\newcommand{\CC}{{\mathbb C}}
\newcommand{\ZZ}{{\mathbb Z}}
\newcommand{\RR}{{\mathbb R}}
\newcommand{\AAA}{{\mathbb A}}
\newcommand{\NN}{\mathbb N}
\newcommand{\QQ}{\mathbb Q}

\newcommand{\GL}{\operatorname{GL}}
\newcommand{\LU}{\operatorname{U}}
\newcommand{\lieu}{\operatorname{\mathfrak{u}}}
\newcommand{\Lie}{\operatorname{Lie}}
\newcommand{\Hur}{\operatorname{Hur}}
\newcommand{\Spec}{\operatorname{Spec}\nolimits}
\newcommand{\Proj}{\operatorname{Proj}\nolimits}
\newcommand{\End}{\operatorname{End}}
\newcommand{\Hom}{\operatorname{Hom}}

\newcommand{\Ker}{\operatorname{Ker}}

\newcommand{\Ima}{\operatorname{Im}}

\newcommand{\sdim}{\mathop{\text{\rm dim}}\nolimits}
\newcommand{\codim}{\mathop{\text{\rm codim}}\nolimits}

\newcommand{\GIT}{/\!\!/}
\newcommand{\II}{\mathcal{I}}
\newcommand{\EE}{\mathcal{E}}
\newcommand{\GG}{\mathcal{G}}
\newcommand{\MM}{\mathbb{M}}
\newcommand{\HR}{/\!\!\GIT}
\newcommand{\IS}{\mathcal{I}^s}
\newcommand{\M}{\mathcal{M}}
\newcommand{\TM}{\widetilde{{\mathcal{M}}}}
\newcommand{\bdot}[1]{\overset{\bullet}{#1}}
\newcommand{\NA}{\mathcal{A}}
\newcommand{\HH}{\mathbb{H}}
\newcommand{\T}{\operatorname{T}}
\newcommand{\Rep}{\mathrm{\mathrm{Rep}}}

\usepackage{xcolor}

\title{\bfseries Quiver description of bow varieties in general type}
\author{Tiziano Gaibisso}
\date{}

\begin{document}
\maketitle
\begin{sloppypar}
\begin{abstract}
We provide a quiver description for Cherkis bow varieties in arbitrary type. We explain how this generalizes the construction of Nakajima quiver varieties. We give criteria for stability, non-emptiness, smoothness and discuss deformations. In the appendix, we discuss the relation between the quiver description and the original Cherkis' construction of bow varieties.
\end{abstract}

\tableofcontents

\section{Introduction}
In \cite{cherkis2011instantons}, Cherkis defined an interesting family of spaces, starting from a generalization of quivers called bows, in order to construct moduli spaces of instantons on ALF spaces. We will refer to these spaces as \textit{bow varieties}. Let us recall that a quiver $Q$ is a (finite) directed graph. Then, a \textit{bow} is defined as a pair $B=(\II,\EE)$ consisting of a set $\II$ of closed intervals, called \textit{wavy lines}, and a set $\EE$ of \textit{arrows} going from the endpoints of some wavy line to the starting points of some wavy line (not necessarily distinct). For instance:
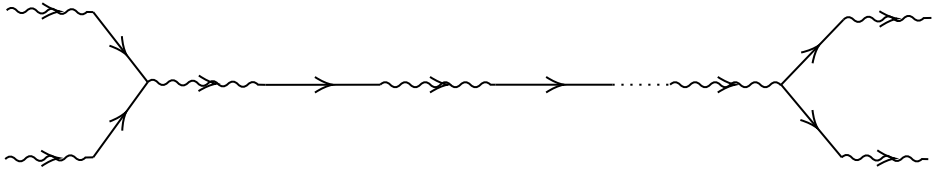
\begin{center}
    \input{Affine-D-Bows}
\end{center}
We notice that any bow $B$ has a natural \textit{underlying quiver} obtained by replacing wavy lines with vertices.
In other words, as noticed by Cherkis, we can think of a general bow as limiting to a quiver as interval lengths tend to zero. This yields a natural definition of \textit{Dynkin bow}.
Cherkis introduced an analogous of the framed dimension vector of a quiver, which he called a ``representation of a bow".
Given a representation of a bow, Cherkis considered an infinite-dimensional affine hyperk\"ahler space $\overline{\mathcal{M}}$ together with a Hamiltonian action by an infinite-dimensional group $\GG^\RR$.
Hence, bow varieties are defined as (infinite-dimensional) hyperkähler reductions, i.e. $\mathcal{M}=\overline{\mathcal{M}}\HR\GG^\RR$.
Finally, Cherkis noticed that one may divide the above reductions into two steps (see \cite[Section 8]{cherkis2011instantons}). Firstly, one consider a particular infinite-dimensional normal subgroup $\GG^\RR_{00}$ of $\GG^\RR$ and the hyperk\"ahler reduction $\TM=\overline{\mathcal{M}}\HR\GG^\RR_{00}$. 
Then, the quotient $\TM$ is a finite dimensional hyperk\"ahler manifold which can be described as a product of hyperk\"ahler vector spaces and moduli spaces of solutions to Nahm equations over an interval. The second step consists to take hyperk\"ahler reductions of $\TM$ by the residual action of the finite dimensional group $\GG^\RR/\GG^\RR_{00}$. In this way we have a realization of bow varieties as finite dimensional hyperk\"ahler reductions. This point of view has many advantages, for instance it allows us to use Kempf-Ness type theorems to provide a description of bow varieties as GIT quotients (see Appendix \ref{Appendix A}).

In \cite{Tak16} and \cite{NT17}, Nakajima and Takayama initiated algebro-geometric study of Cherkis bow varieties generated by Dynkin bows of affine type $A$ (and hence of finite type $A$ too), by introducing a linear algebraic description of these varieties, called \textit{quiver description}. The idea is to combine some known results on  moduli spaces of solutions to Nahm equations over an interval (\cite{Don84}, \cite{hurtubise1989classification}) to realize them as complex algebraic symplectic varieties via quivers.
Then, we can construct bow varieties by performing symplectic reductions via GIT. This approach seems appropriate to address algebraic and geometric properties analogously to Nakajima quiver varieties. In this paper, we generalize this approach to any type of bow. Let us briefly recall the quiver description of bow varieties (for more details, see Section \ref{S3}). Fix a bow.
In the quiver description, a representation of a bow is replaced by a \textit{bow diagram}, which consists of a collection, $\Lambda$, of points in the interior of wavy lines, called \textit{x-points}, and an assignment of a non-negative natural number for every segment of a wavy line cut out by x-points. For instance, if $B$ is a bow as above, a bow diagram over $B$ is given by:
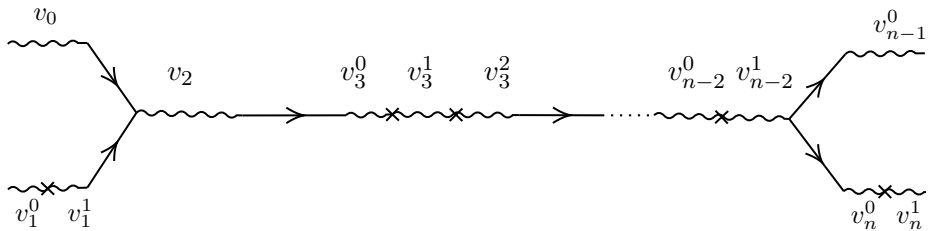
\begin{center}
    \input{diagram-typeD}
\end{center}
The vector $\underline{v}$ of natural numbers is called the \textit{dimension vector}. Given a bow diagram, we define a symplectic complex affine algebraic variety $\TM$ as the product of:
\begin{itemize}
    \item \textit{Two-way parts}: to any arrow $t(e)\xlongrightarrow{e}h(e)\in\EE$, we associate the symplectic vector space $\M^e\coloneqq \T^*\Hom(\CC^{v_{t(e)}},\CC^{v_{h(e)}})$.
    \item  \textit{Triangles}: to any \input{AdjSeg2}, we associate a complex symplectic affine variety $\M^x$ (see \S\ref{triangles}) defined as a locally closed subset of the complex vector space of ``triangles":$$\
    \vcenter{\hbox{\input{Tri-x}}}.$$
\end{itemize}
Given a bow diagram, we define $\GG=\prod\GL(v_\zeta)$, where $\zeta$ runs over the set of segments. Then, there is a natural Hamiltonian action of $\GG$ on $\TM$ given by changing bases. This gives rise to the notion of stability and deformation parameters, respectively $\theta$ and $\lambda$, and the associated bow variety defined as the Hamiltonian reduction $\mathcal{M}=\TM\HR_{\lambda,\theta}\GG$.

One of the reasons we are interested in Cherkis bow varieties is because they can be regarded as generalizations of Nakajima quiver varieties (see Section \ref{S5}).
We need to briefly recall the construction of Nakajima quiver varieties as GIT quotients (see \S\ref{NQV}). Let $Q$ be a quiver with vertex set $Q_0$ and $(v,w)\in\NN^{Q_0}\times\NN^{Q_0}$ a \textit{framed dimension vector}. To these combinatorial data is associated a Hamiltonian action of the group $G_v=\prod\limits_{i \in Q_0}GL(v_i)$ on the cotangent bundle of a vector space denoted by $\T^*\mathrm{\mathrm{Rep}}(Q,v,w)$. 
Then, for any pair of stability and deformation parameters, respectively $\theta \in \ZZ^{Q_0}$ and $\lambda\in\CC^{Q_0}$, we define the Nakajima quiver variety associated to $(Q,v,w)$ of parameters $(\lambda,\theta)$ as the Hamiltonian reduction:
    \begin{equation*}
        \mathcal{M}_{\lambda,\theta}\bigl(Q,v,w\bigr)=\T^*\mathrm{Rep}(Q,v,w)\HR_{\lambda,\theta}G_v.
    \end{equation*}
Coming back to bow varieties, let us consider a particular class of bow diagrams satisfying a symmetry condition, called \textit{cobalanced}. A bow diagram is cobalanced if, for any x-point, the dimensions on its left and on its right are equal. Bow varieties arising from cobalanced bow diagrams will be called \textit{cobalanced bow varieties}. 
In \cite{cherkis2011instantons}, Cherkis noticed that cobalanced bow varieties are isomorphic to Nakajima quiver varieties.
In this paper we prove this result using only the quiver description of bow varieties (Corollary \ref{maincor}) and as a corollary of a stronger result that relates the construction of Nakajima quiver varieties to that of bow varieties (see Theorem \ref{4.1}).
The first step consists of noticing that if $B$ is a bow and $Q$ is the underlying quiver, then cobalanced bow diagrams over $B$ are equivalent to framed dimension vectors over $Q$. Let us fix a cobalanced bow diagram $(B,\Lambda,v)$ and denote by $(Q,v,w)$ the corresponding quiver and framed dimension vector. The group $\GG$ can be regarded as the product of $G_v$ and a residual group $H$. Let $(\lambda,\theta)\in\CC^\II\times\ZZ^\II$ be a pair of, respectively, deformation and stability parameters with respect to the $G_v$-action. 
Then, the symplectic reduction to obtain the bow variety may be divided into two steps by performing two consecutive symplectic reductions, that is:
\begin{equation*}
    \mathcal{M}=\TM\HR\GG=\Bigl(\TM\HR_{0,0}H\Bigr)\HR_{\lambda,\theta} G_v,
\end{equation*}
Finally, we prove: 
\begin{Theorem*}[\ref{4.1}]
There is a $G_v$-equivariant isomorphism of affine varieties
\begin{equation*}\TM\HR_{0,0} H \cong \T^*\mathrm{\mathrm{Rep}}(Q,v,w).\end{equation*}
Moreover the isomorphism respects the symplectic forms and the moment maps. In particular, we have $\mathcal{M}_{\lambda,\theta}\bigl(B,\Lambda,v\bigr)\cong\mathcal{M}_{\lambda,\theta}\bigl(Q,v,w\bigr)$.
\end{Theorem*} 
Notice that parameters $\lambda$ and $\theta$ only relate to the action of $G_v$. 
One may wonder if parameters defined for the entire group $\GG=G_v\times H$ could provide cobalanced bow varieties which are not isomorphic to any Nakajima quiver varieties.
However, it follows from subsections \ref{s.ncb} and \ref{s.gdp}, that the resulting varieties can be constructed by using parameters $(\lambda,\theta)\in\CC^\II\times\ZZ^\II$ as above.
\\

The paper is organized as follows. In Section \ref{S2} we recall basic \textit{Geometric Invariant Theory} (GIT) and the construction of Nakajima quiver varieties via GIT.
In Section \ref{S3}, we present the quiver construction of bow varieties.
In Section \ref{S4}, we prove some basic facts about bow varieties for arbitrary types of bows. We start by providing a numerical criterion for semistability conditions. Then, we prove that two deformation parameters provide isomorphic level sets of the moment maps if they satisfy a particular numerical condition. These criteria allow us to restrict ourselves to considering only a particular class of deformation and stability parameters. Moreover, we show that the stable locus of bow varieties is always smooth. Combining this result on the stable locus with the numerical criterion for semistability we can easily obtain a sufficient numerical criterion for smoothness of bow varieties. Finally, we will provide a necessary (not sufficient) numerical conditions for emptiness of bow varieties.
In Section \ref{S5}, we recall the notion of cobalanced bow diagrams and prove the main result of this paper (Theorem \ref{4.1}) about the relation between Cherkis bow varieties and Nakajima quiver varieties.
In Appendix \ref{Appendix A}, we recall the construction of bow varieties as finite-dimensional hyperk\"ahler reduction \cite{cherkis2011instantons} and discuss the relation with the quiver description.

\section*{Acknowledgement}
The author is particularly grateful to Travis Schedler for his explanations and guidance throughout this work. The author also thanks Amihay Hanany for many helpful discussions and Hiraku Nakajima for helpful comments. This work is part of the author's PhD project at Imperial College London.

\section{Preliminaries}\label{S2}
In this section we will recall some general theory which will be used later in the paper.

\subsection{GIT}
Let us recall some basic facts on Geometric Invariant Theory (GIT). For more details see  \cite{mumford1994geometric} as well as the expository books \cite{mukai2003introduction} and \cite{Ki16}. 

Let $X$ be a complex affine variety and $\CC[X]$ its coordinate ring. 
Let $G$ be a reductive complex algebraic group acting on X. We define the \textit{classical GIT quotient of X by G} as $$X\GIT G \coloneqq    \Spec\CC[X]^G,$$ where $\CC[X]^G$ is the algebra of $G$-invariant functions on X. Let us recall that, since $G$ is reductive, $\CC[X]^G$ is finitely generated and $X\GIT G$ is an affine variety.
Let us recall a useful description of the underlying topological space of $X\GIT G$. Given $x \in X$ we denote the orbit of $x$ under the action of $G$ by $G\cdot x$ and its closure by $\overline{G\cdot x}$. 
The action of $G$ defines an equivalence relation on $X$ as follows. Given $x,x' \in X$, we say that $x \sim x'$ if and only if their orbits are \textit{closure-equivalent}, i.e. $\overline{G\cdot x} \cap \overline{G\cdot x'} \neq \emptyset$. 
\begin{Theorem}[\cite{Ki16}, Theorem 9.5]\label{theo2.1}
    There exists a homeomorphism between $X\GIT G$ and $X/\sim$. In particular, there is an explicit description of $X\GIT G$ as the topological space of closed $G$-orbits in $X$. 
\end{Theorem}
However, it is clear that this affine quotients are very restrictive in many cases; indeed they only depend on closed $G$-orbits in $X$. Therefore, we define some deformations of these quotients, called \textit{twisted GIT quotients} or simply \textit{GIT quotients}. Let us briefly recall their construction. Let $G^{\vee}$ be the character group of $G$. For $\chi \in G^\vee$ and $n\in\NN$, we define the space of $\chi^n$-\textit{semi-invariants} functions on $X$ by: 
$$\CC[X]^{\chi^n,G}=\bigl\{f\in\CC[X] \ \vert \ f(g\cdot x)=\chi(g)^{n}f(x) \text{ for all }g \in G, x \in X\bigr\}.$$ 
Then, the GIT quotient of $X$ by $G$ of \textit{stability parameter} $\chi$, is given by:
$$X\GIT_\chi G \coloneqq     \Proj\bigl(\bigoplus\limits_{n\geq0}\CC[X]^{\chi^n,G}\bigr).$$

Next, we recall the notion of $\chi$\textit{-semistability}.
\begin{Definition}
 A point $x \in X$ is called $\chi$-semistable if and only if there exists $n>0$ and $f \in \CC[X]^{\chi^n,G}$ such that $f(x)\neq0$. We will denote by $X^{\chi-ss}$ the subset of $\chi$-semistable points in $X$. We will refer to it as the \textit{semistable locus}.
\end{Definition}
We can now state a topological characterization, analogously to Theorem \ref{theo2.1}, for general $GIT$ quotients:
\begin{Proposition}\label{prop2.3}
    The semistable locus $X^{\chi-ss}$ is a $G$-invariant open subset of $X$. Furthermore, there is a homeomorphism between $X\GIT_\chi G$ and $X^{\chi-ss}/ \sim$. In particular, there is an explicit description of $X\GIT_\chi G$ as the topological space of the closed $G$-orbits in the semistable locus.
\end{Proposition}
\begin{Definition}
A point $x \in X$ is called $\chi$-stable if and only if it is semistable, its $G$-orbit is closed in $X^{\chi-ss}$ and the isotropy group is finite. We will denote by $X^{\chi-s}$ the subset of $\chi$-stable points in $X$. We will refer to it as the \textit{stable locus}.
\end{Definition}
Whenever the character is clear or irrelevant, we will denote the semistable and the stable locus by $X^{ss}$ and $X^s$, respectively.
\begin{Theorem}[\cite{Ki16}, Theorem 9.23]\label{2.5}
In the previous assumptions, let us additionally assume that $X^s$ is non-empty. Then, the following hold:
\begin{enumerate}
    \item $X^s$ is an open subset of $X^{ss}$.
    \item $X^{\chi-s}/G$ can be naturally identified with an open subset of $X\GIT_\chi G$.
    \item If $X$ is nonsingular and for every stable point $x$ the stabilizer $G_x$ is trivial, then $X^s/G$ is a non-singular open subvariety of $X\GIT_\chi G$ of dimension $\dim X^s/G= \dim X-\dim G$.
\end{enumerate}
\end{Theorem}
 In order to have some characterization of semistability conditions we are going to recall some criteria which are particularly useful in explicit computations (for an example, see \S\ref{E.1}). 
 The first criterion we will state is \textit{Mumford's numerical criterion}. Indeed, it is common to find specific versions of this criterion when working with particular classes of varieties. For instance, in the case of Nakajima's quiver varieties it takes the form of Theorem \ref{ncN}.
 More generally, for bow varieties we will prove that it has the form of Theorem \ref{ncB}. However, before stating the Mumford's numerical criterion, we need to introduce more notation. Let $\lambda : \CC^\times \rightarrow G$ be a one parameter subgroup of $G$ and $\chi\in G^\vee$ a character. Then, we define $\langle\lambda,\chi\rangle$ as the unique integer such that $\chi\circ\lambda(t)=t^{\langle\lambda,\chi\rangle}$. Given $x \in X$ we say that $\lim_{t\to0} \lambda(t)\cdot x$ exists in $X$ if
$t\mapsto\lambda(t)\cdot x$ extends to a morphism of algebraic varieties $\AAA^1 \rightarrow X$. In such a case, the limit is defined as the image of the origin $0 \in \AAA^1$. Let us now state the criterion as formulated in \cite[Section 2]{king1994moduli}.
\begin{Theorem}[Mumford's numerical criterion]\label{ncM}
  Under the previous assumptions, the following hold:
  \begin{enumerate}
       \item  A point $x \in X$ is $\chi$-semistable if and only if for every one parameter subgroup $\lambda$ such that $\lim_{t\to0} \lambda(t)\cdot x$ exists in $X$,  $\langle\lambda,\chi\rangle\geq0$.
       \item A point $x \in X$ is $\chi$-stable if and only if for every non-trivial one parameter subgroup $\lambda$ such that $\lim_{t\to0} \lambda(t)\cdot x$ exists in $X$,  $\langle\lambda,\chi\rangle >0$.
   \end{enumerate} 
\end{Theorem}
Next, there is another useful characterization of semistability due to King (see \cite[Section 2]{king1994moduli}). Given a $G$-character $\chi$, we lift the $G$-action on $X$ to a $G$-action on $X\times\CC$ by: $g\cdot(x,z)=(gx,\chi(g)^{-1}z)$ with $g\in G, \ x \in X, z \in \CC$.
\begin{Theorem}[King's criterion]\label{ncK}
    Under the previous assumptions, the following hold:
    \begin{enumerate}
        \item A point $x$ in $X$ is semistable if and only if the closure of the orbit $G\cdot(x,z)$ does not intersect the zero section $X\times\{0\}$ for all $z\neq0$.
        \item A point $x$ in $X$ is stable if and only if for all $z\neq0$ the orbit of $(x,z)$ in $X\times\CC$ is closed and the stabilizer $G_{(x,z)}$ is finite.
    \end{enumerate}
\end{Theorem}
Let us finish this subsection by recalling the construction of symplectic quotients via GIT (see \cite{Ki16} and \cite{thomas2005notes}). Let $G$ be a complex algebraic reductive group and $X$ a complex nonsingular affine variety with an algebraic symplectic structure. Let us assume that there is a Hamiltonian action of $G$ on $X$ with a moment map $\mu: X \rightarrow \mathfrak{g}^*\coloneqq    \Lie(G)^*$. Let $\chi \in G^\vee$ be rational $G$-character and $\lambda \in \mathfrak{g}^*$ a $G$-invariant element. Then, we have an action of $G$ on the fiber $\mu^{-1}(\lambda)$. So, we can define the \textit{symplectic quotient} (or Hamiltonian reduction) of $X$ by $G$ with \textit{stability parameter} $\chi$ and \textit{deformation parameter} $\lambda$, as:
\begin{center}
    $\mathcal{M}_{\chi, \lambda} \coloneqq     X\HR_{\chi, \lambda}G=\mu^{-1}(\lambda)\GIT_{\chi} G$.
\end{center}
It can be proved that $\mathcal{M}_{\chi,\lambda}$ inherits a Poisson structure from $X$. If the open subset induced by stable points is smooth, then the restriction of the Poisson structure to it is non-degenerate (see Theorem 9.53 \cite{Ki16}).

\subsection{Nakajima quiver varieties}\label{NQV}
The purpose of this subsection is to briefly recall the construction of Nakajima quiver varieties via GIT. We refer the interested reader to the original works by Nakajima \cite{Nak94}, \cite{nakajima1998quiver} as well as \cite{ginzburg2009lectures} and \cite{Ki16}.

Let $Q=(Q_0,Q_1)$ be a (finite) quiver, with $Q_0$ the (finite) set of vertices and $Q_1$ the set of arrows. Let $t(e)$ and $h(e)$ be, respectively, the starting and ending points of an arrow $e \in Q_1$, i.e. $t(e) \xlongrightarrow{e} h(e)$.

Given $v \in \mathbb{N}^{Q_0}$ a \textit{dimension vector} for $Q$, we define the vector space of $v$-dimensional $Q$-representations by 
\begin{equation*}
    \mathrm{\mathrm{Rep}}(Q,v)=\bigoplus\limits_{e:i\rightarrow j \in Q_1}\Hom(V_i,V_j)
\end{equation*} 
 where $V_i = \CC^{v_i}$ for any $i \in Q_0$. We define $G_v=\prod\limits_{i \in Q_0} \GL_{v_i}$ the associated \textit{gauge group}. Thus, we have a natural linear action of $G_v$ on $\mathrm{\mathrm{Rep}}(Q,v)$ given by changing bases, that is
\begin{equation}\label{eq2.9}
(g_i)_{i \in Q_0} \cdot (x_e)_{e \in Q_1} = (g_{h(e)}x_e g_{t(e)}^{-1})_{e \in Q_1},
\end{equation}
where $(g_i)_i \in G_v$ and $(x_e)_e \in \mathrm{\mathrm{Rep}}(Q,v)$.  We notice that two representations are in the same $G_v$-orbit if and only if they are isomorphic $Q$-representations. 

Let us denote $\overline{Q}$ the double quiver of $Q$, that is the quiver obtained by adding to $Q$ the reverse arrow for each arrow in $Q$. Given $e \in Q_1$, we denote $e^*$ the reverse arrow in $\overline{Q}$. Thanks to the perfect trace pairing $\Hom(U,V) \times \Hom(V,U) \rightarrow \CC$, follows that $T^{*}\mathrm{\mathrm{Rep}}(Q,v)$ is nothing but the space of representations of dimension $v$ of the double quiver $\overline{Q}$ of $Q$, namely:
\begin{equation}\label{eq2.10}
T^{*}\mathrm{\mathrm{Rep}}(Q,v)=\mathrm{\mathrm{Rep}}(\overline{Q},v)=\bigoplus\limits_{\substack{e:i\rightarrow j \in Q_1}}\Bigl[\Hom(V_i,V_j)\oplus \Hom(V_j,V_i)\Bigr].\end{equation}
We denote by $(x_e,y_{e^*})_{e\in Q_1}$ an element of $\mathrm{\mathrm{Rep}}(\overline{Q},v)$. 
 Hence, through the previous identification, the lift of the $G_v$-action $(\ref{eq2.9})$ on $\T^*\Rep(Q,v)$ is again given by changing bases. We consider $\T^*\Rep(Q,v)$ with the standard symplectic structure: 
\begin{equation}\label{eq2.11}
    \omega=tr(dx\wedge dy),
\end{equation}
explicitly given by: 
\begin{equation*}\omega\bigl((x^1,y^1),(x^2,y^2)\bigr)\coloneqq    \sum\limits_{e\in Q_1} tr_{V_{t(e)}}\bigl(y^2_{e^*}x^1_e-y^1_{e^*}x^2_e\bigr).\end{equation*}
 Then, the $G_v$ action is Hamiltonian with standard moment map:
   \begin{equation}\label{eq2.12}
       \begin{gathered}
       \mu = \mu_v : \mathrm{\mathrm{Rep}}(\overline{Q},v)=\bigoplus\limits_{e:i\rightarrow j \in Q_1}\bigl[\Hom(V_i,V_j)\oplus \Hom(V_j,V_i)\bigr] \rightarrow \mathfrak{g}_v \\
       (x,y) \mapsto [x,y]\coloneqq    \Bigl( \sum\limits_{\substack{e \in Q_1 \\ h(e)=i}} x_e y_{e^*}-\sum\limits_{\substack{e \in Q_1\\ t(e)=i}}y_{e^*}x_e \Bigr)_{i \in Q_0}
    \end{gathered}
    \end{equation}
   where we have used the trace pairing to identify  the Lie algebra of the gauge group, $\mathfrak{g}_v\coloneqq   \Lie(G_v)= \bigoplus\limits_{i \in Q_0}\mathfrak{gl}_{v_i}$, with its dual. 

There is also another important quiver associated to $Q$, called the \textit{framed quiver associated to} $Q$, that will be denoted by $Q^F$. The set of vertices of $Q^F$ is defined as $Q_0 \sqcup Q_0'$ where $Q_0'$ is another copy of the set $Q_0$ equipped with a bijection $Q_0' \rightarrow Q_0$, $i' \mapsto i$. The set of arrows is defined as a disjoint union of $Q_1$ and a set of arrows $\{I_i:i' \mapsto i\}_{i \in Q_0}$.
 \begin{Example}\label{ex2.2}\text{}
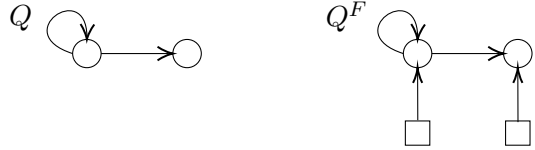
 \begin{center}
        \input{PictureQF}
    \end{center}

\end{Example}
A dimension vector for $Q^F$ is a pair of dimension vectors for $Q$, i.e. $(v,w) \in \NN^{I}\times\NN^{I}$, which we will call \textit{framed dimension vector} for $Q$. The space of representations of $Q^F$ of dimension vector $(v,w)$ is $\mathrm{\mathrm{Rep}}(Q,v,w)=\mathrm{\mathrm{Rep}}(Q,v) \oplus \bigoplus\limits_{i \in Q_0}\Hom(\CC^{w_i},\CC^{v_{i}})$. We will also use the notation $W_i=\CC^{w_i}$ for any $i \in Q_0$. 
\begin{Example}For the quiver of Example \ref{ex2.2}, we have:
\begin{center}
    $\mathrm{\mathrm{Rep}}(Q,v,w)=$\begin{tikzcd}
\mathbb{C}^{v_1} \arrow[r, "x_{21}"]  \arrow["x_{11}"', loop, distance=2em, in=125, out=55] & \mathbb{C}^{v_2}  \\
\mathbb{C}^{w_1}  \arrow[u, "I_1"']                                                                                           & \mathbb{C}^{w_2} \arrow[u, "I_2"]                 
\end{tikzcd}
\end{center}
\end{Example}
$\Rep(Q,v,w)$ is the space of $(v,w)$-dimensional $Q^F$-representations, then $G_{v,w}=G_v\times G_w$ acts by changing bases. However, we consider $Q^F$ as the framing of $Q$ and so we will consider $\Rep(Q,v,w)$ as a $G_v$-module. The lift of the $G_v$-action on $\T^*\mathrm{\mathrm{Rep}}(Q,v,w)=\mathrm{\mathrm{Rep}}(\overline{Q^F},(v,w))$ is again Hamiltonian.
Explicitly, let $(x,y,J,I) \in \mathrm{\mathrm{Rep}}(\overline{Q^F},v,w)$ where $x\in \mathrm{\mathrm{Rep}}(Q,v)$, $y\in \mathrm{\mathrm{Rep}}(Q^{op},v)$, $I\in \bigoplus\limits_{i \in I}\Hom(W_i,V_i)$ and $J\in \bigoplus\limits_{i\in I}\Hom(V_i,W_i)$. Then, the standard symplectic form is given by:
\begin{equation*}
    \omega = tr(dx \wedge dy + dI \wedge dJ).
\end{equation*}
The $G_v$-action is:
\begin{equation}\label{eq.2.16}
    g\cdot(x,y,I,J)=(gxg^{-1},gyg^{-1},gI,Jg^{-1}).
\end{equation}
where $g \in G_v$ and $(x,y,J,I) \in \mathrm{Rep}(\overline{Q^F},v,w)$.
The standard moment map is:
\begin{align*}
    \mu=\mu_{v,w}:\mathrm{Rep}(\overline{Q^F},v,w) &\longrightarrow \mathfrak{g}_v\\
    (x,y,I,J) &\mapsto [x,y]+IJ
\end{align*}
 \begin{Remark}
    Let $Q$ be a quiver and $v \in \NN^{Q_0}$ be a dimension vector for $Q$.
    \begin{enumerate}[label=(\roman*)]
        \item Rational characters of $G_v$ have the form:
    \begin{equation}\label{eq2.17}
    \chi_{\theta}:G_v \rightarrow \CC^{\times}, (g_i)_{i \in Q_0} \mapsto\prod\limits_{i \in Q_0} det(g_i)^{-\theta_i}
    \end{equation}
    for some $\theta=(\theta_i) \in \mathbb{Z}^{Q_0}$. So, we will regard stability parameters for $G_v$-actions as elements in $\ZZ^{Q_0}$. 
    \item Deformation parameters for a Hamiltonian $G_v$-action are given by elements in $(\mathfrak{g}_v)^{G_v}$, i.e. elements in $\mathfrak{g}_v$ fixed by $G_v$. Then the set of deformation parameters can be identified with $\CC^{Q_0}$ after diagonally embedding the latter in $\mathfrak{g}_v$.
    \end{enumerate}
 \end{Remark} 
\begin{Definition}
    Let $(\lambda,\theta) \in \CC^{Q_0} \times \ZZ^{Q_0}$ be fixed parameters. We define the \textit{Nakajima quiver variety} associated to a quiver $Q$ and framed dimension vector $(v,w) \in \NN^{Q_0}\times\NN^{Q_0}$ with parameters $(\lambda,\theta)$:
    \begin{equation*}
        \mathcal{M}_{\lambda,\theta}(Q,v,w)\coloneqq    \mathcal{M}_{\lambda,\theta}(v,w)\coloneqq    \mathrm{Rep}(\overline{Q^F},v,w) \HR_{\lambda,\theta}G_v\coloneqq    \mu^{-1}(\lambda)\GIT_{\theta}G_{v}
    \end{equation*}
    where $\GIT_\theta\coloneqq    \GIT_{\chi_\theta}$.
\end{Definition}
 Let us recall a characterization of semistability condition for the $G_v$-action on $\mathrm{Rep}(\overline{Q^{F}},v,w)$, which is just a special case of Mamford's numerical criterion (Theorem \ref{ncM}).
\begin{Theorem}[Theorem 10.32 \cite{Ki16}]\label{ncN}
        Let $Q$ be a quiver, $(v,w)\in\NN^{Q_0}\times\NN^{Q_0}$ and $\theta \in \ZZ^{Q_0}$. Then an element $(z,I,J) \in \mathrm{Rep}(\overline{Q^F},v,w)$ is $\theta$-semistable if and only if given a $\overline{Q}$-subrepresentation $V' \subset V$ we have:
        \begin{itemize}
            \item $V' \subset \Ker J \Rightarrow \theta \cdot \sdim V' \leq 0$.
            \item $V' \supset \Ima I \Rightarrow \theta \cdot \codim V' \geq 0$.
        \end{itemize}
    \end{Theorem}

\subsection{Example: \texorpdfstring{$\T^{*}\mathbb{P}^1$}{}}\label{E.1}
In this subsection we explicitly describe one of the simplest Nakajima's quiver varieties. Later, we will recover the same variety as a bow variety and this will work as a guide to understand the relation between Nakajima's quiver varieties and Cherkis bow varieties. Let $Q$ be the $A_1$-quiver, i.e. a single vertex and no arrows. Let $(v,w)=(1,2)$, $\lambda=0$ and $\theta \in \NN^+$. So, $G_v=\CC^\times$ and
\begin{gather*}
    \mathrm{Rep}(\overline{Q^F},v,w)=\Bigl\{ \CC \mathrel{\mathop{\rightleftarrows}^{\mathrm{J}}_{\mathrm{I}}} \CC^2 \Bigr\} \\
    \mathrm{Rep}(\overline{Q^F},v,w)^{\theta-ss}=\Bigl\{\CC \mathrel{\mathop{\rightleftarrows}^{\mathrm{J}}_{\mathrm{I}}} \CC^2, \ \vert \ J\text{ is injective }\Bigr\}
\end{gather*}
where the second identification follows from Theorem \ref{ncN}. The associated moment map is
\begin{equation*}
    \mu : \Bigl\{ \CC \mathrel{\mathop{\rightleftarrows}^{\mathrm{J}}_{\mathrm{I}}} \CC^2 \Bigr\} \rightarrow \CC, \ \ \bigl(I,J  \bigr) \mapsto IJ,
\end{equation*}
hence we have
\begin{equation*}
    \Bigl(\mu^{-1}(0)\Bigr)^{\theta-ss}=\Bigl\{\CC \mathrel{\mathop{\rightleftarrows}^{\mathrm{J}}_{\mathrm{I}}} \CC^2, \ \vert \ J\text{ is injective and }IJ=0\Bigr\}.
\end{equation*}
Recalling that \begin{equation*}
    \T^*\mathbb{P}^{1}\cong\Bigl\{\bigl(L,A) \ \vert \ L\text{ is a line in }\CC^2, \ A \in M_2(\CC), \ AL=0, \ \Ima A \subseteq L\},
\end{equation*}
we have
\begin{equation*}
    \mathcal{M}_{0,\theta}\bigl(1,2\bigr)\cong \T^*\mathbb{P}^1
\end{equation*}
for any $\theta >0$.

\section{Quiver description of bow varieties}\label{S3}
In this section we will construct bow varieties as Poisson algebraic varieties. Starting from a generalization of quivers called \textit{bows}, we can associate a symplectic affine variety with a natural Hamiltonian group action. Then, we define bow varieties as Hamiltonian reductions.

\subsection{Bows}\label{bows}
Let us recall the definition of bow (see \cite{cherkis2011instantons}) and its graphical representation.

A \textit{bow} $B$ is defined by the following data.
\begin{itemize}
    \item A collection $\II$ of oriented closed distinct intervals, called \textit{wavy lines}.
    \item A collection $\EE$ of \textit{arrows}, such that each edge $e \in \EE$ begins at the end of some interval $t(e)$ and ends at the beginning of some (possibly the same) interval $h(e)$. 
\end{itemize}

   \begin{Remark}
        Let us denote by $l_\sigma$ the length of an interval $\sigma$. It is clear from the definition that a bow degenerates into a quiver as all the interval lengths $l_\sigma$ go to $0$. Then, we can talk of the \textit{underlying quiver of a bow} which determines the form of the bow. For example, we will say that a bow is a \textit{Dynkin bow of affine type $A$} if the underlying quiver is a Dynkin quiver of affine type $A$, as in Figure \ref{fig:2}.
         \end{Remark} 
    
     In order to graphically represent a bow, we represent an interval by a \textit{wavy line} \begin{tikzpicture}
    \draw[decorate, decoration = {snake, segment length = .2cm}] (0,0) -- (1,0);
\end{tikzpicture}  and an oriented edge by an arrow $\rightarrow$.  
 \begin{Example}\label{E.2} Some examples of bows.
\begin{figure}[H]
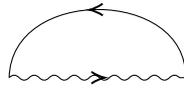

\begin{center}
\include{Jordan}
\end{center}
\caption{Jordan bow} \label{fig:1}
\end{figure}
\begin{figure}[H]
    \centering
    \include{Affine-A-Bows}
    \caption{Bow of affine type $A$}
    \label{fig:2}
\end{figure}
\begin{figure}[H]
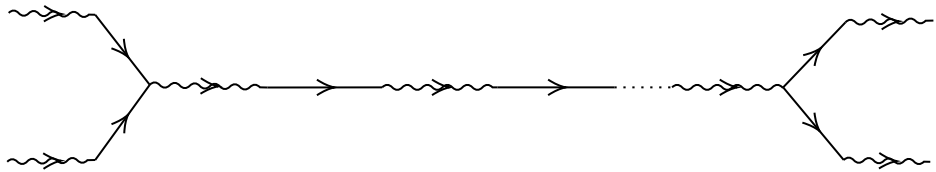

    \centering
\include{Affine-D-Bows}
\caption{Bow of affine type $D$}    \label{fig:3}
\end{figure}
 \end{Example}
\begin{Remark}
In order to provide an algebro-geometric description of bow varieties, we will not need interval lengths. Therefore we may think of a bow as a quiver with ``oriented" vertices.
\end{Remark}
\begin{Remark}
In this paper, affine type A bows (see Figure \ref{fig:2}) have always at least one arrow. This contrasts slightly with \cite{NT17}, where bows of this type with only one wavy line and no arrow are allowed, so that the wavy line forms a circle. The definitions and results below can be easily adapted to include this case.
\end{Remark}

\subsection{Bow diagrams}
Let $B=(\II,\EE)$ be a bow. A \textit{bow diagram} over $B$ consists of the following data. 
\begin{itemize}
    \item A collection $\Lambda$ of distinct points on the wavy lines, called \textit{x-points} and graphically represented by \input{x-point}. Note that x-points partition wavy lines into \textit{segments} $\zeta$, for instance: \input{segment}. Extreme points $t(\zeta)$ and $h(\zeta)$ of a segment $\zeta$ are either end-points of wavy lines or x-points. We denote by $\IS$ the set of segments and for any $\sigma \in \II$ we denote by $\zeta_\sigma$ the first segment in $\sigma$.
    \item A dimension vector $v=\bigl(v_\zeta\bigr)_\zeta \in \NN^{\IS}$.
\end{itemize}

We will think of a bow diagram as a triple $(B,\Lambda,v)$.
\begin{figure}[H]
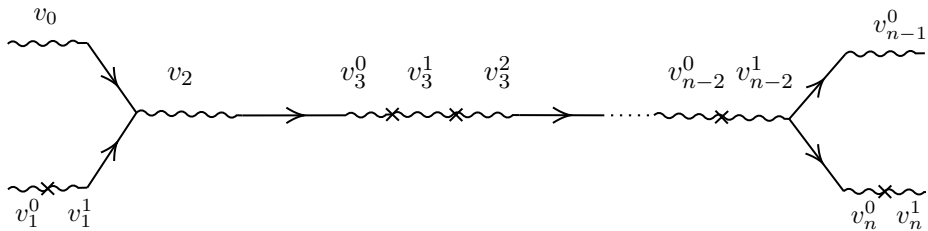

    \centering
    \include{diagram-typeD}
    \caption{Bow diagram over an affine type D bow.}
    \label{fig:4}
\end{figure}
\begin{Remark}
Bow diagrams can be thought of as a `bow version' of framed dimension vectors for quivers. This will become clear later in Section \ref{S5}.
\end{Remark}
For simplicity, we will use the following notation. Given a segment $\zeta$ we define $V_\zeta\coloneqq    \CC^{v_\zeta}$. Given an arrow $e \in \EE$ we denote by $t(e)$ and $h(e)$ the segments, respectively, at the origin and at the end of the arrow $e$.

\subsection{Two-way parts}\label{bif}
Let us fix a bow diagram $(B,\Lambda,v)$. Given an arrow $e \in \EE$, we define
\begin{equation*}
    \M^e\coloneqq    \Hom(V_{t(e)},V_{h(e)})\oplus \Hom(V_{h(e)},V_{t(e)}),
\end{equation*}
where $t(e)$ and $h(e)$ denote the segments containing, respectively, the tail and the head of $e$. We denote elements of $\M^e$ by $(C_e,D_e)$. As in $(\ref{eq2.10})$, we can think $\M^e$ as the cotangent space of the space of representation of a quiver of type $A_2$ with dimension vector $(v_{t(e)},v_{h(e)})$. 
In particular, it can be equipped with the standard symplectic structure, given by the formula $(\ref{eq2.11})$. If we consider the action of $\GL(v_{t(e)})\times \GL(v_{h(e)})$ on $\M^e$ by changing bases, then it is Hamiltonian with moment map given by the formula $(\ref{eq2.12})$:
\begin{equation}
    \mu_e:\M^e \rightarrow \mathfrak{gl}(v_{t(e)})\oplus\mathfrak{gl}(v_{h(e)}), \ \ \ \ \ (C,D) \mapsto [C,D]\coloneqq    (-DC, CD),
\end{equation}
where we identify a Lie algebra $\mathfrak{g}$ with its dual via the trace pairing.

\subsection{Triangles}\label{triangles}
   In this section we will review the construction of the space of \textit{triangles}, which has been introduced by Takayama in \cite{Tak16}. Let us fix a dimension vector $\underline{v} \in \NN^2$. A \textit{triangle} of dimension $\underline{v}$ is a collection of maps $(A,B_1,B_2,a,b)$ defined as follows:
\begin{center}
    \input{GenTri2}
\end{center}
satisfying the following algebraic conditions:
    \begin{equation*}
        B_2A-AB_1+ab= 0. \tag{$a$}
    \end{equation*}
        \begin{gather*}
        \text{There is no subspace }0\neq S \subset V_{1},\text{ }B_1\text{-invariant, s.t. }A(S)=0=b(S).\tag{S1}\\
        \text{There is no subspace }T \subsetneq V_{2},\text{ }B_2\text{-invariant, s.t. }\Ima A+\Ima a \subset T.\tag{S2}
    \end{gather*}

We will denote the set of triangles of dimension $\underline{v}$ by $\M^{\triangle}_{\underline{v}}$ or by $\M^\triangle$ if the dimension vector is not relevant. On $\M^\triangle$ there is a natural action of the group $\GL(\underline{v})=\GL(v_1)\times \GL(v_2)$ by changing bases, namely:
\begin{equation}
    (g_1,g_2)\cdot(A,B_1,B_2,a,b) = (g_2Ag_1^{-1},g_1B_1g_1^{-1},g_2B_2g_2^{-1},bg_1^{-1},g_2a),
\end{equation}
for every $(g_1,g_2)\in \GL(\underline{v})$ and $(A,B_1,B_2,a,b) \in \M^\triangle$. Let us notice that the space of triangles satisfying only condition $(a)$ is an affine variety. Furthermore, as shown in \cite[Proposition 2.2]{NT17}, if we consider the space of triangles satisfying only the condition $(a)$, then we can see conditions $(S1)$ and $(S2)$ as a condition of semistability for the action of $\GL(\underline{v})$. It follows from Proposition \ref{prop2.3} that $\M^\triangle$ is a locally closed subset of the vector space $\Hom(\CC^{v_1},\CC^{v_2})\oplus\End(\CC^{v_1})\oplus\End(\CC^{v_2})\oplus(\CC^{v_1})^*\oplus\CC^{v_2}$ (that is the vector space of triangles that do not satisfy any condition).
Let us now state a useful property of the map $A$ in a triangle.

\begin{Lemma}[\cite{Tak16} Lemma 2.18]\label{3.9}
    If $(A,B_1,B_2,a,b) \in \mathcal{M}^\triangle$, then $A$ has full rank.
\end{Lemma}
Next, following \cite{Tak16}, we endow the space of triangles $\M^\triangle$ with a structure of nonsingular affine variety carrying a holomorphic symplectic structure such that the $\GL(\underline{v})$-action is Hamiltonian.
\begin{Definition}
    Let $\underline{v}=(v_1,v_2) \in \NN^2$, $n=\max(v_1,v_2)$ and $m=\min(v_1,v_2)$. We define $F(\underline{v})=F(v_1,v_2)$, as:
    \begin{itemize}
        \item If $v_1 \neq v_2$, 
      \begin{equation*}\left\{(u,\eta)\in\GL_n \times \mathfrak{gl}_n : \eta=\input{HNF}\right\}\end{equation*}
       where  $(e_0,e,f,g,h) \in \CC \times \CC^{n-m-1} \times (\CC^m)^* \times \CC^m \times \End(\CC^m)$.

    \item If $v_1=v_2$,
    \begin{equation*}
              \GL_n \times \mathfrak{gl}_n \times \CC^n \times (\CC^n)^* \ni  \bigl(u,h,I,J\bigr).
    \end{equation*}
    \end{itemize}
\end{Definition}
It is clear that $F(\underline{v})$ is a nonsingular affine variety for any $\underline{v}\in\NN$. We will call \textit{Hurtubise normal forms} the elements of $F(\underline{v})$. Moreover, we have an algebraic $\GL(\underline{v})$-action on $F(\underline{v})$ given by:
\begin{itemize}
    \item If $v_1 \neq v_2$,
     \begin{equation}
    \bigl(g_n,g_m\bigr)\cdot\bigl(u,\eta) = \Bigl(\begin{bmatrix}
            g_m &  0 \\
            0 &  \operatorname{id} 
        \end{bmatrix}ug_n^{-1}, \begin{bmatrix}
            g_m &  0 \\
            0 &  \operatorname{id} 
        \end{bmatrix}\eta\begin{bmatrix}
            g_m^{-1} &  0 \\
            0 &  \operatorname{id} 
        \end{bmatrix}\Bigr),
        \end{equation}
    where $n=\max(v_1,v_2)$, $m=\min(v_1,v_2)$, $(g_n,g_m) \in \GL(n)\times \GL(m)$ and $(u,\eta) \in F(\underline{v})$.

\item If $v_1=v_2$,
\begin{equation}\bigl(g_1,g_2) \cdot \bigl(u,h,I,J\bigr) = \bigl(g_2ug_1^{-1},g_2hg_2^{-1},g_2I,Jg_2^{-1}\bigr),\end{equation}
where $(g_1,g_2) \in \GL(\underline{v})$ and $(u,h,I,J) \in F(v_1,v_2)$.
\end{itemize}

\begin{Proposition}[\cite{Tak16}]\label{Prop3.13}
Let $\underline{v}=(v_1, v_2) \in \NN^2$. Then we have an isomorphism of $\GL(\underline{v})$-variety $F(\underline{v})\cong\M^\triangle_{\underline{v}}$ given by:
\begin{gather*}
    (A,B_1,B_2,a,b) = \begin{cases}
        \Bigl(
        [\operatorname{id}_{v_2},0,0]u,u^{-1}\eta u,h,g,[0,0,1]u\Bigr) & \text{ if }v_1 > v_2, \\
        \Bigl(u,u^{-1}hu,h-IJ,I,Ju\Bigr) & \text{ if } v_1=v_2,\\
        \Bigl(-u^{-1}\begin{bmatrix}
            \operatorname{id}_{v_1} \\
            0 \\
            0
        \end{bmatrix},-h,-u^{-1}\eta u, u^{-1}\begin{bmatrix}
            0 \\
            1 \\
            0 
        \end{bmatrix},-f\Bigr) & \text{ if } v_1<v_2.
    \end{cases}
\end{gather*}
\end{Proposition}
It follows that $\M^\triangle$ is a nonsingular affine complex $\GL(\underline{v})$-variety of dimension $v_1^2+v_1+v_2^2+v_2$. 
\begin{Proposition}\label{Prop3.14} The space of Hurtubise normal forms $F(\underline{v})$ has a holomorphic symplectic structure given by:
\begin{equation}\label{symp.NT}
    \omega = \begin{cases}
        tr\bigl(d\eta \wedge  duu^{-1}+ \eta duu^{-1}\wedge duu^{-1}\bigr) &\text{ if }v_1\neq v_2, \\
        tr\bigl(dh \wedge  duu^{-1} + hduu^{-1}\wedge duu^{-1} + dI\wedge dJ\bigr)& \text{ if }v_1=v_2.
    \end{cases}
\end{equation}
where $duu^{-1}$ denotes the right-invariant Maurer-Cartan form.
\end{Proposition}

\begin{Remark}We notice that if $v=v_1=v_2$, then $ tr\bigl(dh \wedge  duu^{-1} + hduu^{-1}\wedge duu^{-1})$ is the canonical symplectic structure on $\T^*\GL_v$ obtained as differential of the Liouville form. More generally, we can identify spaces of Hurtubise normal forms above with some moduli spaces of solutions to Nahm's equations over an interval. In particular, it has a hyperk\"ahler structure. For more details, see  Appendix \ref{Appendix A} and \cite[Remark 3.4]{NT17}.\end{Remark}
We will consider the space of triangles $\M^\triangle$ equipped with the holomorphic symplectic structure induced by the previous isomorphism with the space of Hurtubise normal forms.

Let us consider the case $n\coloneqq v_1=v_2$. The action of $\GL_n\times \GL_n$ on $\GL_n\times\mathfrak{gl}_n$ is nothing but the lift of the action of $\GL_n\times\GL_n$ on $\GL_n$ to $\T^*\GL_n$ after identifying $\T^*\GL_n$ and $\GL_n\times\mathfrak{gl}_n$ via right-trivialization and the trace pairing. More generally, the $\GL(\underline{v})$-action on $F(\underline{v})$ is Hamiltonian for any $\underline{v} \in \NN^2$. In terms of triangles, the moment map is given by:
\begin{equation*}\begin{gathered}
    \mu_\triangle : \M^\triangle \rightarrow \mathfrak{gl}_{v_1}\times\mathfrak{gl}_{v_2}, \\
    \Bigl(A,B^-,B^+,a,b\Bigr) \mapsto \Bigl(B^-,-B^{+}\Bigr).
\end{gathered}\end{equation*}

\begin{Remark}\label{Remark.Poissonstructure}
    For a later purpose, we recall the Poisson structures on $\M^\triangle$, associated with the symplectic form given in (\ref{symp.NT}), expressed in terms of $(A,B_1,B_2,a,b)$. Further details can be found in \cite[Section 5.1]{NT17}. 

Let $A_{ij}$ be the linear function on $\M^\triangle$ given by the $(ij)$ entry of $A$. Similarly, we define $(B_1)_{ij},(B_2)_{ij},a_i,b_j$ for every $i\in\{1,\dots,v_2\}$ and $j\in\{1,\dots,v_1\}$. Then, the Poisson bracket is given by:
\begin{equation*}
\begin{aligned}
   \bigl\{&A_{ij},A_{kl}\bigr\}=0, &\bigl\{&a_i,a_j\bigr\}=\bigl\{b_i,b_j\bigr\}=0, \\
   \bigl\{&(B_1)_{ij},(B_1)_{kl}\bigr\}=\delta_{kj}(B_1)_{il}-\delta_{il}(B_1)_{kj},  & \bigl\{&(B_1)_{ij},b_k\bigr\}=-\delta_{ik}b_j,\\
   \bigl\{&(B_2)_{ij},(B_2)_{kl}\bigr\}=\delta_{il}(B_2)_{kj}-\delta_{kj}(B_2)_{il}, & \bigl\{&(B_2)_{ij},a_k\bigr\}=-\delta_{kj}a_i,\\
   \bigl\{&(B_1)_{ij},a_k\bigr\}=\bigl\{(B_2)_{ij},b_k\bigr\}=0, & \bigl\{&(B_1)_{ij},(B_2)_{kl}\bigr\}=0,\\
   \bigl\{&(B_1)_{ij},A_{kl}\bigr\}=-\delta_{il}A_{kj}, &\bigl\{&(B_2)_{ij},A_{kl}\bigr\}=-\delta_{kj}A_{il}\\
   \bigl\{&b_i,a_j\bigr\}=A_{ji}, &\bigl\{&A_{ij},b_k\bigr\}=\bigl\{A_{ij},a_k\bigr\}=0.
\end{aligned}
\end{equation*}
\end{Remark}

 Let $(B,\Lambda,v)$ be a bow diagram and $x \in \Lambda$ a fixed x-point. We denote by $\zeta_x^-$ and $\zeta_x^+$ the adjacent segments to $x$, so that
 \begin{center}
     \input{AdjSeg}
 \end{center}
and $\underline{v}_x=(v_x^-,v_x^+)=(v_{\zeta_{x}^-},v_{\zeta_{x}^+})$. To make notation simpler, we may denote them by $v_{-}$ and $v_{+}$. Hence, we define the space of \textit{triangles at $x$} as: $\M^x\coloneqq    \M^\triangle_{\underline{v}_x}$ and denote its elements by $(A_x,B_x^-,B_x^+,a_x,b_x)$. We will omit $x$ from the notation when there is no ambiguity. We will denote the associated moment map by $\mu_x$.

\subsection{Bow varieties}
Let $B=(\II,\EE)$ be a bow and $(\Lambda,v)$ be a bow diagram over $B$. We define:
\begin{equation}\label{eq3.19}
    \TM\coloneqq\TM(B,\Lambda,v)\coloneqq \prod\limits_{x \in \Lambda}\M^x \times \bigoplus\limits_{e \in \EE}\M^e.
\end{equation}
$\TM(B,\Lambda,v)$ is a product of nonsingular affine symplectic varieties (see \S\ref{bif}, \S\ref{triangles}). Hence, $\TM$ has a structure of nonsingular affine variety with an induced holomorphic symplectic form given by the sum of the symplectic forms on each factors.
Using the notation of previous subsections, we denote its elements by $(A_x,B^\pm_x,a_x,b_x,C_e,D_e)_{x\in\Lambda, \ e\in\EE}=(A_x,B_x^-,B_x^+,a_x,b_x,C_e,D_e)_{x\in\Lambda, \ e \in \EE}$. When there is no ambiguity, we omit subscripts $x$ and $e$.

To any bow diagram we associate the group $\GG=\prod\limits_{\zeta\in\IS}\GL(v_\zeta)$ which acts on $\TM$ by changing bases. Namely, given $(A,B^\pm,a,b,C,D) \in \TM$ and $g=(g_\zeta) \in \GG$, then $g\cdot(A,B^\pm,a,b,C,D)$ is given by
\begin{equation}
  \bigl(g_{\zeta_x^+}A_xg^{-1}_{\zeta_x^-}, g_{\zeta_x^\pm}B_x^\pm g_{\zeta_x^\pm}^{-1},g_{\zeta_x^+}a_x,b_xg_{\zeta_x^-}^{-1},g_{h(e)}C_e g_{t(e)}^{-1},g_{t(e)}D_e g_{h(e)}^{-1}\bigr).
\end{equation}
It follows from the above discussion on two-way parts and triangles  (cf. \S\ref{bif}, \S\ref{triangles}), that the action of $\GG$ on $\TM$ is Hamiltonian with moment map:
\begin{equation*}\begin{gathered}
    \mu:\TM \longrightarrow  \mathfrak{g} := \Lie(\GG) = \bigoplus\limits_{\zeta \in \IS} \End(V_\zeta) \\
    (A_x,B_x^-,B_x^+,a_x,b_x,C_e,D_e) \mapsto \sum\limits_{x \in \Lambda}\mu_x(A_x,B_x^-,B_x^+,a_x,b_x)+\sum\limits_{e\in\EE}\mu_e(C_e,D_e).
\end{gathered}\end{equation*}
Explicitly, $\mu(A,B,a,b,C,D)_\zeta$ is given by:
\begin{equation}\label{eq.3.22.}
\begin{aligned}
    \input{mu2}
\end{aligned}
\end{equation} 
\begin{Remark}\label{rem3.22}
    In order to perform symplectic reductions, we need to choose a stability parameter and a deformation parameter.
 \begin{enumerate}
        \item[(i)] Given $\lambda \in \CC^\II$, we define $\overline{\lambda}=\bigl(\overline{\lambda}_\zeta\bigr) \in \CC^{\IS}$ by $\overline{\lambda}_\zeta= \lambda_\sigma$ if $\zeta=\zeta_\sigma$,  $\overline{\lambda}_\zeta=0$ otherwise. That is, we have defined an injection of $\CC^\II$ in $\CC^{\IS}$ via first segments of intervals. Hence, we can talk about \textit{deformation parameters} $\lambda$ in $\CC^\II$ and use the notation $\mu^{-1}(\lambda)$. Let us notice that $\mu^{-1}(\lambda)$ is a $\GG$-invariant Zariski closed subset of $\TM$.
        \item[(ii)] Given $\theta \in \ZZ^\II$ we define a rational character of $\GG$ by:
        \begin{equation*}
            \chi_\theta : \GG \rightarrow \CC^\times, \ \ \ \bigl(g_\zeta\bigr) \mapsto \prod\limits_{\sigma \in \II}det(g_{\zeta_{\sigma}})^{-\theta_\sigma}
        \end{equation*}
        Therefore, we can talk of $\theta$-(semi)stable points in $\TM$ and call \textit{stability parameter} an element of $\ZZ^\II$. As usual, for any character $\chi$, we have that $\chi$ and $\chi^N$ give the same (semi)stability conditions for any $N \in \ZZ_{>0}$, therefore $\theta$ and $N\theta$ also give the same (semi)stability conditions. 
        Hence the notion of $\theta$-semistability is well-defined for any $\theta \in \QQ^\II$.
    \end{enumerate}
The previous choice of parameters is not restrictive, see subsections  \ref{s.ncb} and \ref{s.gdp}.
    \end{Remark}
\begin{Definition}
    Given a bow diagram $(B,\Lambda,v)$ and parameters $(\lambda,\theta) \in \CC^\II\times\ZZ^\II$, we define the associated bow variety as: 
    \begin{equation}
        \mathcal{M}_{\lambda,\theta}\bigl(B,\Lambda,v):=\mu^{-1}(\lambda)\GIT_\theta\GG.
    \end{equation}
\end{Definition}
\begin{Remark}\label{rem3.26}
    It is clear from the definition that we can think $\TM$ as a locally closed subset of a larger vector space in the following way. Let $(B,\Lambda,v)$ be a bow diagram. For any $x \in \Lambda$ we define the space of \textit{free triangles} at $x$ by:
\begin{equation*}
\MM^x\coloneqq\Hom(V_{\zeta_x^-},V_{\zeta_x^+})\oplus \End(V_{\zeta_x^-})\oplus \End(V_{\zeta_x^+})\oplus \Hom(\mathbb{C},V_{\zeta_x^+}) \oplus
\Hom(V_{\zeta_x^-},\mathbb{C}).
\end{equation*} 
Then, we define
\begin{equation*}
    \MM\coloneqq\bigoplus\limits_{x \in \Lambda}\MM^x \oplus\bigoplus\limits_{e \in \EE}\M^e.
\end{equation*}
and
\begin{equation*}\begin{gathered}
    \mu_1:\MM \mapsto \bigoplus\limits_{x\in\Lambda}\Hom(V_{\zeta_x^-},V_{\zeta_x^+}), \\
    (A_x,B_x^-,B_x^+,a_x,b_x,C_e,D_e) \mapsto (B_x^+A_x-AB_x^-+a_xb_x)_{x\in\Lambda}.
    \end{gathered}
\end{equation*}
It follows that $\TM$ is given by points in $\mu_1^{-1}(0)$ such that each triangle satisfies conditions $(S1)$ and $(S2)$. Moreover, we can define $\mu_2:\MM \rightarrow \mathfrak{g}$ using the equations (\ref{eq.3.22.}). Then, if 
\begin{equation*}
    \overline{\mu}=\mu_1\oplus\mu_2:\MM \mapsto \bigoplus\limits_{x\in\Lambda}\Hom(V_{\zeta_x^-},V_{\zeta_x^+}) \oplus \mathfrak{g},
\end{equation*}
the subvariety $\mu^{-1}(\lambda) \subset \TM$ is given by the elements in $\overline{\mu}^{-1}(0,\lambda)=\mu_1^{-1}(0)\cap\mu_2^{-1}(\lambda)$ satisfying conditions $(S1)$ and $(S2)$. Finally, the action of $\GG$ on $\TM$ can be extended to the action of $\GG$ on $\MM$ by changing bases.
\end{Remark}
\section{Numerical conditions}\label{S4}

\subsection{Numerical criterion for semistability}\label{s.ncb}
In this subsection we will characterize semistability conditions for the action of $\GG$ on $\TM$. It turns out that the subsets of $\GG$-characters we considered to define bow varieties (see Remark \ref{rem3.22}) is not restrictive for our purposes. Let us fix a bow diagram $(B,\Lambda,v)$ for the entire subsection. We start recalling that any rational character of $\GG=\prod\limits_{\zeta \in \IS} \GL(v_\zeta)$ has the form
\begin{equation*}
    \chi_\nu : \GG \mapsto \CC^\times, \ \ \bigl(g_\zeta  \bigr) \mapsto \prod\limits_{\zeta \in \IS} det(g_\zeta)^{-\nu_\zeta}
\end{equation*}
for some $\nu \in \ZZ^{\IS}$. For any $\nu \in \ZZ^{\IS}$, we will refer to $\chi_\nu$-semistability for the $\GG$-action on $\mu^{-1}(\lambda)$ as the $\nu$-semistability. We notice our choice of stability parameters in $\ZZ^\II$ can be regarded as a special case if we consider $\ZZ^\II$ embedded in $\ZZ^{\IS}$ via first segments of intervals.

Let $\nu \in \ZZ^{\IS}$ be a fixed stability parameter and let us define $\theta(\nu)\in \ZZ^\II$ by $\theta(\nu)_\sigma = \sum\limits_{\zeta \subset \sigma} v_\zeta$, for any $\sigma \in \II$.
\begin{Definition}\label{dncB}
    Let $(A,B^-,B^+,a,b,C,D) \in \mu^{-1}(\lambda)$, $\nu\in \ZZ^{\IS}$ and $\theta=\theta(\nu)$. We define the following conditions:
    \begin{enumerate}
        \item[$(\nu_{1}):$] Let $S=\bigoplus\limits_{\zeta\in\IS}S_\zeta \subset \bigoplus\limits_{\zeta\in\IS} V_\zeta$ be a graded subspace $(A,B^\pm,C,D)$-invariant. Suppose $b(S)=0$ and for any $x\in\Lambda$ the restriction of $A_x$ induces an isomorphism $S_{\zeta_{x}^-} \xrightarrow[\cong]{A} S_{\zeta_{x}^+}$. Then
        \begin{equation*}
            \sum\limits_{\zeta\in\IS}\nu_\zeta \sdim S_\zeta = \sum\limits_{\sigma \in \II}\theta_\sigma \sdim S_\sigma \leq 0.
        \end{equation*}
       Here $\sdim S_\sigma$ denotes the dimension of $S_\zeta$ for some segment $\zeta$ in a wavy line $\sigma$. By assumption, it does not depend on the chosen segment.
        \item[$(\nu_{2}):$] Let $T=\bigoplus\limits_{\zeta\in\IS}T_\zeta \subset \bigoplus\limits_{\zeta\in\IS} V_\zeta$ be a graded subspace $(A,B^\pm,C,D)$-invariant. Suppose $\Ima a \subset T$ and for any $x\in\Lambda$ the restriction of $A_x$ on $T$ induces an isomorphism $V_{\zeta^-}/T_{\zeta^-} \xrightarrow[\cong]{A} V_{\zeta^+}/T_{\zeta^+}$. Then
        \begin{equation*}
            \sum\limits_{\zeta\in\IS}\nu_\zeta \codim T_\zeta = \sum\limits_{\sigma \in \II}\theta_\sigma \codim T_\sigma \geq 0.
        \end{equation*}
    Here $\codim T_\sigma$ denotes the codimension of $T_\zeta$ for some segment $\zeta$ in a wavy line $\sigma$. By assumption, it does not depend on the chosen segment. 
    \end{enumerate}
\end{Definition}

The following proposition gives a numerical criterion for $\nu$-semistability and implies that the semistability structure induced by $\nu$ depends only on $\theta(\nu)$.

 \begin{Theorem}\label{ncB}
    Let $\lambda\in\CC^\II$, $\nu \in \ZZ^{\IS}$ and $m=(A,B^\pm,a,b,C,D) \in \mu^{-1}(\lambda)$. Then the following holds.
    \begin{enumerate}
        \item[(i)] $m$ is $\nu$-semistable if and only if $(\nu_{1})$ and $(\nu_{2})$ are satisfied.
        \item[(ii)] $m$ is $\nu$-stable if and only if $(\nu_{1})$ and $(\nu_{2})$ are satisfied with strict inequalities unless $S=0$ and $T=V$.
    \end{enumerate}
\end{Theorem}
\begin{proof}
Let us starting by proving $(i)$. We will use the argument in \cite[Prop. 2.8]{NT17}. We start by proving that if $m \in \mu^{-1}(\lambda)$ is $\nu$-semistable then condition $(\nu_1)$ is satisfied. Let $S=\bigoplus\limits_{\zeta}S_\zeta$ a graded subspace satisfying assumptions in $(\nu_1)$ but not the inequality in $(\nu_1)$, and $S^\perp$ a graded complement of $S$ in $V$. We consider the one parameter subgroup $\gamma(t)$ of $\GG$ defined by $\gamma\vert_S=t\operatorname{id}$ and $\gamma\vert_{S^\perp}=\operatorname{id}$. Then, 
\begin{equation*}
    \langle\chi_\nu,\gamma\rangle=-\sum\limits_{\zeta}\nu_\zeta\dim S_\zeta \lneq 0,
\end{equation*}
where the last inequality follows from the assumption on $S$. Therefore, because of Mumford's numerical criterion (see Theorem \ref{ncM}), we obtain a contradiction by proving the existence of $\overline{m}=\lim\limits_{t\rightarrow0}\gamma(t)\cdot m$ in $\mu^{-1}(\lambda)$. We notice that if such a limit exists in $\MM$ it has to be an element of $\mu_1^{-1}(0) \cap \mu_2^{-1}(\lambda)$ since they are closed subsets. So, we need to prove the existence of the limit $\overline{m}$ in $\MM$ such that its triangles satisfy conditions $(S1)$ and $(S2)$. This has essentially been proved in \cite[Prop. 2.8]{NT17}, so we will omit it. Let us notice that, as noticed in \cite{NT17}, we could prove only the existence of such a limit via Hurtubise normal forms. With a similar argument, we can prove that $\nu$-semistability implies condition $(\nu_2)$. To conclude, the proof that $(\nu_1)$ and $(\nu_2)$ imply $\nu$-semistability, follows by modifying the proof in \cite[Prop. 2.8]{NT17}, which uses King's criterion \ref{ncK}, in a similar fashion to the previous case. Point $(ii)$ follows from the criteria of Mumford and King for stable points (see Theorem \ref{ncM} and Theorem \ref{ncK}).
\end{proof}
\begin{Remark}
For bow varieties of affine type $A$, if $\nu\in\ZZ^{\II}\subset\ZZ^{\IS}$ via the last segments of intervals, then, the numerical criterion above coincides with the numerical criterion proved by Nakajima and Takayama (see \cite[Proposition 2.8]{NT17}). Moreover, it makes our choice to assign a character $\chi_\theta$ for any $\theta \in \ZZ^{\II}$ not restrictive and equivalent to the choice of Nakajima and Takayama of embedding $\ZZ^\II$ via last segments of intervals.
\end{Remark}   

\subsection{Deformation parameters}\label{s.gdp}
In the definition of bow varieties, we used only deformation parameters of a particular form. In this subsection, we will prove that such a choice is not restrictive. Similar considerations hold in the hyperk\"ahler construction, see \cite{cherkis2011instantons} and \cite[Remark 4.5]{T15}. 

Let us fix $(B,\Lambda,v)$ a bow diagram. As in Remark \ref{rem3.22}, we consider $\CC^{\II}$ embedded in $\CC^{\II^s}$ via first segment of intervals and $\CC^{\II^s}$ diagonally embedded into $\mathfrak{g}$. For any $\nu \in \CC^{\II^s}$, we define $\lambda(\nu)\in\CC^{\II}$ by $\lambda(\nu)_\sigma=\sum\limits_{\zeta\in\sigma}\nu_\zeta$.

We need to introduce some notation. Given a wavy line $\sigma\in\II$, we order x-points on it by following its orientation and denote them by $x_{\sigma,i}$ with $i=1,\cdots,w_\sigma$. Similarly, we order segments and denote them by $\zeta_\sigma^j$ with $j=0,\cdots,w_\sigma$. That is: 
\begin{center}
    \input{notation-gdp}
\end{center}
Then, we will use the following notation for triangles associated with an x-point $x_{\sigma,i}$:
\begin{center}
    \input{notation-gdp2}
\end{center}

\begin{Theorem}\label{nCd}
There is an isomorphism of affine $\GG$-varieties between $\mu^{-1}(\nu)$ and $\mu^{-1}(\lambda(\nu))$. Furthermore, it induces an isomorphism of Poisson varieties between $\M_{\nu,\theta}$ and $\M_{\lambda(\nu),\theta}$ for every $\theta\in\ZZ^\II$.
\end{Theorem}
\begin{proof}
Let us start by defining
    \begin{equation}\label{eq4.2}\begin{gathered}
     \phi: \TM \longrightarrow \TM, \\(A,B^\pm,a,b,C,D) \mapsto (A,\overline{B}^\pm,a,b,C,D),
    \end{gathered}\end{equation}
where $\overline{B}_{\sigma,i}^\pm=B_{\sigma,i}^\pm + \sum\limits_{j=i}^{w_\sigma}\nu_{\zeta_\sigma^j}\operatorname{id}_{V_{\sigma,i}^\pm}$ for any $\sigma \in \II$ and $i\in\{1,\cdots,w_\sigma\}$.
We notice it is well-defined. Indeed, since we are just translating components $B$ of triangles by adding scalar matrices then invariant subspaces of $B$ and $\overline{B}$ coincide. Therefore, if $(A,B^\pm,a,b)$ satisfies conditions $(S1)$ and $(S2)$, so $(A,\overline{B}^\pm,a,b)$ does. Condition $(a)$ is also preserved. Indeed, for any $x \in \Lambda$ we have $\overline{B}^\pm=B^\pm + \alpha\operatorname{id}$ for some $\alpha\in\CC$ and so, if $(A,B^\pm,a,b)$ satisfies condition $(a)$, then:
\begin{equation*}
    \overline{B}^+A-A\overline{B}^-+ab=(B^++\alpha)A-A(B^-+\alpha)+ab=0.
\end{equation*}
So, the map (\ref{eq4.2}) is an isomorphism of affine $\GG$-varieties. We need to show that the image of its restriction to $\mu^{-1}(\nu)$ is $\mu^{-1}(\lambda(\nu))$. Let $(A,B^\pm,a,b,C,D) \in \mu^{-1}(\nu)$. Then, given $\zeta$ a segment, $\mu(A,\overline{B}^\pm,a,b,C,D)_{\zeta}$ is equal to
\begin{align*}
  \sum\limits_{e: h(e)=\zeta}C_eD_e+B_{\sigma,1}^- + \sum\limits_{j=1}^{w_\sigma}\nu_{\zeta_\sigma^j}=&\lambda(\nu)_\sigma& \text{ if }&\input{first-gdp} \\ 
  (B_{\sigma,i+1}^-+\sum\limits_{j=i+1}^{w_\sigma}\nu_{\zeta_\sigma^j})-(B_{\sigma,i}^++\sum\limits_{j=i}^{w_\sigma}\nu_{\zeta_\sigma^j})=&0&  \text{ if }&\input{second-gdp} \\
  -\sum\limits_{e:\ t(e)=\zeta}D_eC_e-\bigl(B_{\sigma,w_\sigma}^++\nu_{\zeta_\sigma^{w_\sigma}}\bigr)=&0& \text{ if }&\input{third-gdp} \\
   \sum\limits_{e:\ h(e)=\zeta}C_eD_e-\sum\limits_{e:\ t(e)=\zeta}D_eC_e=\nu_\zeta=&\lambda(\nu)_\sigma&  \text{ if }&\input{fourth-gdp}
\end{align*}

Therefore, we have an isomorphism as algebraic varieties between $\M_{\nu,\theta}$ and $\M_{\lambda(\nu),\theta}$. To conclude the proof, we prove that $\phi$ preserves the Poisson bracket. Note that $\TM=\prod\limits_{x\in\Lambda}\M^x\times\bigoplus\limits_{e\in\EE}\M^e$ and $\phi=\prod\limits_{x\in\Lambda}\phi^x\times\bigoplus\limits_{e\in\EE}\phi^e$, where $\phi^a:\M^a\rightarrow\M^a$ for every $a\in \Lambda\cup\EE$. Then, since the Poisson structure on $\TM$ is the one induced by the Poisson structure on each factor, we have to show that $\phi^a$ is Poisson for every $a$. For every arrow $e\in\EE$, $\phi^e$ is the identity and in particular it is Poisson. Given $x\in\Lambda$, $\phi^x(A_x,B^\pm_x,a_x,b_x)=(A_x,B^\pm_x+\alpha_x\operatorname{id}_{v_x^\pm},a_x,b_x)$ where $\alpha_x\in\CC$ depends on $x$ as defined in (\ref{eq4.2}). It follows from explicit formulas for the Poisson bracket on triangles (see Remark \ref{Remark.Poissonstructure}) that $\phi^x$ is Poisson.
\end{proof}
\begin{Remark}
    Let us consider the map $\phi$ defined in (\ref{eq4.2}). One may use Hurtubise normal forms (cf. \S\ref{triangles}) to prove that for every $x\in\Lambda$, the map $\phi^x$ is a symplectic isomorphism. Let us focus on $x$ cobalanced, i.e. $n\coloneqq v_x^-=v_x^+$. Then, in terms of Hurtubise normal forms, we have:
\begin{equation*}
\begin{aligned}
    \phi^x: \GL_n\times\mathfrak{gl}_n\times \CC^n\times(\CC^n)^* \longrightarrow \GL_n\times\mathfrak{gl}_n&\times\CC^n\times(\CC^n)^*.\\
    (u,h,I,J) \mapsto (u,h+\alpha_x&\operatorname{id},I,J)
\end{aligned}
\end{equation*}
For some $\alpha_x\in\CC$. From the explicit formula of the symplectic structure (\ref{symp.NT}) we can now deduce that it is a symplectic morphism by a direct computation. Alternatively, one could notice that it is a translation along the fibres of $\T^*\GL_n$ induced by a closed 1-form and therefore it is symplectic with respect to the canonical symplectic structure.
\end{Remark}

\subsection{The stable locus}
In this section we want to study the subset of bow varieties induced by the stable locus of level sets of the moment map $\mu$. The strategy is to prove that we can apply Theorem \ref{2.5} and therefore it has to be smooth. In order to prove it, we will generalize the argument used in \cite{NT17} to the case of bows with any underlying quiver. This particular property of bow varieties, together with the numerical criterion proved in the previous section, will allow us to state a sufficient numerical condition for smoothness of bow varieties.

Let us fix a bow $B=(\II,\EE)$, a bow diagram $(\Lambda,v)$ over $B$ and a pair of parameters $(\lambda,\theta)\in\CC^\II\times\ZZ^\II$. 
\begin{Proposition}\label{prop3.30}
    Let $\theta \in \ZZ^\II$. If $(A,B^-,B^+,a,b,C,D) \in \mu^{-1}(\lambda)^{\theta-s}$ then its stabilizer is trivial.
\end{Proposition}
\begin{proof} 
The proof for the affine type A case in \cite[Lemma 2.10]{NT17} works here. Indeed, let us consider an element $g \in \GG$ stabilizing $(A,B^\pm,a,b,C,D)$. Then $S=\bigoplus\limits_{\zeta\in\IS}\Ima(g_\zeta-\operatorname{id})$ and $T=\bigoplus\limits_{\zeta\in\IS}\Ker(g_\zeta-\operatorname{id})$ satisfy, respectively, $(\nu_1)$ and $(\nu_2)$ in Definition \ref{dncB}. Hence, the numerical criterion for stable points (Theorem \ref{ncB}) implies $S=0$ (or equivalently $T=V$).
\end{proof}
Next, we want to study the stable locus of $\mu^{-1}(\lambda)$ by considering the latter as a subvariety of $\MM$ as explained in Remark \ref{rem3.26}. Given $m=(A,B^-,B^+,a,b,C,D) \in \MM$ we have maps:
\begin{equation}
\begin{gathered}
    \GG \xrightarrow{\psi^m} \MM \xrightarrow{\overline{\mu}} \bigoplus\limits_{x\in\Lambda}\Hom(V_{\zeta_x^-},V_{\zeta_x^+})\oplus\mathfrak{g},\\
    g \mapsto g\cdot m \mapsto \overline{\mu}(g\cdot m).
\end{gathered}
\end{equation}
We can differentiate them:
\begin{equation}
    \mathfrak{g} \xrightarrow{d(\psi^m)_{\operatorname{id}}} \MM \xrightarrow{d(\overline{\mu})_m} \bigoplus\limits_{x\in\Lambda}\Hom(V_{\zeta_x^-},V_{\zeta_x^+})\oplus\mathfrak{g}.
\end{equation}
Explicitly, we have
\begin{equation}
    d(\psi^m)_{\operatorname{id}}(\xi_\zeta)=\begin{bmatrix}
        \xi_{\zeta^+}A-A\xi_{\zeta^-}\\
        [\xi_{\zeta^\pm},B_{\zeta^\pm}]\\
        -b\xi_{\zeta^-}\\
        \xi_{\zeta^+}a\\
        \xi_{\zeta_{h(e)}}C_e-C_e\xi_{\zeta_{t(e)}}\\
        \xi_{\zeta_{t(e)}}D_e-D_e\xi_{\zeta_{h(e)}}
    \end{bmatrix},
\end{equation}
and $d(\overline{\mu})_m=(d\mu_1)_m+(d\mu_2)_m$, where
\begin{equation}
    (d\mu_1)_m(\begin{bmatrix}
        \bdot{A}\\
        \bdot{B^\pm}\\
        \bdot{b}\\
        \bdot{a}\\
        \bdot{C}\\
        \bdot{D}
    \end{bmatrix})=\Bigl(B_x^+\bdot{A_x}+\bdot{B_x^+}A_x-\bdot{A_x}B_x^--A_x\bdot{B_x^-}+\bdot{a}b+a\bdot{b}\Bigr),
\end{equation}
and $(d\mu_2)_m$ is given by
\begin{equation}
    \input{dmu2}
\end{equation}
The following proposition has been proved in \cite[Proposition 2.12]{NT17} for the affine type A case.
\begin{Proposition}\label{4.13}
    Under the previous assumptions, the following hold.
\begin{enumerate}[label=(\roman*)]
    \item If $m\in \mu^{-1}(\lambda)$ is $\theta$-stable, then $d(\psi^m)_{\operatorname{id}}$ is injective.
    \item If $m \in \mu^{-1}(\lambda)$, then $(d\mu_1)_m$ is surjective.
    \item If $m \in \mu^{-1}(\lambda)$ is $\theta$-stable, then $(d\overline{\mu})_m$ is surjective.
\end{enumerate}
\end{Proposition}
\begin{proof}We can use the argument in \cite[Proposition 2.12]{NT17} because it does not involve restrictions on the number of arrows incoming to and outgoing from a wavy line.
\end{proof}
\begin{Corollary}\label{prop3.31}
    For any $(\lambda,\theta) \in \CC^\II\times\ZZ^\II$, $\mu^{-1}(\lambda)^{\theta-s}$ is a (possibly empty) nonsingular subvariety of $\TM$. In particular, if $\M^s(B,\Lambda,v)$ is the (possibly empty) open subvariety of $\M(B,\Lambda,v)$ induced by the stable locus, it has to be nonsingular of dimension $\dim\MM-\dim\Bigl(\bigoplus\limits_{x\in\Lambda}\Hom(V_{\zeta_x^-},V_{\zeta_x^+})\Bigr)-2\dim\GG$.
\end{Corollary}
\begin{proof}
    The proof follows from Proposition \ref{prop3.30} and \ref{4.13}. Alternatively, it follows from Proposition \ref{prop3.30} and general theory on Hamiltonian reductions.
\end{proof}
\begin{Proposition}\label{prop4.15} Let us suppose that for any interval $\sigma \in \II$ there exists at least one segment $\zeta(\sigma)$ such that $v_{\zeta(\sigma)}=0$. Then, the (possibly empty) bow variety is a nonsingular affine variety which is independent of the stability and deformation parameters.
\end{Proposition}
\begin{proof}
    It follows from numerical criterion (Theorem \ref{ncB}) that every point in $\mu^{-1}(\lambda)$ is stable with respect to any stability parameter. Moreover, from Theorem \ref{nCd}, we can suppose to embed $\CC^\II \in \CC^{\IS}$ via $\sigma \mapsto \zeta({\sigma})$ and so, from the assumption on the dimension vector, the fiber of the moment map is independent of the deformation parameters. The claim follows from Corollary \ref{prop3.31}.
\end{proof}
\begin{Example}
    Let us consider the bow diagram:
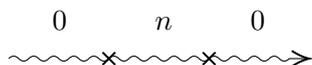
\begin{equation*}
    \input{ex4.3}
\end{equation*}
We notice that $\mu^{-1}(\lambda)=\mu^{-1}(0)\neq\emptyset$ for any $\lambda\in\CC$. From Proposition \ref{prop4.15} we know that bow varieties associated with this bow diagram are nonsingular of dimension $2n$. Let us notice that this is an example of a balanced bow diagram of type A. This class of bow varieties can be interpreted as Coulomb branches (in the sense of Braverman-Finkelberg-Nakajima \cite{braverman2016towards}) for unitary gauge theories (see \cite{NT17} for details).
\end{Example}

\subsection{Local numerical conditions for emptiness}
In this subsection we will state a necessary numerical local condition for bow diagrams in order to generate a non-empty bow variety. We begin by generalizing \cite[Proposition 3.8]{SW23}, originally observed in the proof of \cite[Proposition 7.1]{NT17}, to the case allowing multiple arrows and an arbitrary deformation parameter.
\begin{Proposition}
    Let $(B,\Lambda,v)$ be a bow diagram, $\lambda\in\CC^{\II}$ a deformation parameter and $(A,B^-,B^+,a,b,C,D)\in\mu^{-1}(\lambda)$.
    \begin{enumerate}[label=(\roman*)]
        \item For any local configuration of $(B,\Lambda,v)$ of the type:
    \begin{center}
        \input{locLEFT}
    \end{center}
    the map $\alpha=(A,b,D_{e_1},\dots,D_{e_k}):V_0 \rightarrow V_{-1}\oplus\CC\oplus\bigoplus\limits_{j=1}^{k}V_j$ is injective.
    \item For any local configuration of $(B,\Lambda,v)$ of the type:
        \begin{center}
            \input{locRIGHT}
        \end{center}
    the map $\beta=A+a+D_{e_1}+\cdots+D_{e_k}:V_{-1}\oplus\CC\oplus\bigoplus\limits_{j=1}^{k}V_j\rightarrow V_0$ is surjective.
        \end{enumerate}
\end{Proposition}
\begin{proof}
    \begin{enumerate}[label=(\roman*)]
        \item By assumption $B_x^-=-\sum\limits_{j=1}^{k}C_kD_k+\lambda_\sigma$, where $\sigma$ is the interval we are considering. In particular, $\bigcap\limits_{j=1}^{k}\Ker(D_{e_j})$ is contained in the eigenspace of $B_x^-$ of eigenvalue $\lambda_\sigma$. Thus, $\Ker(\alpha) \subset \Ker(A)\cap\Ker(b)$ is $B_x^-$-invariant. Hence, $(S1)$ implies that $\alpha$ is injective.
        \item By assumption $B_x^+=-\sum\limits_{j=1}^{k}D_jC_j$. Then, $\Ima A+\Ima a \subset \Ima\beta$ and $\Ima\beta$ is $B_x^+$-invariant. Hence, $(S2)$ implies that $\beta$ is surjective.
    \end{enumerate}
\end{proof}
\begin{Corollary}\label{SW}
   Suppose $\M_{\lambda,\theta}(B,\Lambda,v)\neq\emptyset$, for some $\lambda\in\CC^\II$ and $\theta\in\ZZ^\II$. Then, for any local configuration in $(B,\Lambda,v)$ of the following types
    \begin{equation*}
        \input{loc1}
    \end{equation*}
    we have $v_0\leq v_{-1}+\sum\limits_{j=1}^{k}v_j+1$.
\end{Corollary}
The following example shows that not even for finite type A bows is this necessary condition also sufficient.

\begin{Example}\label{ExSW}
    Let us consider the bow diagram:
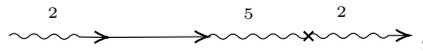
\begin{center}
    \input{exlocalsusy},
\end{center}
which clearly satisfies the necessary condition in Corollary \ref{SW}. Suppose there exists a solution $m=(A,B^\pm,a,b,C,D)$ to $\mu=\lambda$ for some $\lambda=(\lambda_1,\lambda_2)\in\CC^2$. That is:
\begin{center}
   \begin{tikzcd}
0 \arrow[r, "C", shift left] & \mathbb{C}^3 \arrow[rd, "b_1"'] \arrow[rr, "A"] \arrow["B^-"', loop, distance=2em, in=125, out=55] \arrow[l, "D", shift left] &                               & \mathbb{C}^2 \arrow["B^+"', loop, distance=2em, in=125, out=55] \arrow[rr] \arrow[rd, "b_2"'] &                       & 0 \\
                             &                                                                                                                               & \mathbb{C} \arrow[ru, "a_1"'] &                                                                                               & \mathbb{C} \arrow[ru] &  
\end{tikzcd}
\end{center}
satisfying:
\begin{enumerate}
    \item[(a)] $B^+A-AB^-+a_1b_1=0$
   \item [(S1)$_1$] If $S^-\subset\Ker A \cap\Ker b_1$ is $B^-$-invariant, then $S^-=0$.
    \item [(S2)] $A$ is surjective.
    \item [(S1)$_2$] If $S^-\subset \Ker b_2$ is $B^+$-invariant, then $S^-=0$.
    \item[(I)] $B^-=\lambda_2\operatorname{id}$ 
\end{enumerate}
Therefore, from (I) and (S1)$_1$, we have $\Ker A \cap \Ker b_1 = \{0\}$ and in particular $b_1 \neq 0$. Hence, using (a) and (S2), we deduce that $B^+=\lambda_2\operatorname{id}$. This is in contraddiction with (S1)$_2$.
\end{Example}

\subsection{Hanany--Witten transitions}
In \cite[Section 7]{NT17}, Nakajima and Takayama introduced an equivalence relation for affine type A bow diagrams such that two equivalent bow diagrams give rise to the same bow varieties. It admits the following generalization to bow diagrams with arbitrary underlying quivers.
\begin{Definition}
    Let $(B,\Lambda,v)$ and $(B',\Lambda',v')$ be bow diagrams. We say that one is obtained from the other via a \textit{Hanany--Witten transition} if we can go from one to the other by exchanging the position of an arrow and an x-point in a local subdiagram as follows:
\begin{equation}
    \begin{aligned}\input{HWdiagramIntro}\label{hwtransition-intro}\end{aligned}
\end{equation}
Two bow diagrams related by a sequence of Hanany--Witten transitions are said to be \textit{Hanany--Witten equivalent}.
\end{Definition}
We notice that, given a local configuration as in (\ref{hwtransition-intro}), we can perform the Hanany--Witten transition if and only if it satisfies the corresponding inequalities given in Corollary \ref{SW}.

We observe that two Hanany--Witten equivalent bow diagrams share the same underlying bow (not necessarily of type A), and therefore the same spaces of deformation and stability parameters.
\begin{Theorem}\label{NTHW}
    Let $(B,\Lambda,v)$ and $(B,\Lambda',v')$ be bow diagrams related by a Hanany--Witten transition. Then, there exists an isomorphism of Poisson algebraic varieties
    \begin{equation*}
        \M_{\lambda,\theta}(B,\Lambda,v) \cong \M_{\lambda,\theta}(B,\Lambda',v')
    \end{equation*}
    for every $\lambda\in\CC^\II$ and $\theta\in\ZZ^\II$.
\end{Theorem}
\begin{proof}
Since Hanany--Witten transitions are local moves, Nakajima and Takayama's argument for the affine type A case works here (see \cite[Proposition 7.1]{NT17}). Indeed, in the notation of (\ref{hwtransition-intro}), Nakajima and Takayama's proof provides a $\GL(v_-)\times\GL(v_+)$-equivariant isomorphism of symplectic varieties between the reductions of $\TM(B,\Lambda,v)$ and $\TM(B,\Lambda',v')$ at, respectively, $v$ and $v'$. The construction of this isomorphism only involves the local subdiagrams affected by the Hanany--Witten transition and therefore it works as for the type A case.
\end{proof}
\begin{Corollary}\label{SWHW}
    Let $(B,\Lambda,v)$ be a bow diagram such that $\M_{\lambda,\theta}\neq\varnothing$ for some $\lambda\in\CC^\II$ and $\theta \in \ZZ^\II$. Then, we can perform every possible sequence of Hanany--Witten transitions for $(B,\Lambda,v)$.
\end{Corollary}
\begin{proof}
    We observe that, from Theorem \ref{NTHW}, Hanany--Witten transitions preserve the property of bow diagrams of generating a non-empty bow variety. Therefore, the Corollary follows from Corollary \ref{SW}.
\end{proof}
\begin{Remark}
Let us consider bow diagrams of affine type A. In this case, Corollary \ref{SWHW} provides a necessary condition for bow diagrams to generate a non-empty bow variety which is stronger than the one in Corollary \ref{SW}.
In Example \ref{ExSW}, we showed that the condition in \ref{SW} is not a sufficient condition. However, in the example, we considered the bow diagram
\begin{center}
\input{exlocalsusy},
\end{center}
for which it is easy to prove that not every sequence of Hanany--Witten transitions can be performed. Indeed, it is Hanany--Witten equivalent to the bow diagram
\begin{center}
\input{exlocalsusy2}.
\end{center}
Therefore, it is natural to wonder if the possibility of performing every Hanany--Witten transitions for a bow diagram of affine type A guarantees the existence of a non-empty bow variety. We will prove this in a sequel paper on affine type A bow varieties and supersymmetric brane diagrams.
\end{Remark}

\section{Cobalanced bow varieties}\label{S5}
In his original work \cite{cherkis2011instantons}, Cherkis stated that Nakajima quiver varieties correspond to bow varieties associated with particular bow diagrams, called \textit{cobalanced}. For a proof in the affine type A case, see \cite[Theorem 2.15]{NT17}. In this section, we provide a proof of this correspondence for arbitrary bows using the quiver description. We deduce this correspondence as a corollary of a stronger result. Indeed, we show that the construction of cobalanced bow varieties can be divided into two steps, with the first step recovering the framework of Nakajima quiver varieties.

\subsection{Cobalanced bow diagrams}\label{s5.1}
\begin{Definition}
Given a bow diagram $(B,\Lambda,v)$, an x-point $x$ is called \textit{cobalanced} if $v_{\zeta_x^-}=v_{\zeta_x^+}$. A bow diagram is said to be a \textit{cobalanced bow diagram} if every x-point is cobalanced. In such a case, we also say that the dimension vector $v$ is a \textit{cobalanced dimension vector}.
\end{Definition}
    Let $B=(\II,\EE)$ be a bow with underlying quiver $Q=(I,E)$. Then, a cobalanced bow diagram over $B$ is equivalent to a framed dimension vector for $Q$, in the following way. Given $(v,w) \in \NN^I\times\NN^I$ a framed dimension vector for $Q$, because of the natural bijection between $I$ and $\II$, we can define a set of x-points over $B$ by taking $w_\sigma$ x-points on each wavy line $\sigma$ and we can define a dimension vector $\Tilde{v}\in \NN^{\IS}$ by defining for any segment $\zeta$, $\Tilde{v}_\zeta = v_\sigma$ where $\sigma$ is the wavy line containing $\zeta$. Conversely, let us take a cobalanced bow diagram $(\Lambda,\Tilde{v})$ over $B$. Then we can define $v \in \NN^\II$ by $v_\sigma=v_\zeta$ where $\zeta$ is any segment in the interval $\sigma$ and $w\in \NN^\II$ by $w_\sigma=$number of x-points on $\sigma$. This is a one-to-one correspondence between cobalanced bow diagrams and framed dimension vectors.
    
    From now on, given $(B,\Lambda)$ we will not distinguish between a cobalanced dimension vector for $(B,\Lambda)$ and a dimension vector for the underlying quiver $Q$, i.e. we will denote by $v$ both of them.

\begin{Remark}\label{r5.2}
    Thanks to Lemma \ref{3.9}, in the cobalanced case, if $(A,B^\pm,a,b,C,D) \in \TM$ then $A_x$ is an isomorphism for any x-point. Furthermore, given a free triangle $(A,B^-,B^+,a,b)$ satisfying condition $(a)$, if $A$ is an isomorphism, then $(A,B^-,B^+,a,b)$ satisfies conditions $(S1)$ and $(S2)$. In other words, in the cobalanced case we can replace conditions $(S1)$ and $(S2)$ on triangles in $\mu_1^{-1}(0)$ with the condition    that maps $A$ are isomorphisms.
\end{Remark}

\subsection{Example: from bow to quiver varieties}\label{P1}
Let us show through a concrete example how to recover the framework of Nakajima quiver varieties starting from bow data.

Let $B$ be a Dynkin bow of finite type $A_1$ (that is a wavy line) and $Q$ the underlying quiver. Let  $\Lambda=\{0,1\}$ be the set of x-points numbered following the wavy line orientation. Note that they partition the wavy line into three segments. Let $\lambda=0$ be the deformation parameter, $\theta \in \NN_{>0}$ the stability parameter and $v=(1,1,1)$ the (cobalanced) dimension vector. Hence $\GG=\CC^\times\times\CC^\times\times\CC^\times$ and the bow diagram is given by: 
\begin{equation*}
\begin{aligned}    \input{bowA1}
\end{aligned}
\end{equation*}

By definition, elements of $\TM$ are given by $(A_0,B_{0}^\pm,a_0,b_0,A_1,B_{1}^\pm,a_1,b_1)$ illustrated by the following diagram
\begin{center}
    \input{BowEX-1}
\end{center}
such that $A_0$, $A_1$ are isomorphism (i.e. conditions $(S1)$ and $(S2)$, see Remark \ref{r5.2}) and $B_i^+A_i-A_iB_i^-+a_ib_i=0$ for $i=0,1$. 

Let us decompose $\GG=\CC^\times \times \bigl(\CC^\times\times\CC^\times\bigr)=G_1 \times H$, where $G_1=\GL(1)$. Since the $\GG$-action is Hamiltonian with moment map $\mu,$ the induced action of its subgroups $G_1$ and $H$ are also Hamiltonian with moment maps, respectively:
\begin{gather*}
\mu^{G_1}:\TM \rightarrow \CC, \ (A_0,B_{0}^\pm,a_0,b_0,A_1,B_{1}^\pm,a_1,b_1) \mapsto B_0^-, \\ 
\mu^H:\TM \rightarrow \CC \oplus \CC, \ (A_0,B_{0}^\pm,a_0,b_0,A_1,B_{1}^\pm,a_1,b_1) \mapsto \bigl(-B_o^++B_1^-, -B_1^+\bigr). \end{gather*}
Since $H$ acts freely on $\TM$, then the (GIT) quotient $(\mu^H)^{-1}(0)/H$ has a symplectic structure (this just follows from the general theory on Hamiltonian reductions). Furthermore, since the action of $G_1$ and $H$ commute, the $G_1$-action induces a Hamiltonian action on the quotient, with moment map induced by $\mu^{G_1}$. Therefore, we have
\begin{equation*}
    \Bigl((\mu^H)^{-1}(0)/H\Bigr)\HR_{\lambda,\theta}G_1=\Bigl(\TM\HR_{0,0} H\Bigr)\HR_{\lambda,\theta}G_1=\TM\HR_{\lambda,\theta}\GG,
\end{equation*}
where the previous identifications follow by construction. Let us focus on $\TM\HR_{0,0} H$. We have that $\bigl(\mu^H\bigr)^{-1}(0)$ is given by
\begin{center}
    \input{BowEX-2}
\end{center}
such that $A_0$, $A_1$ are isomorphisms, $B_2=0$ and $A_0B_0A_0^{-1}-a_0b_0A_0^{-1}=B_1=A_1^{-1}a_1b_1$. Thus, 
\begin{center}
$\TM\HR_{0,0} H=\bigl(\mu^H\bigr)^{-1}(0)/H\cong\Bigl\{
\begin{tikzcd}
\mathbb{C} \arrow[r, "a_0"', shift right] & \mathbb{C} \arrow[l, "b_0"', shift right] \arrow[r, "b_1", shift left] \arrow["B_0"', loop, distance=2em, in=125, out=55] & \mathbb{C} \arrow[l, "a_1", shift left]
\end{tikzcd} \ \vert \ B_0 - a_0b_0 = a_1b_1\Bigr\}\cong $
\end{center}
\begin{equation*}
\Bigl\{ \CC \mathrel{\mathop{\rightleftarrows}^{\mathrm{J}}_{\mathrm{I}}} \CC^2  \ \vert \ J=\begin{pmatrix}  b_0 \\ b_1 \end{pmatrix}, \ I = \begin{pmatrix}  a_0, & a_1\end{pmatrix} \Bigr\}=\mathrm{Rep}\bigl(\overline{Q^F},1,2\bigr). 
\end{equation*}
In particular, we have an isomorphism between $\TM\HR_{0,0} H$ and $\mathrm{Rep}\bigl(\overline{Q^F},1,2\bigr)$ as algebraic varieties. It is also $G_1$-equivariant if we consider the $G_1$-action on $\mathrm{Rep}\bigl(\overline{Q^F},1,2\bigr)$ given by $(\ref{eq.2.16})$. We will show that it also respects the symplectic structures (see \S\ref{s5.3}). Furthermore, it respects the moment maps:
\begin{center}
    \small{\input{BowEX-3}}
\end{center}
where $\hat{\mu}$ is the induced moment map for the induced $G_1$-action on $\TM\HR_{0,0} H.$ Consequently, we obtain an isomorphism of algebraic varieties between the associated bow variety $\mathcal{M}_{\lambda,\theta}\bigl(B,\Lambda,v\bigr)$ and the Nakajima quiver variety $\mathcal{M}_{\lambda,\theta}\bigl(Q,v,w\bigr)$ which is isomorphic to $\T^*\mathbb{P}^1$ (see \S\ref{E.1}).

\subsection{Bow varieties and quiver varieties}\label{s5.3}
In this subsection we will generalize the previous example to every cobalanced bow diagram.

Let us establish the set-up. Let $B=(\II,\EE)$ be a bow. We will denote the underlying quiver by $Q=(\II,\EE)$. Let $(B,\Lambda,v)$ be a cobalanced bow diagram over $B$ and $(v,w)$ the corresponding framed dimension vector for $Q$, as in \S\ref{s5.1}. Let us assume that a basis for each involved vector space has been fixed, that is $V_\zeta=\CC^{v_\sigma}$ for any segment $\zeta \subset \sigma \in\II$ and $V_\sigma=\CC^{v_\sigma}$, $W_\sigma=\CC^{w_\sigma}$ for any $\sigma \in \II$.
Then, $G_v=\prod\limits_{\sigma \in \II}\GL(v_\sigma)$ and $\GG=\Bigr(\prod\limits_{\sigma \in \II}\GL(v_{\zeta_\sigma})\Bigl) \times \Bigl(\prod\limits_{\zeta \neq \zeta_\sigma}\GL(v_\zeta)\Bigr)=G_v \times H$ where we remind the reader that $\zeta_\sigma$ denotes the first segment on a wavy line $\sigma$. 
Finally, let us fix a pair of parameters $(\lambda,\theta)\in\CC^\II\times\ZZ^\II$, which will be simultaneously considered as deformation and stability parameters for both quiver and bow varieties by injecting $\II$ into $\IS$ via first segments of intervals. 
Let us recall that the $G_v$-action on $\T^*\mathrm{Rep}(Q,v,w)$ is Hamiltonian with respect to the standard symplectic structure and the moment map is denoted by $\mu_{v,w}$ (see \S\ref{NQV}). We also know that the $\GG$-action on $\TM$ is Hamiltonian with moment map $\mu$ (see \S\ref{triangles}). Furthermore, regarding $H$ as subgroup of $\GG$, we have a Hamiltonian $H$-action on $\TM$ with natural moment map $\mu^{H}$ induced by $\mu$.

We are now ready to state the main theorem of this paper.
\begin{Theorem}\label{4.1}
Under the previous assumption, $\TM\HR_{0,0} H$ inherits a structure of symplectic algebraic variety and a Hamiltonian $G_v$-action. Moreover, there is a  $G_v$-equivariant isomorphism of symplectic affine varieties
\begin{equation*}
\TM\HR_{0,0} H \cong \T^*\mathrm{Rep}(Q,v,w)\end{equation*}
which respects the moment maps.
\end{Theorem}

\begin{proof}
Let us take a closer look at the action of $H$ on $\TM$ in the cobalanced case. First, we notice that $H$ acts freely. Indeed, let us fix a wavy line $\sigma\in\II$ with at least an x-point on it (otherwise there is only the first segment in it and $H$ is trivial). We order x-points and segments in $\sigma$ following the wavy line's orientation; namely:
\begin{center}
    \input{xptsorder}.
\end{center}
We consider the action of $\prod\limits_{i=1}^{w_\sigma}\GL(v_{\zeta_i})$ on $\prod\limits_{x=1}^{w_\sigma}\TM^x$ given by:
\begin{equation*}
    k\cdot(A_x,B^{\pm}_x,a_x,b_x)=(k_1A_0,B_0^{-},k_1B^{+}_0k_1^{-1},k_1a_o,b_0,k_2A_1k_1^{-1},\dots),
\end{equation*}
for any $k=(k_i)_i \in \prod\limits_{i=1}^{w_\sigma}\GL(v_{\zeta_i})$ and $(A_x,B_x^{\pm},a_x,b_x)_x \in \prod\limits_{x=1}^{w_\sigma}\TM^x$. It follows from Remark \ref{r5.2} that this action is free and so also the $H$-action on $\TM$ is free.
 Therefore, the Hamiltonian reduction $\TM\HR_{0,0} H:=\bigl(\mu^H\bigr)^{-1}(0)\GIT H=\bigl(\mu^H\bigr)^{-1}(0)/H$ is non-singular and inherits a symplectic structure from $\TM$. Let us describe $\TM\HR_{0,0} H$ in terms of Hurtubise normal forms. We have:
\begin{equation*}
    \TM\cong\prod\limits_{\sigma \in \II}\prod\limits_{x \in \sigma}\Bigl(\T^*\GL(v_\sigma) \times \T^*\CC^{v_\sigma}\Bigr) \times \bigoplus\limits_{s \in \EE}\M^s.
\end{equation*}
Using the notation in \S\ref{triangles}, for an element $(u_x,h_x,I_x,J_x,C_s,D_s)_{x,s} \in \TM$ the condition of belonging to $(\mu^H)^{-1}(0)$ is equivalent to satisfying the following conditions:
\begin{center}
    \input{mu2H}.
\end{center}
It follows that we have an isomorphism of affine $\GG$-varieties:
\begin{equation*}
    (\mu^H)^{-1}(0) \cong \prod\limits_{\sigma \in \II}\prod\limits_{x \in \sigma}\Bigl(\GL(v_\sigma) \times \T^*\CC^{v_\sigma}\Bigr) \times \bigoplus\limits_{s \in \EE}\M^s.
\end{equation*}
Let us notice that for any element $m \in(\mu^H)^{-1}(0)$, there exists a unique element in its orbit of the form $(u_x,h_x,I_x,J_x,C_s,D_s)_{x\in\Lambda,s\in\EE}$ such that $u_x$ is the identity map for any $x \in \Lambda$.
This induces an isomorphism of affine $\GG$-varieties
\begin{equation}\label{5.4}
    \TM\HR_{0,0} H \cong \bigoplus\limits_{\sigma \in \II}\bigoplus\limits_{x \in \sigma} \T^*\CC^{v_\sigma} \oplus \bigoplus\limits_{s \in \EE}\M^s.
\end{equation}
Furthermore, by equipping the variety on the right-hand side with the standard symplectic structure, the previous isomorphism respects the symplectic structures.
We will denote elements of $\bigoplus\limits_{\sigma \in \II}\bigoplus\limits_{x \in \sigma} \T^*\CC^{v_\sigma} \oplus \bigoplus\limits_{s \in \EE}\M^s$ by $\bigl(\Tilde{I}_x, \Tilde{J}_x, C_s,D_s\bigr)_{\substack{x\in\Lambda \\ s \in \EE}}$.
 Finally, regarding $G_v$ as subgroup of $\GG$ (via the first segments), we have a Hamiltonian $G_v$-action on $\TM$ with a natural moment map $\mu^{G_v}$. 
 Since the actions of $G_v$ and $H$ on $\TM$ commute, we have a Hamiltonian action of $G_v$ on the Hamiltonian reduction $\TM\HR_{0,0} H$ with an induced moment map denoted by $\hat{\mu}$. 
 Given an x-point $x$, let us denote by $\sigma(x)$ the wavy line containing $x$. Then, the isomorphism $(\ref{5.4})$ is $G_v$-equivariant with respect to the $G_v$-action on $\bigoplus\limits_{x \in \Lambda} \T^*\CC^{v_{\sigma(x)} } \oplus \bigoplus\limits_{s \in \EE}\M^s$ given by changing bases. 
 Let us notice that the $G_v$-action on $\bigoplus\limits_{x \in \Lambda} \T^*\CC^{v_{\sigma(x)} } \oplus \bigoplus\limits_{s \in \EE}\M^s$ is a Hamiltonian action with a standard moment map because it is the lifting of the $G_v$-action on $\bigoplus\limits_{x \in \Lambda} \CC^{v_{\sigma(x)} } \oplus \bigoplus\limits_{s \in \EE}\Hom(V_{t(s)},V_{h(s)})$ given by changing bases. Furthermore, isomorphism $(\ref{5.4})$ respects these moment maps, so we obtain the following description of $\hat{\mu}$:
 \begin{gather}
     \hat{\mu}: \bigoplus\limits_{x \in \Lambda} \T^*\CC^{v_{\sigma(x)} } \oplus \bigoplus\limits_{s \in \EE}\M^s \rightarrow \mathfrak{g}_v, \\
 \notag    \bigl(\Tilde{I}_x,\Tilde{J}_x,C_s,D_s\bigr) \mapsto \ \Tilde{I}\Tilde{J}+[C,D]
 \end{gather}
where $\Tilde{I}\Tilde{J}=\sum\limits_{x\in\Lambda}\Tilde{I}_x\Tilde{J}_x$.

Let us consider isomorphism $(\ref{5.4})$:
\begin{equation*}
    \TM\HR_{0,0} H= \bigoplus\limits_{\sigma \in \II}\bigoplus\limits_{x \in \sigma}\Bigl(\CC^{v_\sigma} \oplus \bigl(\CC^{v_\sigma}\bigr)^*\Bigr) \oplus \bigoplus\limits_{s \in \EE}\T^*\Hom(V_{o(s)},V_{i(s)}).
\end{equation*}
Then, for each wavy line, let us order its x-points by following the wavy line orientation. We define
\begin{gather*}
   \Phi: \TM\HR_{0,0} H \rightarrow \T^*\mathrm{Rep}(Q,v,w),\\
     \bigl(\Tilde{I}_x, \Tilde{J}_x, C_s,D_s\bigr)_{\substack{x \in \Lambda\\s \in \EE}} \mapsto \bigl(x_s,y_s,J_\sigma,I_\sigma\bigr)_{\substack{\sigma \in I\\ s \in E}}
\end{gather*}
where $x_s=C_s$, $y_s=D_s$ and the ordering of x-points on a wavy line $\sigma$ gives \begin{gather*}I_\sigma=\bigoplus\limits_{x \in \sigma}\Tilde{I}_x : \CC^{w_\sigma} \rightarrow \CC^{v_\sigma}\\J_\sigma=\bigoplus\limits_{x \in \sigma}\Tilde{J}_x : \CC^{v_\sigma} \rightarrow \CC^{w_\sigma}.\end{gather*} Hence, the first part of the statement follows. For the second part, it is sufficient to recall what are the moment maps here. Indeed, in terms of Hurtubise normal forms, we have:
\begin{gather*}
    \mu_{v,w}: \T^*\mathrm{Rep}(Q,v,w) \rightarrow \mathfrak{g}_v, \ \bigl(I,J,x,y) \mapsto IJ+[x,y],\\
    \hat{\mu}:\TM\HR_{0,0} H \rightarrow \mathfrak{g}_v, \ \bigl(\Tilde{I},\Tilde{J},C,D\bigr) \mapsto \sum\limits_x\Tilde{I}_x\Tilde{J}_x + [C,D].\end{gather*}\end{proof}

\begin{Remark}
The proof we made make clear that symplectic structures are preserved under the isomorphism $\Phi:\TM\HR_{0,0} H \rightarrow \T^*\mathrm{Rep}(Q,v,w)$. However, we could define it without using Hurtubise normal forms but only the quiver description as follows. We start defining a map $\phi : \bigl(\mu^H\bigr)^{-1}(0) \rightarrow \T^*\mathrm{Rep}(Q,v,w)$. Let us recall that if $(A,B^\pm,a,b,C,D) \in \bigl(\mu^H\bigr)^{-1}(0)$ then $B$ maps are encoded in $(A,a,b,C,D)$ and maps $A$ are isomorphisms (as for maps $u,h$ in the Hurtubise normal form, see proof of Theorem \ref{4.1}). Because of the fact that maps $A$ are isomorphisms, maps $a,b,C$ and $D$ are equivalent to linear maps defined without using vector spaces placed on non-first segments. This way to rewrite points in $\bigl(\mu^H\bigr)^{-1}(0)$ gives a natural map to $\T^*\mathrm{Rep}(Q,v,w)$. An example of this can be found in Example \ref{P1}.
\end{Remark}
It follows from Theorem \ref{4.1} that bow varieties with cobalanced dimension vector are Nakajima quiver varieties. Conversely, any Nakajima quiver variety can be described as a bow variety with a cobalanced bow diagram. 
\begin{Corollary}\label{maincor} Under the previous assumption, we have an isomorphism of Poisson algebraic varieties:
    \begin{equation*}
    \mathcal{M}_{\lambda,\theta}\bigl(B,\Lambda,v\bigr) = \mathcal{M}_{\lambda,\theta}\bigl(Q,v,w).
\end{equation*}
\end{Corollary}

\appendix
\section{Cherkis' construction}\label{Appendix A}
In this appendix we will review the construction of bow varieties as finite dimensional hyperk\"ahler reductions by Cherkis. In this way, the quiver description above will be motivated.

Let us notice that, originally, Cherkis introduced bow varieties as infinite-dimensional hyperk\"ahler reductions, i.e. hyperk\"ahler reductions of infinite dimensional affine spaces by an infinite dimensional gauge group, see \cite[\S2]{cherkis2011instantons}. In the same article, Cherkis introduced an alternative description of bow varieties as finite dimensional hyperk\"ahler reductions, see \cite[\S8.2]{cherkis2011instantons} and \cite[\S3.3]{NT17}. 

\subsection{Hyperk\"ahler quotients}
In this subsection we introduce some notations and recall the definition of hyperk\"ahler manifolds and hyperk\"ahler quotients by an action of a compact real Lie group.
\begin{Definition}
A \textit{hyperk\"ahler manifold} is a Riemannian manifold $(M,g)$ with three K\"ahler structures $I,J,K$ satisfying the quaternionic relation $IJK=-1$.
\end{Definition}
Given $(M,g,I,J,K)$ a hyperk\"ahler manifold, we denote by $\omega_I,\omega_J,\omega_K$ the induced K\"ahler forms with respect to $I,J,K$ respectively.
\begin{Definition}
    Let $U$ be a compact real Lie group and $\lieu$ the corresponding Lie algebra. A \textit{tri-hamiltonian} action of $U$ on a hyperk\"ahler manifold $(M,g,I,J,K)$ is a smooth action which preserves the hyperk\"ahler structure and is Hamiltonian with respect to all the three K\"ahler forms. If $\mu_I,\mu_J,\mu_K$ denote the moment maps with respect to the three complex structures, we define the \textit{hyperk\"ahler moment map} by:
\begin{equation*}\mu=(\mu_I,\mu_J,\mu_K): M \rightarrow \lieu^*\otimes\RR^3.\end{equation*}
In this case, given $\nu=(\nu_1,\nu_2,\nu_3) \in \lieu^*\otimes\RR^3$ $U$-invariant,  we define the \textit{hyperk\"ahler reduction} of $M$ by $G$ with parameter $\nu$ by:
$$M\HR_\nu U \coloneqq    \mu^{-1}(\nu)/U.$$
\end{Definition}
\begin{Theorem}(Hitchin-Karlhede-Lindstr\"om-Ro\v{c}ek, \cite{hitchin1987hyperkahler})\label{HKLR}
    Under the previous assumptions, if $U$ acts freely, then $M\HR_\nu U$ is smooth and inherits a structure of a hyperk\"ahler manifold.
\end{Theorem}
\begin{Remark}
    If the group action is not free, generally the quotient is not even a manifold. However, Dancer and Swann have proved in \cite{DS97-geometryofsingular} that it has a stratification into hyperk\"ahler manifolds. For more details, see also \cite{May22}.
\end{Remark}
Let us also recall that $\omega_\CC\coloneqq    \omega_J+i\omega_K$ is a holomorphic symplectic structure on the complex manifold $(M,I)$. In the previous set-up, the map $\mu_\CC\coloneqq    \mu_J+i\mu_K$ is holomorphic with respect to $I$.
\subsection{Fundamental moduli spaces of solutions to Nahm equations}
We will define in this subsection moduli spaces of solutions to Nahm equations over a compact interval having poles on the boundary points. We will mainly follow \cite{Bie97}, \cite{bielawski1998asymptotic} and \cite[\S1]{Tak16}. See also \cite{bielawski2007lie} as well as the expository work \cite{May20}.

Let $I=[0,c]\subset\RR$ be an interval parameterized by $0\leq s\leq c$. Given $m,n\in\NN$ such that $m\leq n$, we define the \textit{Nahm space} $\NA_n(m;c)$ as the set of analytic maps
\begin{equation}
    (T_0,T_1,T_2,T_3):(0,c] \rightarrow \lieu(n)\otimes\HH
\end{equation}
such that $T_0$ is analytic at $s=0$ and for $i=1,2,3$, the map $T_i$ has the expansion near $s=0$:
\begin{equation}\label{A.2}
    T_i(s)=\input{Nahmatrix}
\end{equation}
i.e., the $(n-m)\times(n-m)$ lower-diagonal blocks have simple poles, the $m\times m$ upper-diagonal block is analytic in $0$ and the off diagonal blocks are of the form $s^{\frac{n-m-1}{2}}\times(\text{analytic in }s)$. Let us notice that $(\rho_1,\rho_2,\rho_3)$ defines the standard $(n-m)$-dimensional irreducible representation of $\mathfrak{su}(2)$. For more details about boundary conditions, see \cite[\S5.1]{May20}. Since the pole is fixed at a boundary point, we can endow $\NA_n(m;c)$  with a $\operatorname{L}^2$-metric and in particular with a structure of infinite dimensional hyperk\"ahler manifold, as in \cite{bielawski1998asymptotic}. We define $\GG^\RR$ as the group given by smooth maps $g:I \rightarrow \LU(n)$ such that
\begin{equation}
    g(0)\in
    \begin{bmatrix}
        \LU(m) & 0 \\
        0 & \operatorname{id}
    \end{bmatrix}.
\end{equation}
$\GG^\RR$ acts on $\NA_n(m;c)$ by:
\begin{equation}
    g: (T_0,T_1,T_2,T_3) \mapsto (gT_0g^{-1}-\frac{dg}{ds}g^{-1},gT_1g^{-1},gT_2g^{-1},gT_3g^{-1}),
\end{equation}
where $g\in\GG^\RR$ and $(T_0,T_1,T_2,T_3) \in \NA_n(m;c)$. Then, we define the following normal subgroups of $\GG^\RR$:
\begin{equation}\begin{gathered}
\GG^\RR_{00} = \{g\in\GG^\RR \ \vert\ g(0)=g(c)=\operatorname{id}\}, \\
\end{gathered}\end{equation}
The $\GG^\RR_{00}$-action on $\NA_n(m;c)$ is tri-hamiltonian and the hyperk\"ahler moment map equations with parameter $0$ are given by the \textit{Nahm's equations}:
\begin{equation}
\begin{gathered}
    \mu_1=\frac{dT_1}{ds}+[T_0,T_1]+[T_2,T_3]=0, \\
    \mu_2=\frac{dT_2}{ds}+[T_0,T_2]+[T_3,T_1]=0, \\
    \mu_3=\frac{dT_3}{ds}+[T_0,T_3]+[T_1,T_2]=0.
\end{gathered}    
\end{equation}
So, we define the moduli spaces of solutions to Nahm's equations with at most a pole at the left endpoint, by:
\begin{equation}
    F_n(m;c) = \NA_n(m;c)\HR_{0} \GG^\RR_{00}.
\end{equation}
\begin{Remark}
    When $m=n$ we recover the moduli space of solutions to Nahm's equations with regular limits on the boundary points. These moduli spaces were originally studied by Kronheimer \cite{kronheimer2004hyperkahler}.
\end{Remark}
$F_n(m;c)$ inherits a structure of finite dimensional hyperk\"ahler manifold from $\NA_n(m;c)$. Explicitly, the tangent space at a solution $T$ is given by the solutions $(t_0,t_1,t_2,t_3)$ of the following system:
\begin{equation}
    \begin{gathered}
        \Dot{t}_0+[T_0,t_0]+[T_1,t_1]+[T_2,t_2]+[T_3,t_3]=0,\\
        \Dot{t}_1+[T_0,t_1]-[T_1,t_0]+[T_2,t_3]-[T_3,t_2]=0,\\
        \Dot{t}_2+[T_0,t_2]-[T_1,t_3]-[T_2,t_0]+[T_3,t_1]=0,\\
        \Dot{t}_3+[T_0,t_3]+[T_1,t_2]-[T_2,t_1]-[T_3,t_0]=0.
    \end{gathered}
\end{equation}
Let us define the norm on $\lieu(n)$ by $\lVert U \rVert^2=tr(UU^*)$.  Then, $F_n(m;c)$ can be endowed with the following hyperk\"ahler metric
\begin{equation}
    g\Bigl((t_0,t_1,t_2,t_3),(t_0',t_1',t_2',t_3')\Bigr)=\int_0^ctr(t_0t_0^*+t_1t_1'^*+t_2t_2'^*+t_3t_3'^*)ds,
\end{equation}
and the complex structures are given by $t_0+It_1+Jt_2+Kt_3$. Moreover, there is a natural residual action of $\GG^\RR/\GG^\RR_{00} \cong \LU(m)\times \LU(n)$ on $F_n(m;c)$. This action is tri-hamiltonian with hyperk\"ahler moment map given by:
\begin{equation}
\begin{gathered}\mu: F_n(m;c) \rightarrow (\lieu(m)\oplus\lieu(n))^3,\\
[T_0,T_1,T_2,T_3] \mapsto \begin{pmatrix}
    \pi T_1(0), & \pi T_2(0), & \pi T_3(0) \\
    -T_1(c),& -T_2(c), & -T_3(c)
\end{pmatrix},
\end{gathered}\end{equation}
where $\pi:u(n) \rightarrow u(m)$ is the projection onto $m\times m$ upper-diagonal blocks, see \cite[\S5.4]{May20}.
Next, we want to describe the complex-symplectic structure of these moduli spaces. Let us fix the complex structure $I$. Then $(F_n(m;c),I)$ is a holomorphic manifold with complex-symplectic form $\omega_\CC=\omega_J+i\omega_K$. Let us fix complex coordinates on $\NA_n(m;c)$:
\begin{equation}
    \alpha=T_0+iT_1, \ \ \ \beta=T_2+iT_3.
\end{equation}
Then, Nahm equations can be rewritten as:
    \begin{align*}
        \mu_\CC=\frac{d\beta}{ds}+[\alpha,\beta]=0 & & \textit{Complex Nahm equation}, \\
        \mu_\RR=\frac{d}{ds}(\alpha+\alpha^*)+[\alpha,\alpha^*]+[\beta,\beta^*]=0 & & \textit{Real Nahm equation}.
    \end{align*}
Finally, let us denote by $\GG^\CC$ and $\GG^\CC_{00}$ the complexifications of $\GG^\RR$ and $\GG^\RR_{00}$, respectively. We have a $\GG^\CC$-action given by:
\begin{equation}
g:(\alpha,\beta) \mapsto (Ad_g\alpha-\Dot{g}g^{-1},Ad_g\beta).    
\end{equation}
We notice that only the complex Nahm equation is preserved by the $\GG^\CC_{00}$-action. However, the following theorem allows us to see $(F_n(m;c),I)$ as a complex-symplectic quotient.
\begin{Theorem}[\cite{Don84}]
Every $\GG^\CC_{00}$ orbits in $\mu^{-1}_\CC(0)$ meets $\mu_\RR^{-1}(0)$ in exactly one $\GG^\RR_{00}$-orbit. In particular, we have an isomorphism between $F_n(m;c)$ and $\mu_\CC^{-1}(0)/\GG^\CC_{00}$.
\end{Theorem}
Then, on $F_n(m;c)$ there is a natural residual action of $\GG^\CC/\GG^\CC_{00}\cong GL_m\times GL_n$.
Let us define the space of Hurtubise normal forms as follows. 
\begin{Definition}
Let $m,n \in \NN$ such that $m\leq n$. We define $\Hur(m,n)$, as:
    \begin{itemize}
        \item If $m<n$, 
        \begin{equation*}
     \left\{\bigl(u,\eta\bigr) \in \GL_n \times \mathfrak{gl}_n \ \vert \ \eta= \input{HNF}\right\}
    \end{equation*}
  where $(e_0,e,f,g,h) \in \CC \times \CC^{n-m-1} \times (\CC^m)^* \times \CC^m \times \End(\CC^m).$  
    \item If $m=n$,
    \begin{equation*}
    \begin{gathered}
               \GL_n \times \mathfrak{gl}_n \ni  \bigl(u,\eta\bigr).
    \end{gathered}
    \end{equation*}
    \end{itemize}
On these spaces we have the following $\GL_{m}\times \GL_{n}$-action:
\begin{equation}\label{A.16}
    \bigl(g_n,g_m\bigr)\cdot\bigl(u,\eta) = \Bigl(\begin{bmatrix}
            g_m &  0 \\
            0 &  \operatorname{id}_{n-m} 
        \end{bmatrix}ug_n^{-1}, \begin{bmatrix}
            g_m &  0 \\
            0 &  \operatorname{id}_{n-m} 
        \end{bmatrix}\eta\begin{bmatrix}
            g_m^{-1} &  0 \\
            0 &  \operatorname{id}_{n-m} 
        \end{bmatrix}\Bigr),
        \end{equation}
    where $(g_m,g_n) \in \GL(m)\times \GL(n)$ and $(u,\eta) \in \Hur(m,n)$.
\end{Definition}
The previous spaces have a standard symplectic structure given by:
\begin{equation*}
    \omega_s=-tr(d\eta\wedge duu^{-1}+\eta duu^{-1}\wedge duu^{-1}),
\end{equation*}
where $duu^{-1}$ is the right invariant Maurer-Cartan form. With respect to this symplectic structure, the $\GL_m\times\GL_n$-action is Hamiltonian with moment map given by:
\begin{equation}\label{A.17}
    \mu: \Hur(m,n) \rightarrow \mathfrak{gl}_m\oplus\mathfrak{gl}_n, \ (u,\eta) \mapsto (\pi(\eta),-u^{-1}\eta u)
\end{equation}
where $\pi:\mathfrak{gl}_n \rightarrow \mathfrak{gl}_m$ is the projection onto $m \times m$ upper-diagonal blocks.
\begin{Remark}
    Let us notice that the symplectic $\GL_{m}\times \GL_{n}$-varieties $\Hur(m,n)$ are different from the spaces $F(m,n)$ defined in \S\ref{triangles}. The relation between these two constructions will be explained in \S\ref{THK}.
\end{Remark}
As shown in \cite[Theorem 1.15]{hurtubise1989classification}, given $(\alpha,\beta)$ a solution to the complex Nahm equation there exists a gauge transformation $u$ (not necessarily regular at $s=0$) such that:
\begin{equation}
    u\cdot(\alpha,\beta)=\Big(0, \begin{bmatrix}
            h & 0 & g \\
            f & 0 & e_0 \\
            0 & \operatorname{id} & e
        \end{bmatrix}\Bigr),
\end{equation}
where $(e_0,e,f,g,h) \in \CC \times \CC^{n-m-1} \times (\CC^m)^* \times \CC^m \times \End(\CC^m)$.
\begin{Proposition}[\cite{Bie97}, \cite{bielawski1998asymptotic}]\label{PropA18}
    Given $m\leq n$ integers, the following map
\begin{equation}
    F_n(m;c) \rightarrow \Hur(m,n), \ (\alpha,\beta) \mapsto \Bigl(u(c),(u\cdot\beta)(c)\Bigr),
\end{equation}
is a isomorphism of holomorphic symplectic $\GL_m\times\GL_n$-manifolds.
\end{Proposition}
\begin{Remark}
    Although the length of the interval affect the hyperk\"ahler structure on $F_n(m;c)$, up to a symplectomorphism it does not affect the differential structure and the complex-symplectic structure. We notice that this is not surprising, indeed rescaling fibers of Hurtubise normal forms we rescale the complex-symplectic structure on it.
\end{Remark}
Finally, let us discuss moduli spaces of solutions to Nahm equations having poles at the right endpoint of a compact interval and regular limit at the left endpoint. As above, we define $\overline{\NA_n(m;c)}$ the hyperk\"ahler manifold of Nahm data $(T_0(s),T_1(s),T_2(s),T_3(s))$ having poles at $s=c$. We denote by $\overline{\GG}^\RR$ the associated gauge group, which is defined as $\GG^\RR$ but interchanging the role of the boundary points. Hence, we have an isomorphism of groups send $g(s)$ to $g(c-s)$. Then, we define
\begin{equation}
    \overline{F_n(m;c)}=\overline{\NA_n(m;c)}\HR_0\GG^\RR_{00}.
\end{equation} 
We have an isomorphism of hyperk\"ahler manifolds (see \cite[Proposition 5.1]{May20}):
\begin{equation}\label{A.21}
    F_n(m;c) \rightarrow \overline{F_n(m;c)}, \ \ \  T(s) \mapsto -T(c-t).
\end{equation}
Therefore, using isomorphism (\ref{A.21}) and Proposition \ref{PropA18}, we have an isomorphism of complex-symplectic manifolds between $\Hur(m,n)$ and $\overline{F_n(m;c)}$. However, the idea of Hurtubise \cite{hurtubise1989classification} can be used to construct such an isomorphism without involving moduli spaces of solutions to Nahm equations having poles at the left endpoint. In particular, we find that $\overline{F_n(m;c)}$ is biholomorphic to $\Hur(m,n)$ and the pull-back of the holomorphic symplectic structure on $\overline{F_n(m;c)}$ is given by $-\omega_s=tr(d\eta\wedge duu^{-1}+\eta duu^{-1}\wedge duu^{-1})$.
We will denote the space of Hurtubise normal forms endowed with this symplectic structure by $\overline{\Hur(m,n)}$. Moreover, the residual action of $\overline{\GG}^\RR/\GG^\RR_{00}\cong \GL_n\times\GL_m$ on $\overline{F_n(m;c)}$ induces the action of $\GL_n\times\GL_m$ given in (\ref{A.16}). It is clear from the previous case that this action is Hamiltonian with a moment map that differ from the moment map in (\ref{A.17}) by a sign. Let us notice that the isomorphism $F_n(m;c) \rightarrow \overline{F_n(m;c)}$ can be rewritten in terms of Hurtubise normal forms as:
\begin{equation}
    \Hur(m,n) \rightarrow \overline{\Hur(m,n)}, \ (u,\eta) \mapsto (u,-\eta).
\end{equation}
We notice that a translation of the interval induces an isomorphism of hyperk\"ahler manifolds.
\begin{Remark}
    We can ``glue" the previous moduli spaces to obtain moduli spaces of solutions to Nahm equations having poles on both boundary points. Contrarily, we could define moduli spaces of solutions having poles at both boundary points and consider the previous constructions as special cases. This is the point of view taken in \cite[Chapter 5]{May20}.
\end{Remark}

\subsection{Triangles and Nahm equations}\label{THK}
In this subsection we will combine moduli spaces defined in the previous subsection to construct spaces of triangles, as defined in \S\ref{triangles}, from solutions to Nahm equations (see \cite{Tak16}). Let us fix $c,c' \in \RR^{>0}$ and $v_-,v_+\in\NN$. We define $\NA_{v_-,v_+}(c,c')$ the space of maps $(T_0,T_1,T_2,T_3)$ from $I=[-c,c']$ to $\lieu(v_-)\otimes\HH$ if $-c\leq s<0$ and to $\lieu(v_+)\otimes\HH$ if $0<s\leq c'$, that are analytic away from $s=0$ and satisfying the following matching conditions at $s=0$:
\begin{itemize}
    \item If $v_-=v_+$, they have regular left and right limits at $s=0$.
    \item If $v_-<v_+$, the limit of $T_i(s)$ at $s=0$ from the left is equal to the $v_-\times v_-$ upper-diagonal block of the limit from the right. Furthermore, $\lim\limits_{s\rightarrow0^+}T_i(s)$ is regular if $i=0$, and of the form (\ref{A.2}) if $i=1,2,3$.
    \item If $v_->v_+$, they satisfy analogous matching conditions.
\end{itemize}.
Associated to these spaces we have the gauge group $\GG^\RR$ given by maps $g$ on $I$ with image in $\LU(v_-)$ if $s\in[-c,0)$ and in $\LU(v_+)$ if $s\in(0,c']$, which are analytic away from $s=0$ and satisfying the following matching conditions at $s=0$:
\begin{itemize}
    \item If $v_-=v_+$, $g$ has regular left and right limits at $s=0$
    \item If $v_-<v_+$, then
    \begin{equation*}
        g(0^+)=\begin{bmatrix}
            g(0^-) & 0\\
            0 & \operatorname{id}_{(v_+-v_-)}
        \end{bmatrix}
    \end{equation*}
    \item If $v_->v_+$, $g$ satisfies analogous matching conditions.
\end{itemize}
$\GG^\RR$ acts on $\NA_{v_-,v_+}(c,c')$ by:
\begin{equation*}
    g:(T_0,T_1,T_2,T_3) \mapsto (gT_0g^{-1}-\Dot{g}g^{-1},gT_1g^{-1},gT_2g^{-1},gT_3g^{-1}).
\end{equation*}
We denote by $\GG^\RR_{00}$the subgroup of $\GG^\RR$ of maps $g$ such that $g$ is the identity at boundary points. Hence, we define
\begin{equation}
    F_{v_-,v_+}(c,c')=\begin{cases}
        \NA_{v_-,v_+}(c,c')\HR_0\GG^\RR_{00} & \text{ if } \ v_-\neq v_+, \\
        \NA_{v,v}(c,c') \times \CC^v\times(\CC^v)^*\HR_0\GG^\RR_{00} & \text{ if } \ v_-= v_+=v. \\
    \end{cases} 
\end{equation}
where in case $v_-=v_+$, $g\in\GG^\RR_{00}$ acts on $\CC^v\times(\CC^v)^*$ sending $(I,J)$ to $(g(0)I,Jg(0)^{-1})$ and $\CC^v\times(\CC^v)^*$ has the standard hyperk\"ahler structure as total space of a cotangent bundle.
We can separate these reductions into two steps: first we quotient by $\GG^\RR_0=\{g\in\GG^\RR_{00} \ \vert \ g(0)=\operatorname{id}\}$. Then, we take the quotient by $\GG^\RR_{00}/\GG^\RR_{0}\cong \LU(\min(v_-,v_+))$. So, we have:
\begin{equation}
    F_{v_-,v_+}(c,c')\cong \begin{cases}
        \overline{F_{v_-}(v_-;c)}\times F_{v_+}(v_-;c')\HR_0\LU(v_-) & \text{ if } \ v_-<v_+,\\
        \overline{F_{v}(v;c)}\times F_{v}(v;c')\times\CC^v\times(\CC^v)^*\HR_0\LU(v) & \text{ if } \ v_-=v_+=v,\\
        \overline{F_{v_-}(v_+;c)}\times F_{v_+}(v_+;c')\HR_0\LU(v_+) & \text{ if } \ v_->v_+,
    \end{cases}
\end{equation}
We notice that $F_{v_-,v_+}(c,c')$ inherits a structure of hyperk\"ahler manifold (see Theorem \ref{HKLR}). Moreover, there is a residual action of $\GG^\RR/\GG^{\RR}_{00}\cong \LU(v_-)\times \LU(v_+)$ which is tri-Hamiltonian with hyperk\"ahler moment map. Finally, let us discuss the complex-symplectic structure of these moduli spaces.
\begin{Theorem}[\cite{hurtubise1989classification}]\label{Theorem A.29}
    There is an isomorphism of holomorphic symplectic manifolds that intertwines the $\GL_{v_-}\times\GL_{v_+}$-actions:
    \begin{equation*}
        F_{v_-,v_+}(c,c')\cong\begin{cases}
            \overline{\Hur(v_-,v_-)}\times \Hur(v_-,v_+)\HR_0\GL(v_-) & \text{ if } \ v_-<v_+,\\
        \overline{\Hur(v,v)}\times \Hur(v,v)\times\CC^v\times(\CC^v)^*\HR_0\GL(v) & \text{ if } \ v_-=v_+=v,\\
        \overline{\Hur(v_+,v_-)}\times \Hur(v_+,v_+)\HR_0\GL(v_+) & \text{ if } \ v_->v_+,
        \end{cases}
    \end{equation*}
\end{Theorem}
\begin{Corollary}
    We have the following isomorphisms:
\begin{equation*}
    F_{v_-,v_+}(c,c')\cong\begin{cases}
            \Hur(v_-,v_+) & \text{ if } \ v_-<v_+,\\
        \overline{\Hur(v,v)}\times\CC^v\times(\CC^v)^* & \text{ if } \ v_-=v_+=v,\\
        \overline{\Hur(v_+,v_-)} & \text{ if } \ v_->v_+,
        \end{cases}
\end{equation*}
\end{Corollary}
\begin{Remark}
Up to rescaling fibers when $v_-<v_+$, we have recovered the symplectic varieties $F(v_-,v_+)$ defined in section \ref{triangles}. By doing so, we can use same formulas used in \S\ref{triangles} for the symplectic structure and moment map.
\end{Remark}

\subsection{Cherkis' bow varieties}
Let $B=(\II,\EE)$ be a bow. Given an interval $\sigma=[\sigma_L,\sigma_R]$, we order x-points in the natural way and denote them by $x_{\sigma,1},\cdots,x_{\sigma,w_\sigma}$. We fix a point for each segment
\begin{align*}
    p_{\sigma,0}=\sigma_L, && p_{\sigma,i}=\frac{x_{\sigma,i}+x_{\sigma,i+1}}{2} \ \text{ if }\ i=1,\cdots,w_{\sigma}-1, && p_{\sigma,w_\sigma}=\sigma_R,
\end{align*}
and define $c_{\sigma,i}^-=\operatorname{length}(p_{\sigma,i-1},x_{\sigma,i})$, $c_{\sigma,i}^+=\operatorname{length}(x_{\sigma,i},p_{\sigma,i})$. That is:
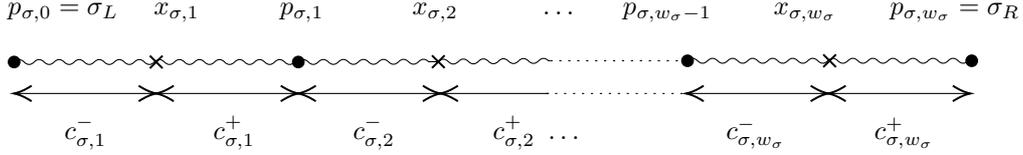
\begin{equation*}
    \input{notation-original-description}
\end{equation*}
Let us fix a bow diagram $(B,\Lambda,v)$. Then we define
\begin{equation}
    \TM^{hk}\coloneqq\prod\limits_{x \in \Lambda}F_{v_x^-,v_x^+}(c_x^-,c_x^+)\times\prod\limits_{e\in\EE}\MM^e.
\end{equation}
It follows from Section \ref{THK} that $\TM^{hk}$ carries a structure of hyperk\"ahler manifold and a tri-hamiltonian action of $\GG^\RR=\prod\limits_{\zeta\in\IS}\LU(v_\zeta)$. Furthermore, there is an isomorphism between $\TM^{hk}$ and $\TM$ as algebraic varieties, where $\TM$ is the affine space defined in (\ref{eq3.19}). This endow $\TM$ with a hyperk\"ahler structure which induces the complex-symplectic structure considered in section \ref{S3}.
As for the quiver description (see Section \ref{s.gdp}), we consider $\RR^\II\subset\RR^{\IS}$ via first segments of intervals and $\RR^{\IS}$ diagonally embedded into $\operatorname{Lie}(\GG^\RR)=\bigoplus\limits_{\zeta\in{\IS}}\lieu(v_\zeta)$. So, for any $\nu\in\RR^{\II}\otimes\RR^3$, the corresponding Cherkis bow variety is given by:
\begin{equation}
    \mathcal{M}^{hk}=\mathcal{M}^{hk}_\nu(B,\Lambda,v)\coloneqq\TM^{hk}\HR_\nu\GG^\RR.
\end{equation}
\begin{Remark}
    As mentioned in the introduction of this appendix, Cherkis originally defined bow varieties as infinite-dimensional hyperk\"ahler reductions. In the same paper, he also introduced two ways to realize bow varieties as finite-dimensional reductions. One of these is the point of view we have taken. The other one was taken by Takayama (\cite{T15}) to construct some ALF spaces as bow varieties. Let us notice that Takayama obtained a different definition of type D bow varieties (see \cite[Remark 7.2]{T15}).

\end{Remark}
Finally, we have an equivalence between the original Cherkis construction and the quiver description in the sense of Kempf-Ness. We will define \textit{the regular locus} of $\mathcal{M}^{hk}_\nu$ as the subset induced by the set of points in the level set of the hyperk\"ahler moment map where $\GG^\RR$ acts freely. 
\begin{Theorem}
    Let $(B,\Lambda,v)$ be a bow diagram. There is a correspondence between hyperk\"ahler parameters $\nu\in \RR^\II\otimes(\RR^2\times\QQ^\II)$ and parameters in the quiver description $(\lambda,\theta)\in \CC^\II\times\QQ^\II$ such that there exists a homeomorphism between $\mathcal{M}^{hk}_\nu$ and $\mathcal{M}_{\lambda,\theta}$. Moreover, it induces an isomorphism of holomorphic symplectic manifolds between the regular locus of $\mathcal{M}^{hk}_\nu$ and the stable locus of $\mathcal{M}_{\lambda.\theta}$.
\end{Theorem}
This results is more or less known. In \cite[Theorem 2.1]{NT17}, Nakajima and Takayama proved it for the affine type A case. They started from the original construction of bow varieties as infinite hyperk\"ahler reduction. Their proof is divided into two steps. The first step of their proof consists in realizing Cherkis bow varieties as finite-dimensional reductions in the same way we defined them in this appendix. Then, the claim follows from the works of Kempf and Ness \cite{kempf1979length},  Kirwan \cite{kirwan1984cohomology}, King \cite{king1994moduli} and others. See also \cite[Chapter 3]{nakajima1999lectures}.
A more detailed proof was given by Takayama in \cite[Proposition 4.7]{T15} for the special case of \textit{elemental} bow varieties. Finally, a version of the Kempf-Ness correspondence which is suitable for our problem has been provided more recently by Mayrand, \cite[Theorem 1.1]{May19}. Let us notice that the mentioned case of elemental bow varieties studied by Takayama immediately follows from the work of Mayrand.

\end{sloppypar}
\noindent Department of Mathematics, Imperial College London, South Kensington Campus, London, SW7 2AZ, UK\\
\noindent \textit{Email address: }t.gaibisso22@imperial.ac.uk
\end{document}

%% file: Affine-D-Bows.tex
\tikzset{every picture/.style={line width=0.75pt}} 

\begin{tikzpicture}[x=0.75pt,y=0.75pt,yscale=-1,xscale=1]

\draw    (141.2,175.86) .. controls (142.9,174.23) and (144.56,174.27) .. (146.19,175.98) .. controls (147.82,177.68) and (149.49,177.72) .. (151.19,176.09) .. controls (152.9,174.46) and (154.56,174.5) .. (156.19,176.21) .. controls (157.82,177.91) and (159.49,177.95) .. (161.19,176.32) .. controls (162.9,174.69) and (164.56,174.73) .. (166.19,176.44) .. controls (167.82,178.14) and (169.49,178.18) .. (171.19,176.55) .. controls (172.9,174.92) and (174.56,174.96) .. (176.19,176.67) .. controls (177.82,178.38) and (179.48,178.42) .. (181.19,176.79) .. controls (182.89,175.16) and (184.55,175.2) .. (186.18,176.9) .. controls (187.81,178.61) and (189.47,178.65) .. (191.18,177.02) .. controls (192.88,175.39) and (194.55,175.43) .. (196.18,177.13) -- (199.92,177.22) -- (199.92,177.22) ;
\draw [shift={(175.35,176.65)}, rotate = 181.33] [color={rgb, 255:red, 0; green, 0; blue, 0 }  ][line width=0.75]    (8.74,-3.92) .. controls (5.56,-1.84) and (2.65,-0.53) .. (0,0) .. controls (2.65,0.53) and (5.56,1.84) .. (8.74,3.92)   ;
\draw    (199.92,177.22) -- (257.2,177.15) ;
\draw [shift={(233.36,177.18)}, rotate = 179.93] [color={rgb, 255:red, 0; green, 0; blue, 0 }  ][line width=0.75]    (8.74,-3.92) .. controls (5.56,-1.84) and (2.65,-0.53) .. (0,0) .. controls (2.65,0.53) and (5.56,1.84) .. (8.74,3.92)   ;
\draw    (257.2,177.15) .. controls (258.87,175.48) and (260.53,175.48) .. (262.2,177.15) .. controls (263.87,178.82) and (265.53,178.82) .. (267.2,177.15) .. controls (268.87,175.48) and (270.53,175.48) .. (272.2,177.15) .. controls (273.87,178.82) and (275.53,178.82) .. (277.2,177.15) .. controls (278.87,175.48) and (280.53,175.48) .. (282.2,177.15) .. controls (283.87,178.82) and (285.53,178.82) .. (287.2,177.15) .. controls (288.87,175.48) and (290.53,175.48) .. (292.2,177.15) .. controls (293.87,178.82) and (295.53,178.82) .. (297.2,177.15) .. controls (298.87,175.48) and (300.53,175.48) .. (302.2,177.15) .. controls (303.87,178.82) and (305.53,178.82) .. (307.2,177.15) .. controls (308.87,175.48) and (310.53,175.48) .. (312.2,177.15) -- (314.76,177.15) -- (314.76,177.15) ;
\draw [shift={(290.78,177.15)}, rotate = 180] [color={rgb, 255:red, 0; green, 0; blue, 0 }  ][line width=0.75]    (8.74,-3.92) .. controls (5.56,-1.84) and (2.65,-0.53) .. (0,0) .. controls (2.65,0.53) and (5.56,1.84) .. (8.74,3.92)   ;
\draw    (314.76,177.15) -- (373.05,177.15) ;
\draw [shift={(348.71,177.15)}, rotate = 180] [color={rgb, 255:red, 0; green, 0; blue, 0 }  ][line width=0.75]    (8.74,-3.92) .. controls (5.56,-1.84) and (2.65,-0.53) .. (0,0) .. controls (2.65,0.53) and (5.56,1.84) .. (8.74,3.92)   ;
\draw    (401.83,177.15) .. controls (403.5,175.48) and (405.16,175.48) .. (406.83,177.15) .. controls (408.5,178.82) and (410.16,178.82) .. (411.83,177.15) .. controls (413.5,175.48) and (415.16,175.48) .. (416.83,177.15) .. controls (418.5,178.82) and (420.16,178.82) .. (421.83,177.15) .. controls (423.5,175.48) and (425.16,175.48) .. (426.83,177.15) .. controls (428.5,178.82) and (430.16,178.82) .. (431.83,177.15) .. controls (433.5,175.48) and (435.16,175.48) .. (436.83,177.15) .. controls (438.5,178.82) and (440.16,178.82) .. (441.83,177.15) .. controls (443.5,175.48) and (445.16,175.48) .. (446.83,177.15) .. controls (448.5,178.82) and (450.16,178.82) .. (451.83,177.15) .. controls (453.5,175.48) and (455.16,175.48) .. (456.83,177.15) -- (457.24,177.15) -- (457.24,177.15) ;
\draw [shift={(434.34,177.15)}, rotate = 180] [color={rgb, 255:red, 0; green, 0; blue, 0 }  ][line width=0.75]    (8.74,-3.92) .. controls (5.56,-1.84) and (2.65,-0.53) .. (0,0) .. controls (2.65,0.53) and (5.56,1.84) .. (8.74,3.92)   ;
\draw  [dash pattern={on 0.84pt off 2.51pt}]  (373.05,177.15) -- (401.83,177.15) ;
\draw    (114,140.62) -- (141.2,175.86) ;
\draw [shift={(130.53,162.04)}, rotate = 232.34] [color={rgb, 255:red, 0; green, 0; blue, 0 }  ][line width=0.75]    (8.74,-3.92) .. controls (5.56,-1.84) and (2.65,-0.53) .. (0,0) .. controls (2.65,0.53) and (5.56,1.84) .. (8.74,3.92)   ;
\draw    (114,213.69) -- (141.2,175.86) ;
\draw [shift={(130.4,190.88)}, rotate = 125.72] [color={rgb, 255:red, 0; green, 0; blue, 0 }  ][line width=0.75]    (8.74,-3.92) .. controls (5.56,-1.84) and (2.65,-0.53) .. (0,0) .. controls (2.65,0.53) and (5.56,1.84) .. (8.74,3.92)   ;
\draw    (457.24,177.15) -- (487.47,213.69) ;
\draw [shift={(475.41,199.12)}, rotate = 230.4] [color={rgb, 255:red, 0; green, 0; blue, 0 }  ][line width=0.75]    (8.74,-3.92) .. controls (5.56,-1.84) and (2.65,-0.53) .. (0,0) .. controls (2.65,0.53) and (5.56,1.84) .. (8.74,3.92)   ;
\draw    (457.24,177.15) -- (488.3,144.47) ;
\draw [shift={(476.08,157.33)}, rotate = 133.54] [color={rgb, 255:red, 0; green, 0; blue, 0 }  ][line width=0.75]    (8.74,-3.92) .. controls (5.56,-1.84) and (2.65,-0.53) .. (0,0) .. controls (2.65,0.53) and (5.56,1.84) .. (8.74,3.92)   ;
\draw    (70.82,139.71) .. controls (72.52,138.08) and (74.19,138.11) .. (75.82,139.81) .. controls (77.45,141.51) and (79.12,141.55) .. (80.82,139.92) .. controls (82.52,138.29) and (84.19,138.33) .. (85.82,140.03) .. controls (87.45,141.73) and (89.12,141.76) .. (90.82,140.13) .. controls (92.52,138.5) and (94.18,138.54) .. (95.81,140.24) .. controls (97.44,141.94) and (99.11,141.97) .. (100.81,140.34) .. controls (102.51,138.71) and (104.18,138.75) .. (105.81,140.45) .. controls (107.44,142.15) and (109.11,142.18) .. (110.81,140.55) -- (114,140.62) -- (114,140.62) ;
\draw [shift={(97.21,140.27)}, rotate = 181.21] [color={rgb, 255:red, 0; green, 0; blue, 0 }  ][line width=0.75]    (8.74,-3.92) .. controls (5.56,-1.84) and (2.65,-0.53) .. (0,0) .. controls (2.65,0.53) and (5.56,1.84) .. (8.74,3.92)   ;
\draw    (70.1,214.6) .. controls (71.73,212.9) and (73.4,212.87) .. (75.1,214.5) .. controls (76.8,216.13) and (78.47,216.09) .. (80.1,214.39) .. controls (81.73,212.69) and (83.4,212.66) .. (85.1,214.29) .. controls (86.8,215.92) and (88.47,215.88) .. (90.1,214.18) .. controls (91.73,212.48) and (93.39,212.45) .. (95.09,214.08) .. controls (96.79,215.71) and (98.46,215.68) .. (100.09,213.98) .. controls (101.72,212.28) and (103.39,212.24) .. (105.09,213.87) .. controls (106.79,215.5) and (108.46,215.47) .. (110.09,213.77) -- (114,213.69) -- (114,213.69) ;
\draw [shift={(96.85,214.04)}, rotate = 178.81] [color={rgb, 255:red, 0; green, 0; blue, 0 }  ][line width=0.75]    (8.74,-3.92) .. controls (5.56,-1.84) and (2.65,-0.53) .. (0,0) .. controls (2.65,0.53) and (5.56,1.84) .. (8.74,3.92)   ;
\draw    (488.3,144.47) .. controls (489.93,142.76) and (491.6,142.73) .. (493.3,144.36) .. controls (495,145.99) and (496.67,145.96) .. (498.3,144.26) .. controls (499.93,142.56) and (501.6,142.52) .. (503.3,144.15) .. controls (505,145.78) and (506.67,145.75) .. (508.3,144.05) .. controls (509.93,142.35) and (511.6,142.32) .. (513.3,143.95) .. controls (515,145.58) and (516.67,145.54) .. (518.3,143.84) .. controls (519.93,142.14) and (521.6,142.11) .. (523.3,143.74) .. controls (525,145.37) and (526.67,145.33) .. (528.3,143.63) -- (532.2,143.55) -- (532.2,143.55) ;
\draw [shift={(515.05,143.91)}, rotate = 178.81] [color={rgb, 255:red, 0; green, 0; blue, 0 }  ][line width=0.75]    (8.74,-3.92) .. controls (5.56,-1.84) and (2.65,-0.53) .. (0,0) .. controls (2.65,0.53) and (5.56,1.84) .. (8.74,3.92)   ;
\draw    (487.47,213.69) .. controls (489.17,212.06) and (490.84,212.09) .. (492.47,213.79) .. controls (494.1,215.49) and (495.76,215.53) .. (497.46,213.9) .. controls (499.16,212.27) and (500.83,212.3) .. (502.46,214) .. controls (504.09,215.7) and (505.76,215.74) .. (507.46,214.11) .. controls (509.16,212.48) and (510.83,212.52) .. (512.46,214.22) .. controls (514.09,215.92) and (515.76,215.95) .. (517.46,214.32) .. controls (519.16,212.69) and (520.83,212.73) .. (522.46,214.43) .. controls (524.09,216.13) and (525.76,216.16) .. (527.46,214.53) -- (530.64,214.6) -- (530.64,214.6) ;
\draw [shift={(513.85,214.24)}, rotate = 181.21] [color={rgb, 255:red, 0; green, 0; blue, 0 }  ][line width=0.75]    (8.74,-3.92) .. controls (5.56,-1.84) and (2.65,-0.53) .. (0,0) .. controls (2.65,0.53) and (5.56,1.84) .. (8.74,3.92)   ;

\end{tikzpicture}

%% file: diagram-typeD.tex
\tikzset{every picture/.style={line width=0.75pt}} 

\begin{tikzpicture}[x=0.75pt,y=0.75pt,yscale=-1,xscale=1]

\draw    (148.25,131.34) .. controls (149.96,129.71) and (151.62,129.75) .. (153.25,131.46) .. controls (154.88,133.17) and (156.54,133.21) .. (158.25,131.59) .. controls (159.96,129.97) and (161.62,130.01) .. (163.25,131.72) .. controls (164.88,133.43) and (166.54,133.47) .. (168.25,131.85) .. controls (169.96,130.23) and (171.62,130.27) .. (173.25,131.98) .. controls (174.87,133.69) and (176.53,133.73) .. (178.24,132.1) .. controls (179.95,130.48) and (181.61,130.52) .. (183.24,132.23) .. controls (184.87,133.94) and (186.53,133.98) .. (188.24,132.36) .. controls (189.95,130.74) and (191.61,130.78) .. (193.24,132.49) .. controls (194.87,134.2) and (196.53,134.24) .. (198.24,132.62) -- (201.18,132.69) -- (201.18,132.69) ;
\draw    (201.18,132.69) -- (252.8,132.63) ;
\draw [shift={(231.79,132.65)}, rotate = 179.93] [color={rgb, 255:red, 0; green, 0; blue, 0 }  ][line width=0.75]    (8.74,-3.92) .. controls (5.56,-1.84) and (2.65,-0.53) .. (0,0) .. controls (2.65,0.53) and (5.56,1.84) .. (8.74,3.92)   ;
\draw    (339.45,132.52) -- (380.56,132.63) ;
\draw [shift={(364.8,132.58)}, rotate = 180.15] [color={rgb, 255:red, 0; green, 0; blue, 0 }  ][line width=0.75]    (8.74,-3.92) .. controls (5.56,-1.84) and (2.65,-0.53) .. (0,0) .. controls (2.65,0.53) and (5.56,1.84) .. (8.74,3.92)   ;
\draw    (406.5,132.63) .. controls (408.21,131) and (409.88,131.05) .. (411.5,132.76) .. controls (413.12,134.47) and (414.79,134.52) .. (416.5,132.9) .. controls (418.21,131.28) and (419.87,131.32) .. (421.5,133.03) .. controls (423.13,134.74) and (424.79,134.78) .. (426.5,133.16) .. controls (428.21,131.54) and (429.88,131.59) .. (431.5,133.3) .. controls (433.12,135.01) and (434.78,135.05) .. (436.49,133.43) .. controls (438.2,131.81) and (439.87,131.86) .. (441.49,133.57) .. controls (443.12,135.28) and (444.78,135.32) .. (446.49,133.7) .. controls (448.2,132.08) and (449.87,132.13) .. (451.49,133.84) .. controls (453.12,135.55) and (454.78,135.59) .. (456.49,133.97) .. controls (458.2,132.35) and (459.87,132.4) .. (461.48,134.11) .. controls (463.11,135.82) and (464.77,135.86) .. (466.48,134.24) .. controls (468.19,132.62) and (469.86,132.67) .. (471.48,134.38) -- (473.95,134.45) -- (473.95,134.45) ;
\draw [shift={(440.23,133.54)}, rotate = 46.55] [color={rgb, 255:red, 0; green, 0; blue, 0 }  ][line width=0.75]    (-4.47,0) -- (4.47,0)(0,4.47) -- (0,-4.47)   ;
\draw  [dash pattern={on 0.84pt off 2.51pt}]  (380.56,132.63) -- (406.5,132.63) ;
\draw    (123.74,96.23) -- (148.25,131.34) ;
\draw [shift={(138.75,117.72)}, rotate = 235.08] [color={rgb, 255:red, 0; green, 0; blue, 0 }  ][line width=0.75]    (8.74,-3.92) .. controls (5.56,-1.84) and (2.65,-0.53) .. (0,0) .. controls (2.65,0.53) and (5.56,1.84) .. (8.74,3.92)   ;
\draw    (123.74,169.02) -- (148.25,131.34) ;
\draw [shift={(138.61,146.16)}, rotate = 123.04] [color={rgb, 255:red, 0; green, 0; blue, 0 }  ][line width=0.75]    (8.74,-3.92) .. controls (5.56,-1.84) and (2.65,-0.53) .. (0,0) .. controls (2.65,0.53) and (5.56,1.84) .. (8.74,3.92)   ;
\draw    (473.95,134.45) -- (501.19,170.84) ;
\draw [shift={(490.45,156.49)}, rotate = 233.19] [color={rgb, 255:red, 0; green, 0; blue, 0 }  ][line width=0.75]    (8.74,-3.92) .. controls (5.56,-1.84) and (2.65,-0.53) .. (0,0) .. controls (2.65,0.53) and (5.56,1.84) .. (8.74,3.92)   ;
\draw    (473.95,134.45) -- (501.95,101.88) ;
\draw [shift={(491.08,114.52)}, rotate = 130.68] [color={rgb, 255:red, 0; green, 0; blue, 0 }  ][line width=0.75]    (8.74,-3.92) .. controls (5.56,-1.84) and (2.65,-0.53) .. (0,0) .. controls (2.65,0.53) and (5.56,1.84) .. (8.74,3.92)   ;
\draw    (84.06,96.7) .. controls (85.71,95.01) and (87.37,94.99) .. (89.06,96.64) .. controls (90.75,98.29) and (92.41,98.27) .. (94.06,96.58) .. controls (95.71,94.89) and (97.37,94.87) .. (99.06,96.52) .. controls (100.75,98.17) and (102.41,98.15) .. (104.06,96.46) .. controls (105.71,94.77) and (107.37,94.75) .. (109.06,96.4) .. controls (110.75,98.05) and (112.41,98.03) .. (114.06,96.34) .. controls (115.71,94.65) and (117.37,94.63) .. (119.06,96.28) -- (123.74,96.23) -- (123.74,96.23) ;
\draw    (84.18,169.93) .. controls (85.81,168.23) and (87.47,168.19) .. (89.18,169.82) .. controls (90.89,171.45) and (92.55,171.41) .. (94.18,169.7) .. controls (95.81,168) and (97.47,167.96) .. (99.17,169.59) .. controls (100.88,171.22) and (102.54,171.18) .. (104.17,169.47) .. controls (105.8,167.77) and (107.47,167.73) .. (109.17,169.36) .. controls (110.88,170.99) and (112.54,170.95) .. (114.17,169.24) .. controls (115.8,167.54) and (117.47,167.5) .. (119.17,169.13) -- (123.74,169.02) -- (123.74,169.02) ;
\draw [shift={(103.96,169.48)}, rotate = 43.68] [color={rgb, 255:red, 0; green, 0; blue, 0 }  ][line width=0.75]    (-4.47,0) -- (4.47,0)(0,4.47) -- (0,-4.47)   ;
\draw    (501.19,170.84) .. controls (502.87,169.19) and (504.54,169.21) .. (506.19,170.89) .. controls (507.84,172.57) and (509.51,172.59) .. (511.19,170.94) .. controls (512.87,169.29) and (514.54,169.31) .. (516.19,170.99) .. controls (517.84,172.67) and (519.51,172.69) .. (521.19,171.04) .. controls (522.87,169.39) and (524.54,169.41) .. (526.19,171.09) .. controls (527.84,172.77) and (529.51,172.79) .. (531.19,171.14) .. controls (532.87,169.49) and (534.54,169.51) .. (536.19,171.19) .. controls (537.84,172.87) and (539.51,172.89) .. (541.19,171.24) -- (542.2,171.25) -- (542.2,171.25) ;
\draw [shift={(521.7,171.05)}, rotate = 45.57] [color={rgb, 255:red, 0; green, 0; blue, 0 }  ][line width=0.75]    (-4.47,0) -- (4.47,0)(0,4.47) -- (0,-4.47)   ;
\draw    (252.8,132.63) .. controls (254.43,130.93) and (256.1,130.9) .. (257.8,132.53) .. controls (259.49,134.16) and (261.16,134.13) .. (262.8,132.44) .. controls (264.44,130.75) and (266.11,130.72) .. (267.8,132.35) .. controls (269.5,133.98) and (271.17,133.95) .. (272.8,132.25) -- (275.98,132.19) -- (275.98,132.19) ;
\draw [shift={(275.98,132.19)}, rotate = 43.93] [color={rgb, 255:red, 0; green, 0; blue, 0 }  ][line width=0.75]    (-4.47,0) -- (4.47,0)(0,4.47) -- (0,-4.47)   ;
\draw    (275.98,132.19) .. controls (277.66,130.54) and (279.33,130.55) .. (280.98,132.22) .. controls (282.64,133.89) and (284.31,133.9) .. (285.98,132.25) .. controls (287.65,130.59) and (289.32,130.6) .. (290.98,132.27) .. controls (292.64,133.94) and (294.31,133.95) .. (295.98,132.3) .. controls (297.65,130.64) and (299.32,130.65) .. (300.98,132.32) .. controls (302.64,133.99) and (304.31,134) .. (305.98,132.35) .. controls (307.65,130.69) and (309.32,130.7) .. (310.98,132.37) .. controls (312.64,134.04) and (314.31,134.05) .. (315.98,132.4) .. controls (317.65,130.74) and (319.32,130.75) .. (320.98,132.42) .. controls (322.64,134.09) and (324.31,134.1) .. (325.98,132.45) .. controls (327.65,130.8) and (329.32,130.81) .. (330.98,132.48) .. controls (332.64,134.15) and (334.31,134.16) .. (335.98,132.5) -- (339.45,132.52) -- (339.45,132.52) ;
\draw [shift={(307.71,132.36)}, rotate = 45.29] [color={rgb, 255:red, 0; green, 0; blue, 0 }  ][line width=0.75]    (-4.47,0) -- (4.47,0)(0,4.47) -- (0,-4.47)   ;
\draw    (501.95,101.88) .. controls (503.6,100.19) and (505.26,100.17) .. (506.95,101.82) .. controls (508.64,103.47) and (510.3,103.45) .. (511.95,101.76) .. controls (513.6,100.07) and (515.26,100.05) .. (516.95,101.7) .. controls (518.64,103.35) and (520.3,103.33) .. (521.95,101.64) .. controls (523.6,99.95) and (525.26,99.93) .. (526.95,101.58) .. controls (528.64,103.23) and (530.3,103.21) .. (531.95,101.52) .. controls (533.59,99.83) and (535.25,99.81) .. (536.94,101.46) -- (541.62,101.41) -- (541.62,101.41) ;

\draw (95.33,77.54) node [anchor=north west][inner sep=0.75pt]  [font=\small]  {$v_{0}$};
\draw (86.18,173.33) node [anchor=north west][inner sep=0.75pt]  [font=\small]  {$v_{1}^{0} \ \ \ v_{1}^{1}$};
\draw (514.36,78.33) node [anchor=north west][inner sep=0.75pt]  [font=\small]  {$v_{n-1}^{0}$};
\draw (503.19,174.24) node [anchor=north west][inner sep=0.75pt]  [font=\small]  {$v_{n}^{0} \ \ v_{n}^{1}$};
\draw (164,103) node [anchor=north west][inner sep=0.75pt]    {$v_{2} \ \ \ \ \ \ \ \ \ \ \ \ \ \ \ \ \ v_{3}^{0} \ \ \ \ v_{3}^{1} \ \ \ \ \ v{_{3}^{2}} \ \ \ \ \ \ \ \ \ \ \ \ \ \ \ \ v_{n-2}^{0} \ \ \ v_{n-2}^{1}$};

\end{tikzpicture}

%% file: Adjseg2.tex
\begin{tikzpicture}[x=0.5pt,y=0.5pt,yscale=-1,xscale=1]

\draw    (190.2,129.93) .. controls (191.87,128.26) and (193.53,128.26) .. (195.2,129.93) .. controls (196.87,131.6) and (198.53,131.6) .. (200.2,129.93) .. controls (201.87,128.26) and (203.53,128.26) .. (205.2,129.93) .. controls (206.87,131.6) and (208.53,131.6) .. (210.2,129.93) .. controls (211.87,128.26) and (213.53,128.26) .. (215.2,129.93) .. controls (216.87,131.6) and (218.53,131.6) .. (220.2,129.93) .. controls (221.87,128.26) and (223.53,128.26) .. (225.2,129.93) .. controls (226.87,131.6) and (228.53,131.6) .. (230.2,129.93) .. controls (231.87,128.26) and (233.53,128.26) .. (235.2,129.93) .. controls (236.87,131.6) and (238.53,131.6) .. (240.2,129.93) .. controls (241.87,128.26) and (243.53,128.26) .. (245.2,129.93) .. controls (246.87,131.6) and (248.53,131.6) .. (250.2,129.93) .. controls (251.87,128.26) and (253.53,128.26) .. (255.2,129.93) .. controls (256.87,131.6) and (258.53,131.6) .. (260.2,129.93) .. controls (261.87,128.26) and (263.53,128.26) .. (265.2,129.93) .. controls (266.87,131.6) and (268.53,131.6) .. (270.2,129.93) .. controls (271.87,128.26) and (273.53,128.26) .. (275.2,129.93) .. controls (276.87,131.6) and (278.53,131.6) .. (280.2,129.93) -- (280.67,129.93) -- (288.67,129.93) ;
\draw [shift={(290.67,129.93)}, rotate = 180] [color={rgb, 255:red, 0; green, 0; blue, 0 }  ][line width=0.75]    (9.84,-4.41) .. controls (6.25,-2.07) and (2.97,-0.6) .. (0,0) .. controls (2.97,0.6) and (6.25,2.07) .. (9.84,4.41)   ;
\draw [shift={(240.43,129.93)}, rotate = 45] [color={rgb, 255:red, 0; green, 0; blue, 0 }  ][line width=0.75]    (-5.03,0) -- (5.03,0)(0,5.03) -- (0,-5.03)   ;

\draw (203.5,101) node [anchor=north west][inner sep=0.75pt]  [font=\small]  {$\zeta ^{-} \  x\  \ \zeta ^{+}$};

\end{tikzpicture}

%% file: Tri-x.tex
\begin{tikzpicture}[x=0.75pt,y=0.75pt,yscale=-1,xscale=1]

\draw    (265.2,150.6) -- (301.2,150.6) ;
\draw [shift={(304.2,150.6)}, rotate = 180] [fill={rgb, 255:red, 0; green, 0; blue, 0 }  ][line width=0.08]  [draw opacity=0] (5.36,-2.57) -- (0,0) -- (5.36,2.57) -- (3.56,0) -- cycle    ;
\draw    (258.2,159.05) -- (274.19,176.82) ;
\draw [shift={(276.2,179.05)}, rotate = 228.01] [fill={rgb, 255:red, 0; green, 0; blue, 0 }  ][line width=0.08]  [draw opacity=0] (5.36,-2.57) -- (0,0) -- (5.36,2.57) -- (3.56,0) -- cycle    ;
\draw    (309,162.09) -- (290.2,179.6) ;
\draw [shift={(311.2,160.05)}, rotate = 137.05] [fill={rgb, 255:red, 0; green, 0; blue, 0 }  ][line width=0.08]  [draw opacity=0] (5.36,-2.57) -- (0,0) -- (5.36,2.57) -- (3.56,0) -- cycle    ;
\draw    (310.42,137.53) .. controls (296.03,119.65) and (324.9,114.36) .. (320.58,132.68) ;
\draw [shift={(319.74,135.43)}, rotate = 290.33] [fill={rgb, 255:red, 0; green, 0; blue, 0 }  ][line width=0.08]  [draw opacity=0] (5.36,-2.57) -- (0,0) -- (5.36,2.57) -- (3.56,0) -- cycle    ;
\draw    (241.42,135.53) .. controls (227.03,117.65) and (255.9,112.36) .. (251.58,130.68) ;
\draw [shift={(250.74,133.43)}, rotate = 290.33] [fill={rgb, 255:red, 0; green, 0; blue, 0 }  ][line width=0.08]  [draw opacity=0] (5.36,-2.57) -- (0,0) -- (5.36,2.57) -- (3.56,0) -- cycle    ;

\draw (237.67,137.67) node [anchor=north west][inner sep=0.75pt]   [align=left] {$\displaystyle \mathbb{C}^{v_{x}^{-}}$};
\draw (306.33,138) node [anchor=north west][inner sep=0.75pt]   [align=left] {$\displaystyle \mathbb{C}^{v_{x}^{+}}$};
\draw (307.83,105.67) node [anchor=north west][inner sep=0.75pt]  [font=\scriptsize] [align=left] {$\displaystyle B_{x}^{+}$};
\draw (277.2,179.05) node [anchor=north west][inner sep=0.75pt]   [align=left] {$\displaystyle \mathbb{C}$};
\draw (239,102.07) node [anchor=north west][inner sep=0.75pt]  [font=\scriptsize]  {$B_{x}^{-}$};
\draw (280,136.6) node [anchor=north west][inner sep=0.75pt]  [font=\scriptsize]  {$A$};
\draw (253.8,172.47) node [anchor=north west][inner sep=0.75pt]  [font=\scriptsize]  {$b$};
\draw (309.8,172.2) node [anchor=north west][inner sep=0.75pt]  [font=\scriptsize]  {$a$};

\end{tikzpicture}

%% file: PictureQF.tex
\begin{tikzpicture}[x=0.75pt,y=0.75pt,yscale=-1,xscale=1]

\draw   (133,100.1) .. controls (133,96.18) and (136.18,93) .. (140.1,93) .. controls (144.02,93) and (147.2,96.18) .. (147.2,100.1) .. controls (147.2,104.02) and (144.02,107.2) .. (140.1,107.2) .. controls (136.18,107.2) and (133,104.02) .. (133,100.1) -- cycle ;
\draw   (183,100.1) .. controls (183,96.18) and (186.18,93) .. (190.1,93) .. controls (194.02,93) and (197.2,96.18) .. (197.2,100.1) .. controls (197.2,104.02) and (194.02,107.2) .. (190.1,107.2) .. controls (186.18,107.2) and (183,104.02) .. (183,100.1) -- cycle ;
\draw    (147.2,100.1) -- (181,100.1) ;
\draw [shift={(183,100.1)}, rotate = 180] [color={rgb, 255:red, 0; green, 0; blue, 0 }  ][line width=0.75]    (7.65,-3.43) .. controls (4.86,-1.61) and (2.31,-0.47) .. (0,0) .. controls (2.31,0.47) and (4.86,1.61) .. (7.65,3.43)   ;
\draw    (133,100.1) .. controls (100.86,89.81) and (139.59,62.71) .. (140.12,91.19) ;
\draw [shift={(140.1,93)}, rotate = 272.01] [color={rgb, 255:red, 0; green, 0; blue, 0 }  ][line width=0.75]    (7.65,-3.43) .. controls (4.86,-1.61) and (2.31,-0.47) .. (0,0) .. controls (2.31,0.47) and (4.86,1.61) .. (7.65,3.43)   ;
\draw   (298,100.1) .. controls (298,96.18) and (301.18,93) .. (305.1,93) .. controls (309.02,93) and (312.2,96.18) .. (312.2,100.1) .. controls (312.2,104.02) and (309.02,107.2) .. (305.1,107.2) .. controls (301.18,107.2) and (298,104.02) .. (298,100.1) -- cycle ;
\draw   (348,100.1) .. controls (348,96.18) and (351.18,93) .. (355.1,93) .. controls (359.02,93) and (362.2,96.18) .. (362.2,100.1) .. controls (362.2,104.02) and (359.02,107.2) .. (355.1,107.2) .. controls (351.18,107.2) and (348,104.02) .. (348,100.1) -- cycle ;
\draw    (312.2,100.1) -- (346,100.1) ;
\draw [shift={(348,100.1)}, rotate = 180] [color={rgb, 255:red, 0; green, 0; blue, 0 }  ][line width=0.75]    (7.65,-3.43) .. controls (4.86,-1.61) and (2.31,-0.47) .. (0,0) .. controls (2.31,0.47) and (4.86,1.61) .. (7.65,3.43)   ;
\draw    (298,100.1) .. controls (265.86,89.81) and (304.59,62.71) .. (305.12,91.19) ;
\draw [shift={(305.1,93)}, rotate = 272.01] [color={rgb, 255:red, 0; green, 0; blue, 0 }  ][line width=0.75]    (7.65,-3.43) .. controls (4.86,-1.61) and (2.31,-0.47) .. (0,0) .. controls (2.31,0.47) and (4.86,1.61) .. (7.65,3.43)   ;
\draw   (299,134) -- (311.2,134) -- (311.2,146.2) -- (299,146.2) -- cycle ;
\draw   (349,134) -- (361.2,134) -- (361.2,146.2) -- (349,146.2) -- cycle ;
\draw    (305.2,133.6) -- (305.11,109.2) ;
\draw [shift={(305.1,107.2)}, rotate = 89.78] [color={rgb, 255:red, 0; green, 0; blue, 0 }  ][line width=0.75]    (7.65,-2.3) .. controls (4.86,-0.97) and (2.31,-0.21) .. (0,0) .. controls (2.31,0.21) and (4.86,0.98) .. (7.65,2.3)   ;
\draw    (355.2,134.6) -- (355.11,109.2) ;
\draw [shift={(355.1,107.2)}, rotate = 89.79] [color={rgb, 255:red, 0; green, 0; blue, 0 }  ][line width=0.75]    (7.65,-2.3) .. controls (4.86,-0.97) and (2.31,-0.21) .. (0,0) .. controls (2.31,0.21) and (4.86,0.98) .. (7.65,2.3)   ;

\draw (100,74.4) node [anchor=north west][inner sep=0.75pt]    {$Q$};
\draw (258,72.4) node [anchor=north west][inner sep=0.75pt]    {$Q^{F}$};

\end{tikzpicture}

%% file: Jordan.tex
\begin{tikzpicture}[x=0.75pt,y=0.75pt,yscale=-1,xscale=1]

\draw    (101.2,100.8) .. controls (102.87,99.13) and (104.53,99.13) .. (106.2,100.8) .. controls (107.87,102.47) and (109.53,102.47) .. (111.2,100.8) .. controls (112.87,99.13) and (114.53,99.13) .. (116.2,100.8) .. controls (117.87,102.47) and (119.53,102.47) .. (121.2,100.8) .. controls (122.87,99.13) and (124.53,99.13) .. (126.2,100.8) .. controls (127.87,102.47) and (129.53,102.47) .. (131.2,100.8) .. controls (132.87,99.13) and (134.53,99.13) .. (136.2,100.8) .. controls (137.87,102.47) and (139.53,102.47) .. (141.2,100.8) .. controls (142.87,99.13) and (144.53,99.13) .. (146.2,100.8) .. controls (147.87,102.47) and (149.53,102.47) .. (151.2,100.8) .. controls (152.87,99.13) and (154.53,99.13) .. (156.2,100.8) .. controls (157.87,102.47) and (159.53,102.47) .. (161.2,100.8) .. controls (162.87,99.13) and (164.53,99.13) .. (166.2,100.8) .. controls (167.87,102.47) and (169.53,102.47) .. (171.2,100.8) .. controls (172.87,99.13) and (174.53,99.13) .. (176.2,100.8) .. controls (177.87,102.47) and (179.53,102.47) .. (181.2,100.8) .. controls (182.87,99.13) and (184.53,99.13) .. (186.2,100.8) -- (189.2,100.8) -- (189.2,100.8) ;
\draw [shift={(148.8,100.8)}, rotate = 180] [color={rgb, 255:red, 0; green, 0; blue, 0 }  ][line width=0.75]    (6.56,-2.94) .. controls (4.17,-1.38) and (1.99,-0.4) .. (0,0) .. controls (1.99,0.4) and (4.17,1.38) .. (6.56,2.94)   ;
\draw    (101.2,100.8) .. controls (102.2,54.6) and (190.2,55.6) .. (189.2,100.8) ;
\draw [shift={(141.33,66.67)}, rotate = 358.84] [color={rgb, 255:red, 0; green, 0; blue, 0 }  ][line width=0.75]    (6.56,-2.94) .. controls (4.17,-1.38) and (1.99,-0.4) .. (0,0) .. controls (1.99,0.4) and (4.17,1.38) .. (6.56,2.94)   ;

\end{tikzpicture}

%% file: Affine-A-Bows.tex
\begin{tikzpicture}[x=0.75pt,y=0.75pt,yscale=-1,xscale=1]

\draw    (165.86,125.27) -- (182.14,164.58) ;
\draw [shift={(175.38,148.25)}, rotate = 247.5] [color={rgb, 255:red, 0; green, 0; blue, 0 }  ][line width=0.75]    (6.56,-2.94) .. controls (4.17,-1.38) and (1.99,-0.4) .. (0,0) .. controls (1.99,0.4) and (4.17,1.38) .. (6.56,2.94)   ;
\draw    (221.46,69.67) -- (182.14,85.95) ;
\draw [shift={(198.47,79.19)}, rotate = 337.5] [color={rgb, 255:red, 0; green, 0; blue, 0 }  ][line width=0.75]    (6.56,-2.94) .. controls (4.17,-1.38) and (1.99,-0.4) .. (0,0) .. controls (1.99,0.4) and (4.17,1.38) .. (6.56,2.94)   ;
\draw    (221.46,180.87) -- (260.77,164.58) ;
\draw [shift={(244.44,171.35)}, rotate = 157.5] [color={rgb, 255:red, 0; green, 0; blue, 0 }  ][line width=0.75]    (6.56,-2.94) .. controls (4.17,-1.38) and (1.99,-0.4) .. (0,0) .. controls (1.99,0.4) and (4.17,1.38) .. (6.56,2.94)   ;
\draw    (277.06,125.27) -- (260.77,85.95) ;
\draw [shift={(267.54,102.28)}, rotate = 67.5] [color={rgb, 255:red, 0; green, 0; blue, 0 }  ][line width=0.75]    (6.56,-2.94) .. controls (4.17,-1.38) and (1.99,-0.4) .. (0,0) .. controls (1.99,0.4) and (4.17,1.38) .. (6.56,2.94)   ;
\draw    (182.14,85.95) .. controls (183.05,88.13) and (182.41,89.67) .. (180.23,90.57) .. controls (178.05,91.47) and (177.41,93.01) .. (178.31,95.19) .. controls (179.22,97.37) and (178.58,98.91) .. (176.4,99.81) .. controls (174.22,100.71) and (173.58,102.25) .. (174.49,104.43) .. controls (175.39,106.61) and (174.75,108.15) .. (172.57,109.05) .. controls (170.39,109.95) and (169.75,111.49) .. (170.66,113.67) .. controls (171.57,115.85) and (170.93,117.39) .. (168.75,118.29) .. controls (166.57,119.19) and (165.93,120.73) .. (166.83,122.91) -- (165.86,125.27) -- (165.86,125.27) ;
\draw [shift={(172.62,108.94)}, rotate = 292.5] [color={rgb, 255:red, 0; green, 0; blue, 0 }  ][line width=0.75]    (6.56,-2.94) .. controls (4.17,-1.38) and (1.99,-0.4) .. (0,0) .. controls (1.99,0.4) and (4.17,1.38) .. (6.56,2.94)   ;
\draw    (260.77,85.95) .. controls (258.59,86.86) and (257.05,86.22) .. (256.15,84.04) .. controls (255.25,81.86) and (253.71,81.22) .. (251.53,82.12) .. controls (249.35,83.03) and (247.81,82.39) .. (246.91,80.21) .. controls (246.01,78.03) and (244.47,77.39) .. (242.29,78.3) .. controls (240.11,79.2) and (238.57,78.56) .. (237.67,76.38) .. controls (236.77,74.2) and (235.23,73.56) .. (233.05,74.47) .. controls (230.87,75.38) and (229.33,74.74) .. (228.43,72.56) .. controls (227.54,70.38) and (226,69.74) .. (223.82,70.64) -- (221.46,69.67) -- (221.46,69.67) ;
\draw [shift={(237.79,76.43)}, rotate = 22.5] [color={rgb, 255:red, 0; green, 0; blue, 0 }  ][line width=0.75]    (6.56,-2.94) .. controls (4.17,-1.38) and (1.99,-0.4) .. (0,0) .. controls (1.99,0.4) and (4.17,1.38) .. (6.56,2.94)   ;
\draw    (182.14,164.58) .. controls (184.32,163.68) and (185.86,164.32) .. (186.76,166.5) .. controls (187.66,168.68) and (189.2,169.32) .. (191.38,168.41) .. controls (193.56,167.5) and (195.1,168.14) .. (196,170.32) .. controls (196.9,172.5) and (198.44,173.14) .. (200.62,172.24) .. controls (202.8,171.33) and (204.34,171.97) .. (205.24,174.15) .. controls (206.14,176.33) and (207.68,176.97) .. (209.86,176.06) .. controls (212.04,175.16) and (213.58,175.8) .. (214.48,177.98) .. controls (215.38,180.16) and (216.92,180.8) .. (219.1,179.89) -- (221.46,180.87) -- (221.46,180.87) ;
\draw [shift={(205.12,174.1)}, rotate = 202.5] [color={rgb, 255:red, 0; green, 0; blue, 0 }  ][line width=0.75]    (6.56,-2.94) .. controls (4.17,-1.38) and (1.99,-0.4) .. (0,0) .. controls (1.99,0.4) and (4.17,1.38) .. (6.56,2.94)   ;
\draw  [dash pattern={on 0.84pt off 2.51pt}]  (260.77,164.58) -- (277.06,125.27) ;

\end{tikzpicture}

%% file: x-point.tex
\begin{tikzpicture}[x=0.75pt,y=0.75pt,yscale=-1,xscale=1]

\draw    (105.2,31) .. controls (106.87,29.33) and (108.53,29.33) .. (110.2,31) .. controls (111.87,32.67) and (113.53,32.67) .. (115.2,31) .. controls (116.87,29.33) and (118.53,29.33) .. (120.2,31) .. controls (121.87,32.67) and (123.53,32.67) .. (125.2,31) .. controls (126.87,29.33) and (128.53,29.33) .. (130.2,31) .. controls (131.87,32.67) and (133.53,32.67) .. (135.2,31) .. controls (136.87,29.33) and (138.53,29.33) .. (140.2,31) .. controls (141.87,32.67) and (143.53,32.67) .. (145.2,31) .. controls (146.87,29.33) and (148.53,29.33) .. (150.2,31) -- (151,31) -- (151,31) ;
\draw [shift={(128.1,31)}, rotate = 45] [color={rgb, 255:red, 0; green, 0; blue, 0 }  ][line width=0.75]    (-5.03,0) -- (5.03,0)(0,5.03) -- (0,-5.03)   ;
\end{tikzpicture}

%% file: segment.tex
\begin{tikzpicture}[x=0.75pt,y=0.75pt,yscale=-1,xscale=1]

\draw    (104.13,47) .. controls (105.88,45.41) and (107.54,45.49) .. (109.13,47.23) .. controls (110.72,48.97) and (112.38,49.05) .. (114.12,47.46) .. controls (115.86,45.87) and (117.53,45.95) .. (119.12,47.69) .. controls (120.71,49.43) and (122.37,49.5) .. (124.11,47.91) -- (124.53,47.93) -- (124.53,47.93) ;
\draw [shift={(124.53,47.93)}, rotate = 47.62] [color={rgb, 255:red, 0; green, 0; blue, 0 }  ][line width=0.75]    (-5.59,0) -- (5.59,0)(0,5.59) -- (0,-5.59)   ;
\draw    (124.53,47.93) .. controls (126.2,46.26) and (127.86,46.26) .. (129.53,47.93) .. controls (131.2,49.6) and (132.86,49.6) .. (134.53,47.93) .. controls (136.2,46.26) and (137.86,46.26) .. (139.53,47.93) -- (144.2,47.93) -- (144.2,47.93) ;
\draw [shift={(144.2,47.93)}, rotate = 45] [color={rgb, 255:red, 0; green, 0; blue, 0 }  ][line width=0.75]    (-5.59,0) -- (5.59,0)(0,5.59) -- (0,-5.59)   ;
\draw    (144.2,47.93) .. controls (145.79,46.19) and (147.45,46.11) .. (149.19,47.7) -- (154.14,47.47) -- (162.14,47.09) ;
\draw [shift={(164.13,47)}, rotate = 177.32] [color={rgb, 255:red, 0; green, 0; blue, 0 }  ][line width=0.75]    (6.56,-2.94) .. controls (4.17,-1.38) and (1.99,-0.4) .. (0,0) .. controls (1.99,0.4) and (4.17,1.38) .. (6.56,2.94)   ;

\draw (120,27.4) node [anchor=north west][inner sep=0.75pt]  [font=\scriptsize]  {$x\ \zeta \ x'$};
\end{tikzpicture}

%% file: GenTri2.tex
\begin{tikzcd}
\mathbb{C}^{v_1} \arrow["B_1", loop, distance=2em, in=55, out=125] \arrow[rr, "A"] \arrow[rd, "b"'] &                             & \mathbb{C}^{v_2} \arrow["B_2", loop, distance=2em, in=55, out=125] \\
                                                                                                    & \mathbb{C} \arrow[ru, "a"'] &                                                                   
\end{tikzcd}

%% file: HNF.tex
\setcounter{MaxMatrixCols}{9}
\begin{NiceMatrix}[columns-width=6mm]
 \NotEmpty & \NotEmpty & \NotEmpty & \NotEmpty & \NotEmpty \\
   \NotEmpty \Block{2-2}<\huge>{h} & \NotEmpty& \NotEmpty \Block{2-2}<\huge>{0} & \NotEmpty & \NotEmpty \Block{2-1}{g\ \ \ }
    & \NotEmpty \Block{2-1}{\quad m} \\
    \\
   \NotEmpty \Block{1-2}{f} & \NotEmpty &  \NotEmpty \Block{1-2}{0}  & \NotEmpty & \NotEmpty \Block{1-1}{e_0\ \ \ } & \NotEmpty \Block{1-1}{\quad 1} \\
   \NotEmpty \Block{2-2}<\huge>{0} & \NotEmpty &\NotEmpty \Block{2-2}<\huge>{\operatorname{id}} & \NotEmpty & \NotEmpty \Block{2-1}{e\ \ \ }  &\NotEmpty\Block{2-1}{\quad n-m-1}\\
    \\
    \CodeAfter
  \SubMatrix[{2-1}{6-5}]
  \OverBrace[shorten,yshift=3mm]{2-1}{2-2}{m}
  \OverBrace[shorten,yshift=3mm]{2-3}{2-4}{n-m-1}
  \OverBrace[shorten,yshift=3mm]{2-5}{2-5}{1}
    \SubMatrix{.}{2-1}{3-5}{\}}[xshift=4mm]
    \SubMatrix{.}{4-1}{4-5}{\}}[xshift=4mm]
    \SubMatrix{.}{5-1}{6-5}{\}}[xshift=4mm]
    \SubMatrix{.}{2-1}{6-2}{\vert}[xshift=4mm]
    \SubMatrix{.}{2-1}{6-4}{\vert}[xshift=4mm]
    \tikz\draw[shorten > = 0.6em, shorten < = 1em](4-|1) -- (4-|6);
    \tikz\draw[shorten > = 0.6em, shorten < = 1em](5-|1) -- (5-|6);
\end{NiceMatrix}

%% file: AdjSeg.tex
\tikzset{every picture/.style={line width=0.75pt}} 

\begin{tikzpicture}[x=0.75pt,y=0.75pt,yscale=-1,xscale=1]

\draw    (190.2,129.93) .. controls (191.87,128.26) and (193.53,128.26) .. (195.2,129.93) .. controls (196.87,131.6) and (198.53,131.6) .. (200.2,129.93) .. controls (201.87,128.26) and (203.53,128.26) .. (205.2,129.93) .. controls (206.87,131.6) and (208.53,131.6) .. (210.2,129.93) .. controls (211.87,128.26) and (213.53,128.26) .. (215.2,129.93) .. controls (216.87,131.6) and (218.53,131.6) .. (220.2,129.93) .. controls (221.87,128.26) and (223.53,128.26) .. (225.2,129.93) .. controls (226.87,131.6) and (228.53,131.6) .. (230.2,129.93) .. controls (231.87,128.26) and (233.53,128.26) .. (235.2,129.93) .. controls (236.87,131.6) and (238.53,131.6) .. (240.2,129.93) .. controls (241.87,128.26) and (243.53,128.26) .. (245.2,129.93) .. controls (246.87,131.6) and (248.53,131.6) .. (250.2,129.93) .. controls (251.87,128.26) and (253.53,128.26) .. (255.2,129.93) .. controls (256.87,131.6) and (258.53,131.6) .. (260.2,129.93) .. controls (261.87,128.26) and (263.53,128.26) .. (265.2,129.93) .. controls (266.87,131.6) and (268.53,131.6) .. (270.2,129.93) .. controls (271.87,128.26) and (273.53,128.26) .. (275.2,129.93) .. controls (276.87,131.6) and (278.53,131.6) .. (280.2,129.93) -- (280.67,129.93) -- (288.67,129.93) ;
\draw [shift={(290.67,129.93)}, rotate = 180] [color={rgb, 255:red, 0; green, 0; blue, 0 }  ][line width=0.75]    (9.84,-4.41) .. controls (6.25,-2.07) and (2.97,-0.6) .. (0,0) .. controls (2.97,0.6) and (6.25,2.07) .. (9.84,4.41)   ;
\draw [shift={(240.43,129.93)}, rotate = 45] [color={rgb, 255:red, 0; green, 0; blue, 0 }  ][line width=0.75]    (-5.03,0) -- (5.03,0)(0,5.03) -- (0,-5.03)   ;

\draw (213,105) node [anchor=north west][inner sep=0.75pt]  [font=\small]  {$\zeta ^{-\ } \ x\ \ \ \zeta ^{+}$};

\end{tikzpicture}

%% file: mu2.tex
\begin{tikzpicture}[x=0.75pt,y=0.75pt,yscale=-1,xscale=1]

\draw    (375.73,83.6) -- (401.69,100.5) ;
\draw [shift={(404.2,102.13)}, rotate = 213.07] [fill={rgb, 255:red, 0; green, 0; blue, 0 }  ][line width=0.08]  [draw opacity=0] (7.14,-3.43) -- (0,0) -- (7.14,3.43) -- (4.74,0) -- cycle    ;
\draw    (377.07,121.6) -- (401.76,103.88) ;
\draw [shift={(404.2,102.13)}, rotate = 144.34] [fill={rgb, 255:red, 0; green, 0; blue, 0 }  ][line width=0.08]  [draw opacity=0] (7.14,-3.43) -- (0,0) -- (7.14,3.43) -- (4.74,0) -- cycle    ;
\draw    (404.2,102.13) .. controls (405.83,100.44) and (407.5,100.41) .. (409.2,102.04) .. controls (410.9,103.67) and (412.57,103.64) .. (414.2,101.94) .. controls (415.84,100.25) and (417.51,100.22) .. (419.2,101.85) .. controls (420.9,103.48) and (422.57,103.45) .. (424.2,101.75) .. controls (425.84,100.06) and (427.51,100.03) .. (429.2,101.66) -- (433.87,101.57) -- (441.87,101.42) ;
\draw [shift={(444.87,101.36)}, rotate = 178.91] [fill={rgb, 255:red, 0; green, 0; blue, 0 }  ][line width=0.08]  [draw opacity=0] (7.14,-3.43) -- (0,0) -- (7.14,3.43) -- (4.74,0) -- cycle    ;
\draw    (444.87,101.36) -- (470.82,118.26) ;
\draw [shift={(473.33,119.89)}, rotate = 213.07] [fill={rgb, 255:red, 0; green, 0; blue, 0 }  ][line width=0.08]  [draw opacity=0] (7.14,-3.43) -- (0,0) -- (7.14,3.43) -- (4.74,0) -- cycle    ;
\draw    (444.87,101.36) -- (469.56,83.64) ;
\draw [shift={(472,81.89)}, rotate = 144.34] [fill={rgb, 255:red, 0; green, 0; blue, 0 }  ][line width=0.08]  [draw opacity=0] (7.14,-3.43) -- (0,0) -- (7.14,3.43) -- (4.74,0) -- cycle    ;
\draw    (375.73,137.6) -- (401.69,154.5) ;
\draw [shift={(404.2,156.13)}, rotate = 213.07] [fill={rgb, 255:red, 0; green, 0; blue, 0 }  ][line width=0.08]  [draw opacity=0] (7.14,-3.43) -- (0,0) -- (7.14,3.43) -- (4.74,0) -- cycle    ;
\draw    (377.07,175.6) -- (401.76,157.88) ;
\draw [shift={(404.2,156.13)}, rotate = 144.34] [fill={rgb, 255:red, 0; green, 0; blue, 0 }  ][line width=0.08]  [draw opacity=0] (7.14,-3.43) -- (0,0) -- (7.14,3.43) -- (4.74,0) -- cycle    ;
\draw    (404.2,156.13) .. controls (405.85,154.45) and (407.52,154.44) .. (409.2,156.09) .. controls (410.88,157.74) and (412.55,157.73) .. (414.2,156.05) .. controls (415.85,154.37) and (417.52,154.36) .. (419.2,156.01) .. controls (420.88,157.66) and (422.55,157.65) .. (424.2,155.97) .. controls (425.85,154.29) and (427.52,154.28) .. (429.2,155.93) .. controls (430.88,157.58) and (432.55,157.57) .. (434.2,155.89) .. controls (435.85,154.21) and (437.52,154.2) .. (439.2,155.85) .. controls (440.88,157.5) and (442.55,157.49) .. (444.2,155.81) .. controls (445.86,154.14) and (447.53,154.13) .. (449.2,155.78) .. controls (450.88,157.43) and (452.55,157.42) .. (454.2,155.74) .. controls (455.85,154.06) and (457.52,154.05) .. (459.2,155.7) -- (460.2,155.69) -- (468.2,155.62) ;
\draw [shift={(471.2,155.6)}, rotate = 179.54] [fill={rgb, 255:red, 0; green, 0; blue, 0 }  ][line width=0.08]  [draw opacity=0] (7.14,-3.43) -- (0,0) -- (7.14,3.43) -- (4.74,0) -- cycle    ;
\draw [shift={(437.7,155.87)}, rotate = 44.54] [color={rgb, 255:red, 0; green, 0; blue, 0 }  ][line width=0.75]    (-3.35,0) -- (3.35,0)(0,3.35) -- (0,-3.35)   ;
\draw    (371.86,204.73) .. controls (373.54,203.08) and (375.21,203.1) .. (376.86,204.78) .. controls (378.51,206.46) and (380.18,206.47) .. (381.86,204.82) .. controls (383.54,203.17) and (385.2,203.18) .. (386.85,204.86) .. controls (388.5,206.54) and (390.17,206.56) .. (391.85,204.91) .. controls (393.53,203.26) and (395.2,203.27) .. (396.85,204.95) .. controls (398.5,206.63) and (400.17,206.64) .. (401.85,204.99) .. controls (403.53,203.34) and (405.2,203.35) .. (406.85,205.03) .. controls (408.5,206.71) and (410.17,206.73) .. (411.85,205.08) .. controls (413.53,203.43) and (415.2,203.44) .. (416.85,205.12) .. controls (418.5,206.8) and (420.17,206.81) .. (421.85,205.16) .. controls (423.53,203.51) and (425.2,203.53) .. (426.85,205.21) .. controls (428.5,206.89) and (430.17,206.9) .. (431.85,205.25) -- (433.87,205.27) -- (441.87,205.34) ;
\draw [shift={(444.87,205.36)}, rotate = 180.49] [fill={rgb, 255:red, 0; green, 0; blue, 0 }  ][line width=0.08]  [draw opacity=0] (7.14,-3.43) -- (0,0) -- (7.14,3.43) -- (4.74,0) -- cycle    ;
\draw [shift={(408.36,205.05)}, rotate = 45.49] [color={rgb, 255:red, 0; green, 0; blue, 0 }  ][line width=0.75]    (-3.35,0) -- (3.35,0)(0,3.35) -- (0,-3.35)   ;
\draw    (444.87,205.36) -- (470.82,222.26) ;
\draw [shift={(473.33,223.89)}, rotate = 213.07] [fill={rgb, 255:red, 0; green, 0; blue, 0 }  ][line width=0.08]  [draw opacity=0] (7.14,-3.43) -- (0,0) -- (7.14,3.43) -- (4.74,0) -- cycle    ;
\draw    (444.87,205.36) -- (469.56,187.64) ;
\draw [shift={(472,185.89)}, rotate = 144.34] [fill={rgb, 255:red, 0; green, 0; blue, 0 }  ][line width=0.08]  [draw opacity=0] (7.14,-3.43) -- (0,0) -- (7.14,3.43) -- (4.74,0) -- cycle    ;
\draw    (381.2,60.6) .. controls (382.87,58.93) and (384.53,58.93) .. (386.2,60.6) .. controls (387.87,62.27) and (389.53,62.27) .. (391.2,60.6) .. controls (392.87,58.93) and (394.53,58.93) .. (396.2,60.6) .. controls (397.87,62.27) and (399.53,62.27) .. (401.2,60.6) .. controls (402.87,58.93) and (404.53,58.93) .. (406.2,60.6) .. controls (407.87,62.27) and (409.53,62.27) .. (411.2,60.6) .. controls (412.87,58.93) and (414.53,58.93) .. (416.2,60.6) -- (420.7,60.6) -- (420.7,60.6) .. controls (422.37,58.93) and (424.03,58.93) .. (425.7,60.6) .. controls (427.37,62.27) and (429.03,62.27) .. (430.7,60.6) .. controls (432.37,58.93) and (434.03,58.93) .. (435.7,60.6) .. controls (437.37,62.27) and (439.03,62.27) .. (440.7,60.6) .. controls (442.37,58.93) and (444.03,58.93) .. (445.7,60.6) -- (449.2,60.6) -- (457.2,60.6) ;
\draw [shift={(460.2,60.6)}, rotate = 180] [fill={rgb, 255:red, 0; green, 0; blue, 0 }  ][line width=0.08]  [draw opacity=0] (7.14,-3.43) -- (0,0) -- (7.14,3.43) -- (4.74,0) -- cycle    ;
\draw [shift={(400.95,60.6)}, rotate = 45] [color={rgb, 255:red, 0; green, 0; blue, 0 }  ][line width=0.75]    (-3.35,0) -- (3.35,0)(0,3.35) -- (0,-3.35)   ;
\draw [shift={(440.45,60.6)}, rotate = 45] [color={rgb, 255:red, 0; green, 0; blue, 0 }  ][line width=0.75]    (-3.35,0) -- (3.35,0)(0,3.35) -- (0,-3.35)   ;

\draw (123,47.6) node [anchor=north west][inner sep=0.75pt]  [font=\small] [align=left] {$\displaystyle -B_{\zeta _{x}^{+}} +B_{\zeta _{x'}^{-}}$};
\draw (351.33,46.06) node [anchor=north west][inner sep=0.75pt]   [align=left] {{\small if}};
\draw (398.39,40.58) node [anchor=north west][inner sep=0.75pt]  [font=\scriptsize] [align=left] {$\displaystyle x\ \ \ \zeta \ \ \ \ x'$};
\draw (370,88.67) node [anchor=north west][inner sep=0.75pt]  [font=\scriptsize] [align=left] {$\displaystyle \vdots $};
\draw (467.1,87.4) node [anchor=north west][inner sep=0.75pt]  [font=\scriptsize] [align=left] {$\displaystyle \vdots $};
\draw (388,79.67) node [anchor=north west][inner sep=0.75pt]  [font=\tiny] [align=left] {$\displaystyle e$};
\draw (447.33,79.33) node [anchor=north west][inner sep=0.75pt]  [font=\tiny] [align=left] {$\displaystyle e'$};
\draw (419,84.33) node [anchor=north west][inner sep=0.75pt]  [font=\scriptsize] [align=left] {$\displaystyle \zeta $};
\draw (352,92.73) node [anchor=north west][inner sep=0.75pt]   [align=left] {{\small if}};
\draw (62.67,80.33) node [anchor=north west][inner sep=0.75pt]  [font=\small] [align=left] {$\displaystyle \sum _{h :\ h( e) =\zeta } C_{e} D_{e} -\sum _{e' :\ t( e') =\zeta } D_{e'} C_{e'}$};
\draw (369.47,144) node [anchor=north west][inner sep=0.75pt]  [font=\scriptsize] [align=left] {$\displaystyle \vdots $};
\draw (388,133.67) node [anchor=north west][inner sep=0.75pt]  [font=\tiny] [align=left] {$\displaystyle e$};
\draw (415,136.33) node [anchor=north west][inner sep=0.75pt]  [font=\scriptsize] [align=left] {$\displaystyle \zeta \ \ \ x$};
\draw (352,146.73) node [anchor=north west][inner sep=0.75pt]   [align=left] {{\small if}};
\draw (89.33,131.67) node [anchor=north west][inner sep=0.75pt]   [align=left] {$\displaystyle \sum _{e:\ h( e) =\zeta } C_{e} D_{e} \ +\ B_{\zeta _{x}^{-}}$};
\draw (467.1,191) node [anchor=north west][inner sep=0.75pt]  [font=\scriptsize] [align=left] {$\displaystyle \vdots $};
\draw (448.33,183.33) node [anchor=north west][inner sep=0.75pt]  [font=\tiny] [align=left] {$\displaystyle e'$};
\draw (404,186.33) node [anchor=north west][inner sep=0.75pt]  [font=\scriptsize] [align=left] {$\displaystyle x\ \ \ \ \zeta $};
\draw (352,196.73) node [anchor=north west][inner sep=0.75pt]   [align=left] {{\small if}};
\draw (77.33,179.67) node [anchor=north west][inner sep=0.75pt]   [align=left] {$\displaystyle -
\sum _{e':\ o( e') =\zeta } D_{e'} C_{e'} \ -\ B_{\zeta _{x}^{+}}$};

\end{tikzpicture}

%% file: notation-gdp.tex
\tikzset{every picture/.style={line width=0.75pt}} 

\begin{tikzpicture}[x=0.75pt,y=0.75pt,yscale=-1,xscale=1]

\draw    (100.1,120.05) .. controls (101.77,118.38) and (103.43,118.38) .. (105.1,120.05) .. controls (106.77,121.72) and (108.43,121.72) .. (110.1,120.05) .. controls (111.77,118.38) and (113.43,118.38) .. (115.1,120.05) .. controls (116.77,121.72) and (118.43,121.72) .. (120.1,120.05) .. controls (121.77,118.38) and (123.43,118.38) .. (125.1,120.05) .. controls (126.77,121.72) and (128.43,121.72) .. (130.1,120.05) .. controls (131.77,118.38) and (133.43,118.38) .. (135.1,120.05) .. controls (136.77,121.72) and (138.43,121.72) .. (140.1,120.05) .. controls (141.77,118.38) and (143.43,118.38) .. (145.1,120.05) .. controls (146.77,121.72) and (148.43,121.72) .. (150.1,120.05) -- (150.6,120.05) -- (150.6,120.05) ;
\draw [shift={(150.6,120.05)}, rotate = 45] [color={rgb, 255:red, 0; green, 0; blue, 0 }  ][line width=0.75]    (-4.47,0) -- (4.47,0)(0,4.47) -- (0,-4.47)   ;
\draw    (150.6,120.05) .. controls (152.27,118.38) and (153.93,118.38) .. (155.6,120.05) .. controls (157.27,121.72) and (158.93,121.72) .. (160.6,120.05) .. controls (162.27,118.38) and (163.93,118.38) .. (165.6,120.05) .. controls (167.27,121.72) and (168.93,121.72) .. (170.6,120.05) .. controls (172.27,118.38) and (173.93,118.38) .. (175.6,120.05) .. controls (177.27,121.72) and (178.93,121.72) .. (180.6,120.05) .. controls (182.27,118.38) and (183.93,118.38) .. (185.6,120.05) .. controls (187.27,121.72) and (188.93,121.72) .. (190.6,120.05) .. controls (192.27,118.38) and (193.93,118.38) .. (195.6,120.05) -- (200.1,120.05) -- (200.1,120.05) ;
\draw [shift={(200.1,120.05)}, rotate = 45] [color={rgb, 255:red, 0; green, 0; blue, 0 }  ][line width=0.75]    (-4.47,0) -- (4.47,0)(0,4.47) -- (0,-4.47)   ;
\draw    (200.1,120.05) .. controls (201.77,118.38) and (203.43,118.38) .. (205.1,120.05) .. controls (206.77,121.72) and (208.43,121.72) .. (210.1,120.05) .. controls (211.77,118.38) and (213.43,118.38) .. (215.1,120.05) .. controls (216.77,121.72) and (218.43,121.72) .. (220.1,120.05) .. controls (221.77,118.38) and (223.43,118.38) .. (225.1,120.05) -- (225.1,120.05) ;
\draw  [dash pattern={on 0.84pt off 2.51pt}]  (226,120.05) -- (250.6,120.05) ;
\draw    (249.6,120.05) .. controls (251.27,118.38) and (252.93,118.38) .. (254.6,120.05) .. controls (256.27,121.72) and (257.93,121.72) .. (259.6,120.05) .. controls (261.27,118.38) and (262.93,118.38) .. (264.6,120.05) .. controls (266.27,121.72) and (267.93,121.72) .. (269.6,120.05) .. controls (271.27,118.38) and (272.93,118.38) .. (274.6,120.05) .. controls (276.27,121.72) and (277.93,121.72) .. (279.6,120.05) .. controls (281.27,118.38) and (282.93,118.38) .. (284.6,120.05) .. controls (286.27,121.72) and (287.93,121.72) .. (289.6,120.05) .. controls (291.27,118.38) and (292.93,118.38) .. (294.6,120.05) .. controls (296.27,121.72) and (297.93,121.72) .. (299.6,120.05) -- (300.1,120.05) -- (300.1,120.05) ;
\draw [shift={(300.1,120.05)}, rotate = 45] [color={rgb, 255:red, 0; green, 0; blue, 0 }  ][line width=0.75]    (-4.47,0) -- (4.47,0)(0,4.47) -- (0,-4.47)   ;
\draw    (300.1,120.05) .. controls (301.77,118.38) and (303.43,118.38) .. (305.1,120.05) .. controls (306.77,121.72) and (308.43,121.72) .. (310.1,120.05) .. controls (311.77,118.38) and (313.43,118.38) .. (315.1,120.05) .. controls (316.77,121.72) and (318.43,121.72) .. (320.1,120.05) .. controls (321.77,118.38) and (323.43,118.38) .. (325.1,120.05) .. controls (326.77,121.72) and (328.43,121.72) .. (330.1,120.05) .. controls (331.77,118.38) and (333.43,118.38) .. (335.1,120.05) .. controls (336.77,121.72) and (338.43,121.72) .. (340.1,120.05) .. controls (341.77,118.38) and (343.43,118.38) .. (345.1,120.05) -- (349.6,120.05) -- (349.6,120.05) ;
\draw  [dash pattern={on 0.84pt off 2.51pt}]  (349.6,120.05) -- (374.2,120.05) ;
\draw    (374.2,120.05) .. controls (375.87,118.38) and (377.53,118.38) .. (379.2,120.05) .. controls (380.87,121.72) and (382.53,121.72) .. (384.2,120.05) .. controls (385.87,118.38) and (387.53,118.38) .. (389.2,120.05) .. controls (390.87,121.72) and (392.53,121.72) .. (394.2,120.05) .. controls (395.87,118.38) and (397.53,118.38) .. (399.2,120.05) .. controls (400.87,121.72) and (402.53,121.72) .. (404.2,120.05) .. controls (405.87,118.38) and (407.53,118.38) .. (409.2,120.05) .. controls (410.87,121.72) and (412.53,121.72) .. (414.2,120.05) .. controls (415.87,118.38) and (417.53,118.38) .. (419.2,120.05) .. controls (420.87,121.72) and (422.53,121.72) .. (424.2,120.05) -- (424.7,120.05) -- (424.7,120.05) ;
\draw [shift={(424.7,120.05)}, rotate = 45] [color={rgb, 255:red, 0; green, 0; blue, 0 }  ][line width=0.75]    (-4.47,0) -- (4.47,0)(0,4.47) -- (0,-4.47)   ;
\draw    (424.7,120.05) .. controls (426.37,118.38) and (428.03,118.38) .. (429.7,120.05) .. controls (431.37,121.72) and (433.03,121.72) .. (434.7,120.05) .. controls (436.37,118.38) and (438.03,118.38) .. (439.7,120.05) .. controls (441.37,121.72) and (443.03,121.72) .. (444.7,120.05) .. controls (446.37,118.38) and (448.03,118.38) .. (449.7,120.05) .. controls (451.37,121.72) and (453.03,121.72) .. (454.7,120.05) .. controls (456.37,118.38) and (458.03,118.38) .. (459.7,120.05) -- (464.2,120.05) -- (472.2,120.05) ;
\draw [shift={(474.2,120.05)}, rotate = 180] [color={rgb, 255:red, 0; green, 0; blue, 0 }  ][line width=0.75]    (8.74,-3.92) .. controls (5.56,-1.84) and (2.65,-0.53) .. (0,0) .. controls (2.65,0.53) and (5.56,1.84) .. (8.74,3.92)   ;

\draw (63.5,115) node [anchor=north west][inner sep=0.75pt]  [font=\normalsize]  {$\sigma :$};
\draw (122,95.05) node [anchor=north west][inner sep=0.75pt]    {$\zeta _{\sigma }^{0} \ \ \ \ \ \ \ \zeta _{\sigma }^{1} \ \ \ \ \ \ \ \ \ \ \ \ \ \ \ \ \ \ \zeta _{\sigma }^{i-1} \ \ \ \ \ \ \zeta _{\sigma }^{i} \ \ \ \ \ \ \ \ \ \ \ \ \ \ \ \zeta _{\sigma }^{w_{\sigma } -1} \ \  \zeta _{\sigma }^{w_{\sigma }}$};
\draw (143,129.9) node [anchor=north west][inner sep=0.75pt]    {$\ x_{\sigma ,1} \ \ \ \ \ x_{\sigma ,2} \ \ \ \ \ \ \ \ \ \ \ \ \ \ \ \ \ x_{\sigma ,i} \ \ \ \ \ \ \ \ \ \ \ \ \ \ \ \ \ \ \ \ \ \ x_{\sigma ,w_{\sigma }}$};

\end{tikzpicture}

%% file: notation-gdp2.tex
\begin{tikzcd}
{V_{\sigma,i}^-} \arrow[rd, "{b_{\sigma,i}}"'] \arrow[rr, "{A_{\sigma,i}}"] \arrow["{B_{\sigma,i}^-}"', loop, distance=2em, in=125, out=55] &                                          & {V_{\sigma,i}^+} \arrow["{B_{\sigma,i}^+}"', loop, distance=2em, in=125, out=55] \\
                                                                                                                                            & \mathbb{C} \arrow[ru, "{a_{\sigma,i}}"'] &                                                                                 
\end{tikzcd}

%% file: first-gdp.tex
\vcenter{\hbox{\begin{tikzpicture}[x=0.75pt,y=0.75pt,yscale=-1,xscale=1]

\draw    (395.73,157.6) -- (421.69,174.5) ;
\draw [shift={(424.2,176.13)}, rotate = 213.07] [fill={rgb, 255:red, 0; green, 0; blue, 0 }  ][line width=0.08]  [draw opacity=0] (7.14,-3.43) -- (0,0) -- (7.14,3.43) -- (4.74,0) -- cycle    ;
\draw    (397.07,195.6) -- (421.76,177.88) ;
\draw [shift={(424.2,176.13)}, rotate = 144.34] [fill={rgb, 255:red, 0; green, 0; blue, 0 }  ][line width=0.08]  [draw opacity=0] (7.14,-3.43) -- (0,0) -- (7.14,3.43) -- (4.74,0) -- cycle    ;
\draw    (424.2,176.13) .. controls (425.85,174.45) and (427.52,174.44) .. (429.2,176.09) .. controls (430.88,177.74) and (432.55,177.73) .. (434.2,176.05) .. controls (435.85,174.37) and (437.52,174.36) .. (439.2,176.01) .. controls (440.88,177.66) and (442.55,177.65) .. (444.2,175.97) .. controls (445.85,174.29) and (447.52,174.28) .. (449.2,175.93) .. controls (450.88,177.58) and (452.55,177.57) .. (454.2,175.89) .. controls (455.85,174.21) and (457.52,174.2) .. (459.2,175.85) .. controls (460.88,177.5) and (462.55,177.49) .. (464.2,175.81) .. controls (465.86,174.14) and (467.53,174.13) .. (469.2,175.78) .. controls (470.88,177.43) and (472.55,177.42) .. (474.2,175.74) .. controls (475.85,174.06) and (477.52,174.05) .. (479.2,175.7) -- (480.2,175.69) -- (488.2,175.62) ;
\draw [shift={(491.2,175.6)}, rotate = 179.54] [fill={rgb, 255:red, 0; green, 0; blue, 0 }  ][line width=0.08]  [draw opacity=0] (7.14,-3.43) -- (0,0) -- (7.14,3.43) -- (4.74,0) -- cycle    ;
\draw [shift={(457.7,175.87)}, rotate = 44.54] [color={rgb, 255:red, 0; green, 0; blue, 0 }  ][line width=0.75]    (-3.35,0) -- (3.35,0)(0,3.35) -- (0,-3.35)   ;

\draw (389.47,164) node [anchor=north west][inner sep=0.75pt]  [font=\scriptsize] [align=left] {$\displaystyle \vdots $};
\draw (408,153.67) node [anchor=north west][inner sep=0.75pt]  [font=\tiny] [align=left] {$ $};
\draw (425,156.33) node [anchor=north west][inner sep=0.75pt]  [font=\scriptsize] [align=left] {$\displaystyle \zeta =\zeta _{\sigma}^0 \ \ \ \ $};
\draw (452,183.4) node [anchor=north west][inner sep=0.75pt]  [font=\scriptsize]  {$x_{\sigma ,1}$};
\end{tikzpicture}
}}

%% file: second-gdp.tex
\vcenter{\hbox{\begin{tikzpicture}[x=0.75pt,y=0.75pt,yscale=-1,xscale=1]

\draw    (400.53,158.89) .. controls (402.2,157.22) and (403.86,157.22) .. (405.53,158.89) .. controls (407.2,160.56) and (408.86,160.56) .. (410.53,158.89) .. controls (412.2,157.22) and (413.86,157.22) .. (415.53,158.89) .. controls (417.2,160.56) and (418.86,160.56) .. (420.53,158.89) .. controls (422.2,157.22) and (423.86,157.22) .. (425.53,158.89) .. controls (427.2,160.56) and (428.86,160.56) .. (430.53,158.89) .. controls (432.2,157.22) and (433.86,157.22) .. (435.53,158.89) .. controls (437.2,160.56) and (438.86,160.56) .. (440.53,158.89) .. controls (442.2,157.22) and (443.86,157.22) .. (445.53,158.89) -- (448,158.89) -- (448,158.89) .. controls (449.67,157.22) and (451.33,157.22) .. (453,158.89) .. controls (454.67,160.56) and (456.33,160.56) .. (458,158.89) .. controls (459.67,157.22) and (461.33,157.22) .. (463,158.89) .. controls (464.67,160.56) and (466.33,160.56) .. (468,158.89) .. controls (469.67,157.22) and (471.33,157.22) .. (473,158.89) .. controls (474.67,160.56) and (476.33,160.56) .. (478,158.89) .. controls (479.67,157.22) and (481.33,157.22) .. (483,158.89) -- (486.33,158.89) -- (494.33,158.89) ;
\draw [shift={(496.33,158.89)}, rotate = 180] [color={rgb, 255:red, 0; green, 0; blue, 0 }  ][line width=0.75]    (6.56,-2.94) .. controls (4.17,-1.38) and (1.99,-0.4) .. (0,0) .. controls (1.99,0.4) and (4.17,1.38) .. (6.56,2.94)   ;
\draw [shift={(424.27,158.89)}, rotate = 45] [color={rgb, 255:red, 0; green, 0; blue, 0 }  ][line width=0.75]    (-3.35,0) -- (3.35,0)(0,3.35) -- (0,-3.35)   ;
\draw [shift={(472.17,158.89)}, rotate = 45] [color={rgb, 255:red, 0; green, 0; blue, 0 }  ][line width=0.75]    (-3.35,0) -- (3.35,0)(0,3.35) -- (0,-3.35)   ;

\draw (421.39,165.08) node [anchor=north west][inner sep=0.75pt]  [font=\scriptsize] [align=left] {$\displaystyle x_{\sigma ,i} \ \ \ \ \ \ \ x_{\sigma ,i+1}$};
\draw (432.13,141.73) node [anchor=north west][inner sep=0.75pt]  [font=\scriptsize]  {$\zeta =\zeta _{\sigma }^{i}$};

\end{tikzpicture}

}}

%% file: third-gdp.tex
\vcenter{\hbox{\begin{tikzpicture}[x=0.75pt,y=0.75pt,yscale=-1,xscale=1]

\draw    (391.86,224.73) .. controls (393.54,223.08) and (395.21,223.1) .. (396.86,224.78) .. controls (398.51,226.46) and (400.18,226.47) .. (401.86,224.82) .. controls (403.54,223.17) and (405.2,223.18) .. (406.85,224.86) .. controls (408.5,226.54) and (410.17,226.56) .. (411.85,224.91) .. controls (413.53,223.26) and (415.2,223.27) .. (416.85,224.95) .. controls (418.5,226.63) and (420.17,226.64) .. (421.85,224.99) .. controls (423.53,223.34) and (425.2,223.35) .. (426.85,225.03) .. controls (428.5,226.71) and (430.17,226.73) .. (431.85,225.08) .. controls (433.53,223.43) and (435.2,223.44) .. (436.85,225.12) .. controls (438.5,226.8) and (440.17,226.81) .. (441.85,225.16) .. controls (443.53,223.51) and (445.2,223.53) .. (446.85,225.21) .. controls (448.5,226.89) and (450.17,226.9) .. (451.85,225.25) -- (453.87,225.27) -- (461.87,225.34) ;
\draw [shift={(464.87,225.36)}, rotate = 180.49] [fill={rgb, 255:red, 0; green, 0; blue, 0 }  ][line width=0.08]  [draw opacity=0] (7.14,-3.43) -- (0,0) -- (7.14,3.43) -- (4.74,0) -- cycle    ;
\draw [shift={(428.36,225.05)}, rotate = 45.49] [color={rgb, 255:red, 0; green, 0; blue, 0 }  ][line width=0.75]    (-3.35,0) -- (3.35,0)(0,3.35) -- (0,-3.35)   ;
\draw    (464.87,225.36) -- (490.82,242.26) ;
\draw [shift={(493.33,243.89)}, rotate = 213.07] [fill={rgb, 255:red, 0; green, 0; blue, 0 }  ][line width=0.08]  [draw opacity=0] (7.14,-3.43) -- (0,0) -- (7.14,3.43) -- (4.74,0) -- cycle    ;
\draw    (464.87,225.36) -- (489.56,207.64) ;
\draw [shift={(492,205.89)}, rotate = 144.34] [fill={rgb, 255:red, 0; green, 0; blue, 0 }  ][line width=0.08]  [draw opacity=0] (7.14,-3.43) -- (0,0) -- (7.14,3.43) -- (4.74,0) -- cycle    ;

\draw (487.1,211) node [anchor=north west][inner sep=0.75pt]  [font=\scriptsize] [align=left] {$\displaystyle \vdots $};
\draw (414,205.33) node [anchor=north west][inner sep=0.75pt]  [font=\scriptsize] [align=left] {$\displaystyle \ \ \ \ \ \zeta =\zeta {_{\sigma }^{w_{\sigma }}}$};
\draw (419,232.4) node [anchor=north west][inner sep=0.75pt]  [font=\scriptsize]  {$x_{\sigma ,w_{\sigma }}$};
\end{tikzpicture}}}

%% file: fourth-gdp.tex
\vcenter{\hbox{\begin{tikzpicture}[x=0.75pt,y=0.75pt,yscale=-1,xscale=1]

\draw    (395.73,103.6) -- (421.69,120.5) ;
\draw [shift={(424.2,122.13)}, rotate = 213.07] [fill={rgb, 255:red, 0; green, 0; blue, 0 }  ][line width=0.08]  [draw opacity=0] (7.14,-3.43) -- (0,0) -- (7.14,3.43) -- (4.74,0) -- cycle    ;
\draw    (397.07,141.6) -- (421.76,123.88) ;
\draw [shift={(424.2,122.13)}, rotate = 144.34] [fill={rgb, 255:red, 0; green, 0; blue, 0 }  ][line width=0.08]  [draw opacity=0] (7.14,-3.43) -- (0,0) -- (7.14,3.43) -- (4.74,0) -- cycle    ;
\draw    (424.2,122.13) .. controls (425.83,120.44) and (427.5,120.41) .. (429.2,122.04) .. controls (430.9,123.67) and (432.57,123.64) .. (434.2,121.94) .. controls (435.84,120.25) and (437.51,120.22) .. (439.2,121.85) .. controls (440.9,123.48) and (442.57,123.45) .. (444.2,121.75) .. controls (445.84,120.06) and (447.51,120.03) .. (449.2,121.66) -- (453.87,121.57) -- (461.87,121.42) ;
\draw [shift={(464.87,121.36)}, rotate = 178.91] [fill={rgb, 255:red, 0; green, 0; blue, 0 }  ][line width=0.08]  [draw opacity=0] (7.14,-3.43) -- (0,0) -- (7.14,3.43) -- (4.74,0) -- cycle    ;
\draw    (464.87,121.36) -- (490.82,138.26) ;
\draw [shift={(493.33,139.89)}, rotate = 213.07] [fill={rgb, 255:red, 0; green, 0; blue, 0 }  ][line width=0.08]  [draw opacity=0] (7.14,-3.43) -- (0,0) -- (7.14,3.43) -- (4.74,0) -- cycle    ;
\draw    (464.87,121.36) -- (489.56,103.64) ;
\draw [shift={(492,101.89)}, rotate = 144.34] [fill={rgb, 255:red, 0; green, 0; blue, 0 }  ][line width=0.08]  [draw opacity=0] (7.14,-3.43) -- (0,0) -- (7.14,3.43) -- (4.74,0) -- cycle    ;

\draw (390,108.67) node [anchor=north west][inner sep=0.75pt]  [font=\scriptsize] [align=left] {$\displaystyle \vdots $};
\draw (487.1,107.4) node [anchor=north west][inner sep=0.75pt]  [font=\scriptsize] [align=left] {$\displaystyle \vdots $};
\draw (439,104.33) node [anchor=north west][inner sep=0.75pt]  [font=\scriptsize] [align=left] {$\displaystyle \zeta $};
\end{tikzpicture}
}}

%% file: dmu2.tex
\begin{tikzpicture}[baseline=(current  bounding  box.center), x=0.75pt,y=0.75pt,yscale=-1,xscale=1]

\draw    (375.73,83.6) -- (401.69,100.5) ;
\draw [shift={(404.2,102.13)}, rotate = 213.07] [fill={rgb, 255:red, 0; green, 0; blue, 0 }  ][line width=0.08]  [draw opacity=0] (7.14,-3.43) -- (0,0) -- (7.14,3.43) -- (4.74,0) -- cycle    ;
\draw    (377.07,121.6) -- (401.76,103.88) ;
\draw [shift={(404.2,102.13)}, rotate = 144.34] [fill={rgb, 255:red, 0; green, 0; blue, 0 }  ][line width=0.08]  [draw opacity=0] (7.14,-3.43) -- (0,0) -- (7.14,3.43) -- (4.74,0) -- cycle    ;
\draw    (404.2,102.13) .. controls (405.83,100.44) and (407.5,100.41) .. (409.2,102.04) .. controls (410.9,103.67) and (412.57,103.64) .. (414.2,101.94) .. controls (415.84,100.25) and (417.51,100.22) .. (419.2,101.85) .. controls (420.9,103.48) and (422.57,103.45) .. (424.2,101.75) .. controls (425.84,100.06) and (427.51,100.03) .. (429.2,101.66) -- (433.87,101.57) -- (441.87,101.42) ;
\draw [shift={(444.87,101.36)}, rotate = 178.91] [fill={rgb, 255:red, 0; green, 0; blue, 0 }  ][line width=0.08]  [draw opacity=0] (7.14,-3.43) -- (0,0) -- (7.14,3.43) -- (4.74,0) -- cycle    ;
\draw    (444.87,101.36) -- (470.82,118.26) ;
\draw [shift={(473.33,119.89)}, rotate = 213.07] [fill={rgb, 255:red, 0; green, 0; blue, 0 }  ][line width=0.08]  [draw opacity=0] (7.14,-3.43) -- (0,0) -- (7.14,3.43) -- (4.74,0) -- cycle    ;
\draw    (444.87,101.36) -- (469.56,83.64) ;
\draw [shift={(472,81.89)}, rotate = 144.34] [fill={rgb, 255:red, 0; green, 0; blue, 0 }  ][line width=0.08]  [draw opacity=0] (7.14,-3.43) -- (0,0) -- (7.14,3.43) -- (4.74,0) -- cycle    ;
\draw    (375.73,137.6) -- (401.69,154.5) ;
\draw [shift={(404.2,156.13)}, rotate = 213.07] [fill={rgb, 255:red, 0; green, 0; blue, 0 }  ][line width=0.08]  [draw opacity=0] (7.14,-3.43) -- (0,0) -- (7.14,3.43) -- (4.74,0) -- cycle    ;
\draw    (377.07,175.6) -- (401.76,157.88) ;
\draw [shift={(404.2,156.13)}, rotate = 144.34] [fill={rgb, 255:red, 0; green, 0; blue, 0 }  ][line width=0.08]  [draw opacity=0] (7.14,-3.43) -- (0,0) -- (7.14,3.43) -- (4.74,0) -- cycle    ;
\draw    (404.2,156.13) .. controls (405.85,154.45) and (407.52,154.44) .. (409.2,156.09) .. controls (410.88,157.74) and (412.55,157.73) .. (414.2,156.05) .. controls (415.85,154.37) and (417.52,154.36) .. (419.2,156.01) .. controls (420.88,157.66) and (422.55,157.65) .. (424.2,155.97) .. controls (425.85,154.29) and (427.52,154.28) .. (429.2,155.93) .. controls (430.88,157.58) and (432.55,157.57) .. (434.2,155.89) .. controls (435.85,154.21) and (437.52,154.2) .. (439.2,155.85) .. controls (440.88,157.5) and (442.55,157.49) .. (444.2,155.81) .. controls (445.86,154.14) and (447.53,154.13) .. (449.2,155.78) .. controls (450.88,157.43) and (452.55,157.42) .. (454.2,155.74) .. controls (455.85,154.06) and (457.52,154.05) .. (459.2,155.7) -- (460.2,155.69) -- (468.2,155.62) ;
\draw [shift={(471.2,155.6)}, rotate = 179.54] [fill={rgb, 255:red, 0; green, 0; blue, 0 }  ][line width=0.08]  [draw opacity=0] (7.14,-3.43) -- (0,0) -- (7.14,3.43) -- (4.74,0) -- cycle    ;
\draw [shift={(437.7,155.87)}, rotate = 44.54] [color={rgb, 255:red, 0; green, 0; blue, 0 }  ][line width=0.75]    (-3.35,0) -- (3.35,0)(0,3.35) -- (0,-3.35)   ;
\draw    (371.86,204.73) .. controls (373.54,203.08) and (375.21,203.1) .. (376.86,204.78) .. controls (378.51,206.46) and (380.18,206.47) .. (381.86,204.82) .. controls (383.54,203.17) and (385.2,203.18) .. (386.85,204.86) .. controls (388.5,206.54) and (390.17,206.56) .. (391.85,204.91) .. controls (393.53,203.26) and (395.2,203.27) .. (396.85,204.95) .. controls (398.5,206.63) and (400.17,206.64) .. (401.85,204.99) .. controls (403.53,203.34) and (405.2,203.35) .. (406.85,205.03) .. controls (408.5,206.71) and (410.17,206.73) .. (411.85,205.08) .. controls (413.53,203.43) and (415.2,203.44) .. (416.85,205.12) .. controls (418.5,206.8) and (420.17,206.81) .. (421.85,205.16) .. controls (423.53,203.51) and (425.2,203.53) .. (426.85,205.21) .. controls (428.5,206.89) and (430.17,206.9) .. (431.85,205.25) -- (433.87,205.27) -- (441.87,205.34) ;
\draw [shift={(444.87,205.36)}, rotate = 180.49] [fill={rgb, 255:red, 0; green, 0; blue, 0 }  ][line width=0.08]  [draw opacity=0] (7.14,-3.43) -- (0,0) -- (7.14,3.43) -- (4.74,0) -- cycle    ;
\draw [shift={(408.36,205.05)}, rotate = 45.49] [color={rgb, 255:red, 0; green, 0; blue, 0 }  ][line width=0.75]    (-3.35,0) -- (3.35,0)(0,3.35) -- (0,-3.35)   ;
\draw    (444.87,205.36) -- (470.82,222.26) ;
\draw [shift={(473.33,223.89)}, rotate = 213.07] [fill={rgb, 255:red, 0; green, 0; blue, 0 }  ][line width=0.08]  [draw opacity=0] (7.14,-3.43) -- (0,0) -- (7.14,3.43) -- (4.74,0) -- cycle    ;
\draw    (444.87,205.36) -- (469.56,187.64) ;
\draw [shift={(472,185.89)}, rotate = 144.34] [fill={rgb, 255:red, 0; green, 0; blue, 0 }  ][line width=0.08]  [draw opacity=0] (7.14,-3.43) -- (0,0) -- (7.14,3.43) -- (4.74,0) -- cycle    ;
\draw    (381.2,60.6) .. controls (382.87,58.93) and (384.53,58.93) .. (386.2,60.6) .. controls (387.87,62.27) and (389.53,62.27) .. (391.2,60.6) .. controls (392.87,58.93) and (394.53,58.93) .. (396.2,60.6) .. controls (397.87,62.27) and (399.53,62.27) .. (401.2,60.6) .. controls (402.87,58.93) and (404.53,58.93) .. (406.2,60.6) .. controls (407.87,62.27) and (409.53,62.27) .. (411.2,60.6) .. controls (412.87,58.93) and (414.53,58.93) .. (416.2,60.6) -- (420.7,60.6) -- (420.7,60.6) .. controls (422.37,58.93) and (424.03,58.93) .. (425.7,60.6) .. controls (427.37,62.27) and (429.03,62.27) .. (430.7,60.6) .. controls (432.37,58.93) and (434.03,58.93) .. (435.7,60.6) .. controls (437.37,62.27) and (439.03,62.27) .. (440.7,60.6) .. controls (442.37,58.93) and (444.03,58.93) .. (445.7,60.6) -- (449.2,60.6) -- (457.2,60.6) ;
\draw [shift={(460.2,60.6)}, rotate = 180] [fill={rgb, 255:red, 0; green, 0; blue, 0 }  ][line width=0.08]  [draw opacity=0] (7.14,-3.43) -- (0,0) -- (7.14,3.43) -- (4.74,0) -- cycle    ;
\draw [shift={(400.95,60.6)}, rotate = 45] [color={rgb, 255:red, 0; green, 0; blue, 0 }  ][line width=0.75]    (-3.35,0) -- (3.35,0)(0,3.35) -- (0,-3.35)   ;
\draw [shift={(440.45,60.6)}, rotate = 45] [color={rgb, 255:red, 0; green, 0; blue, 0 }  ][line width=0.75]    (-3.35,0) -- (3.35,0)(0,3.35) -- (0,-3.35)   ;

\draw (123,47.6) node [anchor=north west][inner sep=0.75pt]  [font=\small] [align=left] {$\displaystyle -\bdot{B_{\zeta _{x}^{+}}} +\bdot{B_{\zeta _{x'}^{-}}}$};
\draw (351.33,46.06) node [anchor=north west][inner sep=0.75pt]   [align=left] {{\small if}};
\draw (398.39,40.58) node [anchor=north west][inner sep=0.75pt]  [font=\scriptsize] [align=left] {$\displaystyle x\ \ \ \zeta \ \ \  x'$};
\draw (370,88.67) node [anchor=north west][inner sep=0.75pt]  [font=\scriptsize] [align=left] {$\displaystyle \vdots $};
\draw (467.1,87.4) node [anchor=north west][inner sep=0.75pt]  [font=\scriptsize] [align=left] {$\displaystyle \vdots $};
\draw (388,79.67) node [anchor=north west][inner sep=0.75pt]  [font=\tiny] [align=left] {$\displaystyle e$};
\draw (447.33,79.33) node [anchor=north west][inner sep=0.75pt]  [font=\tiny] [align=left] {$\displaystyle e'$};
\draw (419,84.33) node [anchor=north west][inner sep=0.75pt]  [font=\scriptsize] [align=left] {$\displaystyle \zeta $};
\draw (352,92.73) node [anchor=north west][inner sep=0.75pt]   [align=left] {{\small if}};
\draw (62.67,80.33) node [anchor=north west][inner sep=0.75pt]  [font=\small] [align=left] {$\displaystyle \sum _{e}\Bigl(\bdot{C_{e}} D_{e}+C_{e}\bdot{D_{e}}\Bigr) -\sum _{e'}\Bigl(\bdot{D_{e'}} C_{e'}+D_{e'} \bdot{C_{e'}}\Bigr)$};
\draw (369.47,144) node [anchor=north west][inner sep=0.75pt]  [font=\scriptsize] [align=left] {$\displaystyle \vdots $};
\draw (388,133.67) node [anchor=north west][inner sep=0.75pt]  [font=\tiny] [align=left] {$\displaystyle e$};
\draw (416,138) node [anchor=north west][inner sep=0.75pt]  [font=\scriptsize] [align=left] {$\displaystyle \zeta \ \ \ x$};
\draw (352,146.73) node [anchor=north west][inner sep=0.75pt]   [align=left] {{\small if}};
\draw (89.33,131.67) node [anchor=north west][inner sep=0.75pt]   [align=left] {$\displaystyle \sum _{e} \Bigl(\bdot{C_{e}}D_{e}+C_{e}\bdot{D_{e}}\Bigr) \ +\ \bdot{B_{\zeta _{x}^{-}}}$};
\draw (467.1,191) node [anchor=north west][inner sep=0.75pt]  [font=\scriptsize] [align=left] {$\displaystyle \vdots $};
\draw (448.33,183.33) node [anchor=north west][inner sep=0.75pt]  [font=\tiny] [align=left] {$\displaystyle e'$};
\draw (403.5,186.33) node [anchor=north west][inner sep=0.75pt]  [font=\scriptsize] [align=left] {$\displaystyle x\ \ \ \zeta $};
\draw (352,196.73) node [anchor=north west][inner sep=0.75pt]   [align=left] {{\small if}};
\draw (77.33,179.67) node [anchor=north west][inner sep=0.75pt]   [align=left] {$\displaystyle -\sum _{e'} \Bigl(\bdot{D_{e'}}C_{e'}+D_{e'}\bdot{C_{e'}}\Bigr) \ -\ \bdot{B_{\zeta _{x}^{+}}}$};

\end{tikzpicture}

%% file: ex4.3.tex
\begin{tikzpicture}[x=0.75pt,y=0.75pt,yscale=-1,xscale=1]

\draw    (190.13,120.67) .. controls (191.8,119) and (193.46,119) .. (195.13,120.67) .. controls (196.8,122.34) and (198.46,122.34) .. (200.13,120.67) .. controls (201.8,119) and (203.46,119) .. (205.13,120.67) .. controls (206.8,122.34) and (208.46,122.34) .. (210.13,120.67) .. controls (211.8,119) and (213.46,119) .. (215.13,120.67) .. controls (216.8,122.34) and (218.46,122.34) .. (220.13,120.67) .. controls (221.8,119) and (223.46,119) .. (225.13,120.67) .. controls (226.8,122.34) and (228.46,122.34) .. (230.13,120.67) -- (230.13,120.67) -- (238.13,120.67) ;
\draw [shift={(240.13,120.67)}, rotate = 180] [color={rgb, 255:red, 0; green, 0; blue, 0 }  ][line width=0.75]    (8.74,-3.92) .. controls (5.56,-1.84) and (2.65,-0.53) .. (0,0) .. controls (2.65,0.53) and (5.56,1.84) .. (8.74,3.92)   ;
\draw    (140.13,120.67) .. controls (141.8,119) and (143.46,119) .. (145.13,120.67) .. controls (146.8,122.34) and (148.46,122.34) .. (150.13,120.67) .. controls (151.8,119) and (153.46,119) .. (155.13,120.67) .. controls (156.8,122.34) and (158.46,122.34) .. (160.13,120.67) .. controls (161.8,119) and (163.46,119) .. (165.13,120.67) .. controls (166.8,122.34) and (168.46,122.34) .. (170.13,120.67) .. controls (171.8,119) and (173.46,119) .. (175.13,120.67) .. controls (176.8,122.34) and (178.46,122.34) .. (180.13,120.67) .. controls (181.8,119) and (183.46,119) .. (185.13,120.67) .. controls (186.8,122.34) and (188.46,122.34) .. (190.13,120.67) -- (190.13,120.67) -- (190.13,120.67) ;
\draw [shift={(190.13,120.67)}, rotate = 45] [color={rgb, 255:red, 0; green, 0; blue, 0 }  ][line width=0.75]    (-4.47,0) -- (4.47,0)(0,4.47) -- (0,-4.47)   ;
\draw    (90.13,120.67) .. controls (91.8,119) and (93.46,119) .. (95.13,120.67) .. controls (96.8,122.34) and (98.46,122.34) .. (100.13,120.67) .. controls (101.8,119) and (103.46,119) .. (105.13,120.67) .. controls (106.8,122.34) and (108.46,122.34) .. (110.13,120.67) .. controls (111.8,119) and (113.46,119) .. (115.13,120.67) .. controls (116.8,122.34) and (118.46,122.34) .. (120.13,120.67) .. controls (121.8,119) and (123.46,119) .. (125.13,120.67) .. controls (126.8,122.34) and (128.46,122.34) .. (130.13,120.67) .. controls (131.8,119) and (133.46,119) .. (135.13,120.67) .. controls (136.8,122.34) and (138.46,122.34) .. (140.13,120.67) -- (140.13,120.67) -- (140.13,120.67) ;
\draw [shift={(140.13,120.67)}, rotate = 45] [color={rgb, 255:red, 0; green, 0; blue, 0 }  ][line width=0.75]    (-4.47,0) -- (4.47,0)(0,4.47) -- (0,-4.47)   ;

\draw (110.67,95.4) node [anchor=north west][inner sep=0.75pt]    {$0\ \ \ \ \ \ \ \ \ n\ \ \ \ \ \ \ \ 0$};

\end{tikzpicture}

%% file: locLEFT.tex
\tikzset{every picture/.style={line width=0.75pt}} 
\begin{tikzpicture}[x=0.75pt,y=0.75pt,yscale=-0.95,xscale=0.95]

\draw    (256.12,86.07) -- (274.18,100.21) -- (302.32,122.26) ;
\draw [shift={(303.89,123.49)}, rotate = 218.07] [color={rgb, 255:red, 0; green, 0; blue, 0 }  ][line width=0.75]    (8.74,-3.92) .. controls (5.56,-1.84) and (2.65,-0.53) .. (0,0) .. controls (2.65,0.53) and (5.56,1.84) .. (8.74,3.92)   ;
\draw    (258.36,162.8) -- (302.38,124.8) ;
\draw [shift={(303.89,123.49)}, rotate = 139.2] [color={rgb, 255:red, 0; green, 0; blue, 0 }  ][line width=0.75]    (8.74,-3.92) .. controls (5.56,-1.84) and (2.65,-0.53) .. (0,0) .. controls (2.65,0.53) and (5.56,1.84) .. (8.74,3.92)   ;
\draw    (303.89,123.49) .. controls (305.54,121.8) and (307.21,121.79) .. (308.89,123.44) .. controls (310.57,125.09) and (312.24,125.07) .. (313.89,123.39) .. controls (315.54,121.71) and (317.21,121.7) .. (318.89,123.35) .. controls (320.57,125) and (322.24,124.98) .. (323.89,123.3) .. controls (325.54,121.62) and (327.21,121.6) .. (328.89,123.25) .. controls (330.57,124.9) and (332.24,124.88) .. (333.89,123.2) .. controls (335.54,121.52) and (337.21,121.5) .. (338.89,123.15) .. controls (340.57,124.8) and (342.24,124.79) .. (343.89,123.11) .. controls (345.54,121.43) and (347.21,121.41) .. (348.89,123.06) .. controls (350.57,124.71) and (352.24,124.69) .. (353.89,123.01) .. controls (355.54,121.33) and (357.21,121.31) .. (358.89,122.96) .. controls (360.57,124.61) and (362.24,124.6) .. (363.89,122.92) .. controls (365.54,121.24) and (367.21,121.22) .. (368.89,122.87) .. controls (370.57,124.52) and (372.24,124.5) .. (373.89,122.82) .. controls (375.54,121.14) and (377.21,121.12) .. (378.89,122.77) .. controls (380.57,124.42) and (382.24,124.4) .. (383.89,122.72) .. controls (385.54,121.04) and (387.21,121.03) .. (388.89,122.68) .. controls (390.57,124.33) and (392.24,124.31) .. (393.89,122.63) .. controls (395.54,120.95) and (397.21,120.93) .. (398.89,122.58) .. controls (400.57,124.23) and (402.24,124.21) .. (403.89,122.53) -- (406.33,122.51) -- (414.33,122.43) ;
\draw [shift={(416.33,122.41)}, rotate = 179.45] [color={rgb, 255:red, 0; green, 0; blue, 0 }  ][line width=0.75]    (8.74,-3.92) .. controls (5.56,-1.84) and (2.65,-0.53) .. (0,0) .. controls (2.65,0.53) and (5.56,1.84) .. (8.74,3.92)   ;
\draw    (207,86.67) .. controls (208.65,84.98) and (210.31,84.96) .. (212,86.61) .. controls (213.69,88.26) and (215.35,88.24) .. (217,86.55) .. controls (218.65,84.86) and (220.31,84.84) .. (222,86.49) .. controls (223.69,88.13) and (225.35,88.11) .. (227,86.42) .. controls (228.65,84.73) and (230.31,84.71) .. (232,86.36) .. controls (233.69,88.01) and (235.35,87.99) .. (237,86.3) .. controls (238.65,84.61) and (240.31,84.59) .. (242,86.24) -- (246.12,86.19) -- (254.12,86.09) ;
\draw [shift={(256.12,86.07)}, rotate = 179.3] [color={rgb, 255:red, 0; green, 0; blue, 0 }  ][line width=0.75]    (8.74,-3.92) .. controls (5.56,-1.84) and (2.65,-0.53) .. (0,0) .. controls (2.65,0.53) and (5.56,1.84) .. (8.74,3.92)   ;
\draw    (209.24,163.4) .. controls (210.89,161.71) and (212.55,161.69) .. (214.24,163.34) .. controls (215.93,164.99) and (217.59,164.97) .. (219.24,163.28) .. controls (220.89,161.59) and (222.55,161.57) .. (224.24,163.22) .. controls (225.93,164.86) and (227.59,164.84) .. (229.24,163.15) .. controls (230.89,161.46) and (232.55,161.44) .. (234.24,163.09) .. controls (235.93,164.74) and (237.59,164.72) .. (239.24,163.03) .. controls (240.88,161.34) and (242.54,161.32) .. (244.23,162.97) -- (248.36,162.92) -- (256.36,162.82) ;
\draw [shift={(258.36,162.8)}, rotate = 179.3] [color={rgb, 255:red, 0; green, 0; blue, 0 }  ][line width=0.75]    (8.74,-3.92) .. controls (5.56,-1.84) and (2.65,-0.53) .. (0,0) .. controls (2.65,0.53) and (5.56,1.84) .. (8.74,3.92)   ;

\draw (249.38,106) node [anchor=north west][inner sep=0.75pt]  [font=\small] [align=left] {$\displaystyle \vdots $};
\draw (278,87) node [anchor=north west][inner sep=0.75pt]  [font=\footnotesize]  {$e_{1}$};
\draw (278,145) node [anchor=north west][inner sep=0.75pt]  [font=\footnotesize]  {$e_{k}$};
\draw (228,61) node [anchor=north west][inner sep=0.75pt]  [font=\normalsize]  {$v_{1}$};
\draw (228,170) node [anchor=north west][inner sep=0.75pt]  [font=\normalsize]  {$v_{k}$};
\draw (322.28,101) node [anchor=north west][inner sep=0.75pt]  [font=\normalsize]  {$v_{0} \ \ \ \ \ \ \ v_{-1} \ $};
\draw (348,116) node [anchor=north west][inner sep=0.75pt]    {$\times $};
\draw (352,131.73) node [anchor=north west][inner sep=0.75pt]  [font=\footnotesize]  {$x$};

\end{tikzpicture}

%% file: locRIGHT.tex
\tikzset{every picture/.style={line width=0.75pt}} 
\begin{tikzpicture}[x=0.75pt,y=0.75pt,yscale=-0.95,xscale=0.95]

\draw    (257.69,87.3) -- (274.18,100.21) -- (303.89,123.49) ;
\draw [shift={(256.12,86.07)}, rotate = 38.07] [color={rgb, 255:red, 0; green, 0; blue, 0 }  ][line width=0.75]    (8.74,-3.92) .. controls (5.56,-1.84) and (2.65,-0.53) .. (0,0) .. controls (2.65,0.53) and (5.56,1.84) .. (8.74,3.92)   ;
\draw    (259.87,161.49) -- (303.89,123.49) ;
\draw [shift={(258.36,162.8)}, rotate = 319.2] [color={rgb, 255:red, 0; green, 0; blue, 0 }  ][line width=0.75]    (8.74,-3.92) .. controls (5.56,-1.84) and (2.65,-0.53) .. (0,0) .. controls (2.65,0.53) and (5.56,1.84) .. (8.74,3.92)   ;
\draw    (305.89,123.47) -- (313.89,123.39) .. controls (315.54,121.71) and (317.21,121.7) .. (318.89,123.35) .. controls (320.57,125) and (322.24,124.98) .. (323.89,123.3) .. controls (325.54,121.62) and (327.21,121.6) .. (328.89,123.25) .. controls (330.57,124.9) and (332.24,124.88) .. (333.89,123.2) .. controls (335.54,121.52) and (337.21,121.5) .. (338.89,123.15) .. controls (340.57,124.8) and (342.24,124.79) .. (343.89,123.11) .. controls (345.54,121.43) and (347.21,121.41) .. (348.89,123.06) .. controls (350.57,124.71) and (352.24,124.69) .. (353.89,123.01) .. controls (355.54,121.33) and (357.21,121.31) .. (358.89,122.96) .. controls (360.57,124.61) and (362.24,124.6) .. (363.89,122.92) .. controls (365.54,121.24) and (367.21,121.22) .. (368.89,122.87) .. controls (370.57,124.52) and (372.24,124.5) .. (373.89,122.82) .. controls (375.54,121.14) and (377.21,121.12) .. (378.89,122.77) .. controls (380.57,124.42) and (382.24,124.4) .. (383.89,122.72) .. controls (385.54,121.04) and (387.21,121.03) .. (388.89,122.68) .. controls (390.57,124.33) and (392.24,124.31) .. (393.89,122.63) .. controls (395.54,120.95) and (397.21,120.93) .. (398.89,122.58) .. controls (400.57,124.23) and (402.24,124.21) .. (403.89,122.53) .. controls (405.54,120.85) and (407.21,120.83) .. (408.89,122.48) .. controls (410.57,124.13) and (412.24,124.12) .. (413.89,122.44) -- (416.33,122.41) -- (416.33,122.41) ;
\draw [shift={(303.89,123.49)}, rotate = 359.45] [color={rgb, 255:red, 0; green, 0; blue, 0 }  ][line width=0.75]    (8.74,-3.92) .. controls (5.56,-1.84) and (2.65,-0.53) .. (0,0) .. controls (2.65,0.53) and (5.56,1.84) .. (8.74,3.92)   ;
\draw    (209,86.64) -- (217,86.55) .. controls (218.65,84.86) and (220.31,84.84) .. (222,86.49) .. controls (223.69,88.13) and (225.35,88.11) .. (227,86.42) .. controls (228.65,84.73) and (230.31,84.71) .. (232,86.36) .. controls (233.69,88.01) and (235.35,87.99) .. (237,86.3) .. controls (238.65,84.61) and (240.31,84.59) .. (242,86.24) .. controls (243.69,87.89) and (245.35,87.87) .. (247,86.18) .. controls (248.65,84.49) and (250.31,84.47) .. (252,86.12) -- (256.12,86.07) -- (256.12,86.07) ;
\draw [shift={(207,86.67)}, rotate = 359.3] [color={rgb, 255:red, 0; green, 0; blue, 0 }  ][line width=0.75]    (8.74,-3.92) .. controls (5.56,-1.84) and (2.65,-0.53) .. (0,0) .. controls (2.65,0.53) and (5.56,1.84) .. (8.74,3.92)   ;
\draw    (211.24,163.38) -- (219.24,163.28) .. controls (220.89,161.59) and (222.55,161.57) .. (224.24,163.22) .. controls (225.93,164.86) and (227.59,164.84) .. (229.24,163.15) .. controls (230.89,161.46) and (232.55,161.44) .. (234.24,163.09) .. controls (235.93,164.74) and (237.59,164.72) .. (239.24,163.03) .. controls (240.88,161.34) and (242.54,161.32) .. (244.23,162.97) .. controls (245.92,164.62) and (247.58,164.6) .. (249.23,162.91) .. controls (250.88,161.22) and (252.54,161.2) .. (254.23,162.85) -- (258.36,162.8) -- (258.36,162.8) ;
\draw [shift={(209.24,163.4)}, rotate = 359.3] [color={rgb, 255:red, 0; green, 0; blue, 0 }  ][line width=0.75]    (8.74,-3.92) .. controls (5.56,-1.84) and (2.65,-0.53) .. (0,0) .. controls (2.65,0.53) and (5.56,1.84) .. (8.74,3.92)   ;

\draw (249.38,106) node [anchor=north west][inner sep=0.75pt]  [font=\small] [align=left] {$\displaystyle \vdots $};
\draw (278,87) node [anchor=north west][inner sep=0.75pt]  [font=\footnotesize]  {$e_{1}$};
\draw (278,145) node [anchor=north west][inner sep=0.75pt]  [font=\footnotesize]  {$e_{k}$};
\draw (228,61) node [anchor=north west][inner sep=0.75pt]  [font=\normalsize]  {$v_{1}$};
\draw (228,170) node [anchor=north west][inner sep=0.75pt]  [font=\normalsize]  {$v_{k}$};
\draw (322.28,101) node [anchor=north west][inner sep=0.75pt]  [font=\normalsize]  {$v_{0} \ \ \ \ \ \ \ v_{-1} \ $};
\draw (348,116) node [anchor=north west][inner sep=0.75pt]    {$\times $};
\draw (352,131.73) node [anchor=north west][inner sep=0.75pt]  [font=\footnotesize]  {$x$};

\end{tikzpicture}

%% file: loc1.tex
\tikzset{every picture/.style={line width=0.75pt}} 
\begin{tikzpicture}[x=0.75pt,y=0.75pt,yscale=-0.95,xscale=0.95]

\draw    (78.45,86.8) -- (96.51,100.95) -- (124.65,122.99) ;
\draw [shift={(126.23,124.22)}, rotate = 218.07] [color={rgb, 255:red, 0; green, 0; blue, 0 }  ][line width=0.75]    (8.74,-3.92) .. controls (5.56,-1.84) and (2.65,-0.53) .. (0,0) .. controls (2.65,0.53) and (5.56,1.84) .. (8.74,3.92)   ;
\draw    (80.69,163.53) -- (124.71,125.53) ;
\draw [shift={(126.23,124.22)}, rotate = 139.2] [color={rgb, 255:red, 0; green, 0; blue, 0 }  ][line width=0.75]    (8.74,-3.92) .. controls (5.56,-1.84) and (2.65,-0.53) .. (0,0) .. controls (2.65,0.53) and (5.56,1.84) .. (8.74,3.92)   ;
\draw    (126.23,124.22) .. controls (127.88,122.54) and (129.55,122.53) .. (131.23,124.18) .. controls (132.91,125.83) and (134.58,125.81) .. (136.23,124.13) .. controls (137.88,122.45) and (139.55,122.43) .. (141.23,124.08) .. controls (142.91,125.73) and (144.58,125.71) .. (146.23,124.03) .. controls (147.88,122.35) and (149.55,122.33) .. (151.23,123.98) .. controls (152.91,125.63) and (154.58,125.62) .. (156.23,123.94) .. controls (157.88,122.26) and (159.55,122.24) .. (161.23,123.89) .. controls (162.91,125.54) and (164.57,125.52) .. (166.22,123.84) .. controls (167.87,122.16) and (169.54,122.14) .. (171.22,123.79) .. controls (172.9,125.44) and (174.57,125.42) .. (176.22,123.74) .. controls (177.87,122.06) and (179.54,122.05) .. (181.22,123.7) .. controls (182.9,125.35) and (184.57,125.33) .. (186.22,123.65) .. controls (187.87,121.97) and (189.54,121.95) .. (191.22,123.6) .. controls (192.9,125.25) and (194.57,125.23) .. (196.22,123.55) .. controls (197.87,121.87) and (199.54,121.86) .. (201.22,123.51) .. controls (202.9,125.16) and (204.57,125.14) .. (206.22,123.46) .. controls (207.87,121.78) and (209.54,121.76) .. (211.22,123.41) .. controls (212.9,125.06) and (214.57,125.04) .. (216.22,123.36) .. controls (217.87,121.68) and (219.54,121.66) .. (221.22,123.31) .. controls (222.9,124.96) and (224.57,124.95) .. (226.22,123.27) -- (228.67,123.24) -- (236.67,123.17) ;
\draw [shift={(238.67,123.15)}, rotate = 179.45] [color={rgb, 255:red, 0; green, 0; blue, 0 }  ][line width=0.75]    (8.74,-3.92) .. controls (5.56,-1.84) and (2.65,-0.53) .. (0,0) .. controls (2.65,0.53) and (5.56,1.84) .. (8.74,3.92)   ;
\draw    (29.33,87.4) .. controls (30.98,85.71) and (32.64,85.69) .. (34.33,87.34) .. controls (36.02,88.99) and (37.68,88.97) .. (39.33,87.28) .. controls (40.98,85.59) and (42.64,85.57) .. (44.33,87.22) .. controls (46.02,88.87) and (47.68,88.85) .. (49.33,87.16) .. controls (50.98,85.47) and (52.64,85.45) .. (54.33,87.1) .. controls (56.02,88.75) and (57.68,88.73) .. (59.33,87.04) .. controls (60.98,85.35) and (62.64,85.33) .. (64.33,86.97) -- (68.45,86.92) -- (76.45,86.83) ;
\draw [shift={(78.45,86.8)}, rotate = 179.3] [color={rgb, 255:red, 0; green, 0; blue, 0 }  ][line width=0.75]    (8.74,-3.92) .. controls (5.56,-1.84) and (2.65,-0.53) .. (0,0) .. controls (2.65,0.53) and (5.56,1.84) .. (8.74,3.92)   ;
\draw    (31.57,164.13) .. controls (33.22,162.44) and (34.88,162.42) .. (36.57,164.07) .. controls (38.26,165.72) and (39.92,165.7) .. (41.57,164.01) .. controls (43.22,162.32) and (44.88,162.3) .. (46.57,163.95) .. controls (48.26,165.6) and (49.92,165.58) .. (51.57,163.89) .. controls (53.22,162.2) and (54.88,162.18) .. (56.57,163.83) .. controls (58.26,165.48) and (59.92,165.46) .. (61.57,163.77) .. controls (63.22,162.08) and (64.88,162.06) .. (66.57,163.7) -- (70.69,163.65) -- (78.69,163.56) ;
\draw [shift={(80.69,163.53)}, rotate = 179.3] [color={rgb, 255:red, 0; green, 0; blue, 0 }  ][line width=0.75]    (8.74,-3.92) .. controls (5.56,-1.84) and (2.65,-0.53) .. (0,0) .. controls (2.65,0.53) and (5.56,1.84) .. (8.74,3.92)   ;
\draw    (353.69,86.96) -- (370.18,99.87) -- (399.89,123.15) ;
\draw [shift={(352.12,85.72)}, rotate = 38.07] [color={rgb, 255:red, 0; green, 0; blue, 0 }  ][line width=0.75]    (8.74,-3.92) .. controls (5.56,-1.84) and (2.65,-0.53) .. (0,0) .. controls (2.65,0.53) and (5.56,1.84) .. (8.74,3.92)   ;
\draw    (355.87,161.15) -- (399.89,123.15) ;
\draw [shift={(354.36,162.45)}, rotate = 319.2] [color={rgb, 255:red, 0; green, 0; blue, 0 }  ][line width=0.75]    (8.74,-3.92) .. controls (5.56,-1.84) and (2.65,-0.53) .. (0,0) .. controls (2.65,0.53) and (5.56,1.84) .. (8.74,3.92)   ;
\draw    (401.89,123.13) -- (409.89,123.05) .. controls (411.54,121.37) and (413.21,121.35) .. (414.89,123) .. controls (416.57,124.65) and (418.24,124.64) .. (419.89,122.96) .. controls (421.54,121.28) and (423.21,121.26) .. (424.89,122.91) .. controls (426.57,124.56) and (428.24,124.54) .. (429.89,122.86) .. controls (431.54,121.18) and (433.21,121.16) .. (434.89,122.81) .. controls (436.57,124.46) and (438.24,124.44) .. (439.89,122.76) .. controls (441.54,121.08) and (443.21,121.07) .. (444.89,122.72) .. controls (446.57,124.37) and (448.24,124.35) .. (449.89,122.67) .. controls (451.54,120.99) and (453.21,120.97) .. (454.89,122.62) .. controls (456.57,124.27) and (458.24,124.25) .. (459.89,122.57) .. controls (461.54,120.89) and (463.21,120.87) .. (464.89,122.52) .. controls (466.57,124.17) and (468.24,124.16) .. (469.89,122.48) .. controls (471.54,120.8) and (473.21,120.78) .. (474.89,122.43) .. controls (476.57,124.08) and (478.24,124.06) .. (479.89,122.38) .. controls (481.54,120.7) and (483.21,120.68) .. (484.89,122.33) .. controls (486.57,123.98) and (488.24,123.96) .. (489.89,122.28) .. controls (491.54,120.6) and (493.21,120.59) .. (494.89,122.24) .. controls (496.57,123.89) and (498.24,123.87) .. (499.89,122.19) .. controls (501.54,120.51) and (503.21,120.49) .. (504.89,122.14) .. controls (506.57,123.79) and (508.24,123.77) .. (509.89,122.09) -- (512.33,122.07) -- (512.33,122.07) ;
\draw [shift={(399.89,123.15)}, rotate = 359.45] [color={rgb, 255:red, 0; green, 0; blue, 0 }  ][line width=0.75]    (8.74,-3.92) .. controls (5.56,-1.84) and (2.65,-0.53) .. (0,0) .. controls (2.65,0.53) and (5.56,1.84) .. (8.74,3.92)   ;
\draw    (305,86.3) -- (313,86.2) .. controls (314.65,84.51) and (316.31,84.49) .. (318,86.14) .. controls (319.69,87.79) and (321.35,87.77) .. (323,86.08) .. controls (324.65,84.39) and (326.31,84.37) .. (328,86.02) .. controls (329.69,87.67) and (331.35,87.65) .. (333,85.96) .. controls (334.65,84.27) and (336.31,84.25) .. (338,85.9) .. controls (339.69,87.55) and (341.35,87.53) .. (343,85.84) .. controls (344.65,84.15) and (346.31,84.13) .. (348,85.77) -- (352.12,85.72) -- (352.12,85.72) ;
\draw [shift={(303,86.33)}, rotate = 359.3] [color={rgb, 255:red, 0; green, 0; blue, 0 }  ][line width=0.75]    (8.74,-3.92) .. controls (5.56,-1.84) and (2.65,-0.53) .. (0,0) .. controls (2.65,0.53) and (5.56,1.84) .. (8.74,3.92)   ;
\draw    (307.24,163.03) -- (315.24,162.93) .. controls (316.89,161.24) and (318.55,161.22) .. (320.24,162.87) .. controls (321.93,164.52) and (323.59,164.5) .. (325.24,162.81) .. controls (326.89,161.12) and (328.55,161.1) .. (330.24,162.75) .. controls (331.93,164.4) and (333.59,164.38) .. (335.24,162.69) .. controls (336.88,161) and (338.54,160.98) .. (340.23,162.63) .. controls (341.92,164.28) and (343.58,164.26) .. (345.23,162.57) .. controls (346.88,160.88) and (348.54,160.86) .. (350.23,162.5) -- (354.36,162.45) -- (354.36,162.45) ;
\draw [shift={(305.24,163.06)}, rotate = 359.3] [color={rgb, 255:red, 0; green, 0; blue, 0 }  ][line width=0.75]    (8.74,-3.92) .. controls (5.56,-1.84) and (2.65,-0.53) .. (0,0) .. controls (2.65,0.53) and (5.56,1.84) .. (8.74,3.92)   ;

\draw (75.05,107) node [anchor=north west][inner sep=0.75pt]  [font=\small] [align=left] {$\displaystyle \vdots $};
\draw (102,88) node [anchor=north west][inner sep=0.75pt]  [font=\footnotesize]  {$e_{1}$};
\draw (102,148) node [anchor=north west][inner sep=0.75pt]  [font=\footnotesize]  {$e_{k}$};
\draw (47,63) node [anchor=north west][inner sep=0.75pt]  [font=\normalsize]  {$v_{1}$};
\draw (49,170) node [anchor=north west][inner sep=0.75pt]  [font=\normalsize]  {$v_{k}$};
\draw (143.28,101) node [anchor=north west][inner sep=0.75pt]  [font=\normalsize]  {$v_{0} \ \ \ \ \ \ \ v_{-1} \ $};
\draw (169,116) node [anchor=north west][inner sep=0.75pt]    {$\times $};
\draw (172,131.73) node [anchor=north west][inner sep=0.75pt]  [font=\footnotesize]  {$x$};
\draw (352.05,107) node [anchor=north west][inner sep=0.75pt]  [font=\small] [align=left] {$\displaystyle \vdots $};
\draw (377,88) node [anchor=north west][inner sep=0.75pt]  [font=\footnotesize]  {$e_{1}$};
\draw (377,148) node [anchor=north west][inner sep=0.75pt]  [font=\footnotesize]  {$e_{k}$};
\draw (319.78,63.43) node [anchor=north west][inner sep=0.75pt]  [font=\normalsize]  {$v_{1}$};
\draw (325.62,170) node [anchor=north west][inner sep=0.75pt]  [font=\normalsize]  {$v_{k}$};
\draw (420.94,101) node [anchor=north west][inner sep=0.75pt]  [font=\normalsize]  {$v_{0} \ \ \ \ \ \ \ v_{-1} \ $};
\draw (447,116) node [anchor=north west][inner sep=0.75pt]    {$\times $};
\draw (452.07,129.07) node [anchor=north west][inner sep=0.75pt]  [font=\footnotesize]  {$x$};
\draw (248,121.73) node [anchor=north west][inner sep=0.75pt]    {$,$};

\end{tikzpicture}

%% file: exlocalsusy.tex
\begin{tikzpicture}[x=0.75pt,y=0.75pt,yscale=-1,xscale=1]

\draw    (90,109.67) .. controls (91.67,108) and (93.33,108) .. (95,109.67) .. controls (96.67,111.34) and (98.33,111.34) .. (100,109.67) .. controls (101.67,108) and (103.33,108) .. (105,109.67) .. controls (106.67,111.34) and (108.33,111.34) .. (110,109.67) .. controls (111.67,108) and (113.33,108) .. (115,109.67) .. controls (116.67,111.34) and (118.33,111.34) .. (120,109.67) .. controls (121.67,108) and (123.33,108) .. (125,109.67) .. controls (126.67,111.34) and (128.33,111.34) .. (130,109.67) -- (131.2,109.67) -- (139.2,109.67) ;
\draw [shift={(141.2,109.67)}, rotate = 180] [color={rgb, 255:red, 0; green, 0; blue, 0 }  ][line width=0.75]    (10.93,-4.9) .. controls (6.95,-2.3) and (3.31,-0.67) .. (0,0) .. controls (3.31,0.67) and (6.95,2.3) .. (10.93,4.9)   ;
\draw    (141.2,109.67) -- (190.4,109.67) ;
\draw [shift={(192.4,109.67)}, rotate = 180] [color={rgb, 255:red, 0; green, 0; blue, 0 }  ][line width=0.75]    (10.93,-4.9) .. controls (6.95,-2.3) and (3.31,-0.67) .. (0,0) .. controls (3.31,0.67) and (6.95,2.3) .. (10.93,4.9)   ;
\draw    (192.4,109.67) .. controls (194.07,108) and (195.73,108) .. (197.4,109.67) .. controls (199.07,111.34) and (200.73,111.34) .. (202.4,109.67) .. controls (204.07,108) and (205.73,108) .. (207.4,109.67) .. controls (209.07,111.34) and (210.73,111.34) .. (212.4,109.67) .. controls (214.07,108) and (215.73,108) .. (217.4,109.67) .. controls (219.07,111.34) and (220.73,111.34) .. (222.4,109.67) .. controls (224.07,108) and (225.73,108) .. (227.4,109.67) .. controls (229.07,111.34) and (230.73,111.34) .. (232.4,109.67) .. controls (234.07,108) and (235.73,108) .. (237.4,109.67) .. controls (239.07,111.34) and (240.73,111.34) .. (242.4,109.67) .. controls (244.07,108) and (245.73,108) .. (247.4,109.67) .. controls (249.07,111.34) and (250.73,111.34) .. (252.4,109.67) .. controls (254.07,108) and (255.73,108) .. (257.4,109.67) .. controls (259.07,111.34) and (260.73,111.34) .. (262.4,109.67) .. controls (264.07,108) and (265.73,108) .. (267.4,109.67) .. controls (269.07,111.34) and (270.73,111.34) .. (272.4,109.67) .. controls (274.07,108) and (275.73,108) .. (277.4,109.67) .. controls (279.07,111.34) and (280.73,111.34) .. (282.4,109.67) .. controls (284.07,108) and (285.73,108) .. (287.4,109.67) .. controls (289.07,111.34) and (290.73,111.34) .. (292.4,109.67) .. controls (294.07,108) and (295.73,108) .. (297.4,109.67) .. controls (299.07,111.34) and (300.73,111.34) .. (302.4,109.67) .. controls (304.07,108) and (305.73,108) .. (307.4,109.67) .. controls (309.07,111.34) and (310.73,111.34) .. (312.4,109.67) .. controls (314.07,108) and (315.73,108) .. (317.4,109.67) .. controls (319.07,111.34) and (320.73,111.34) .. (322.4,109.67) -- (323.2,109.67) -- (331.2,109.67) ;
\draw [shift={(333.2,109.67)}, rotate = 180] [color={rgb, 255:red, 0; green, 0; blue, 0 }  ][line width=0.75]    (10.93,-4.9) .. controls (6.95,-2.3) and (3.31,-0.67) .. (0,0) .. controls (3.31,0.67) and (6.95,2.3) .. (10.93,4.9)   ;

\draw (107.33,87.73) node [anchor=north west][inner sep=0.75pt]  [font=\small]  {$0\ \ \ \ \ \ \ \ \ \ \ \ \ \ \ \ \ \ \ \ \ 3\ \ \ \ \ \ \ \ \ \ 2\ \ \ \ \ \ \ \ \ 0$};
\draw (228,102.8) node [anchor=north west][inner sep=0.75pt]    {$\times \ \ \ \ \ \ \ \ \times $};

\end{tikzpicture}

%% file: HWdiagramIntro.tex
\begin{tikzpicture}[x=0.75pt,y=0.75pt,yscale=-1,xscale=1]

\draw    (46.3,117.29) -- (54.3,117.29) .. controls (55.97,115.62) and (57.63,115.62) .. (59.3,117.29) .. controls (60.97,118.96) and (62.63,118.96) .. (64.3,117.29) .. controls (65.97,115.62) and (67.63,115.62) .. (69.3,117.29) -- (73.7,117.29) -- (73.7,117.29) ;
\draw [shift={(44.3,117.29)}, rotate = 0] [color={rgb, 255:red, 0; green, 0; blue, 0 }  ][line width=0.75]    (6.56,-2.94) .. controls (4.17,-1.38) and (1.99,-0.4) .. (0,0) .. controls (1.99,0.4) and (4.17,1.38) .. (6.56,2.94)   ;
\draw    (105.1,117.29) -- (113.1,117.29) .. controls (114.77,115.62) and (116.43,115.62) .. (118.1,117.29) .. controls (119.77,118.96) and (121.43,118.96) .. (123.1,117.29) .. controls (124.77,115.62) and (126.43,115.62) .. (128.1,117.29) .. controls (129.77,118.96) and (131.43,118.96) .. (133.1,117.29) .. controls (134.77,115.62) and (136.43,115.62) .. (138.1,117.29) .. controls (139.77,118.96) and (141.43,118.96) .. (143.1,117.29) .. controls (144.77,115.62) and (146.43,115.62) .. (148.1,117.29) .. controls (149.77,118.96) and (151.43,118.96) .. (153.1,117.29) .. controls (154.77,115.62) and (156.43,115.62) .. (158.1,117.29) -- (161.9,117.29) -- (161.9,117.29) ;
\draw [shift={(103.1,117.29)}, rotate = 0] [color={rgb, 255:red, 0; green, 0; blue, 0 }  ][line width=0.75]    (6.56,-2.94) .. controls (4.17,-1.38) and (1.99,-0.4) .. (0,0) .. controls (1.99,0.4) and (4.17,1.38) .. (6.56,2.94)   ;
\draw    (75.7,117.29) -- (103.1,117.29) ;
\draw [shift={(73.7,117.29)}, rotate = 0] [color={rgb, 255:red, 0; green, 0; blue, 0 }  ][line width=0.75]    (6.56,-2.94) .. controls (4.17,-1.38) and (1.99,-0.4) .. (0,0) .. controls (1.99,0.4) and (4.17,1.38) .. (6.56,2.94)   ;
\draw    (275.7,117.29) -- (303.1,117.29) ;
\draw [shift={(273.7,117.29)}, rotate = 0] [color={rgb, 255:red, 0; green, 0; blue, 0 }  ][line width=0.75]    (6.56,-2.94) .. controls (4.17,-1.38) and (1.99,-0.4) .. (0,0) .. controls (1.99,0.4) and (4.17,1.38) .. (6.56,2.94)   ;
\draw    (305.1,117.29) -- (313.1,117.29) .. controls (314.77,115.62) and (316.43,115.62) .. (318.1,117.29) .. controls (319.77,118.96) and (321.43,118.96) .. (323.1,117.29) .. controls (324.77,115.62) and (326.43,115.62) .. (328.1,117.29) -- (332.5,117.29) -- (332.5,117.29) ;
\draw [shift={(303.1,117.29)}, rotate = 0] [color={rgb, 255:red, 0; green, 0; blue, 0 }  ][line width=0.75]    (6.56,-2.94) .. controls (4.17,-1.38) and (1.99,-0.4) .. (0,0) .. controls (1.99,0.4) and (4.17,1.38) .. (6.56,2.94)   ;
\draw    (216.9,117.29) -- (224.9,117.29) .. controls (226.57,115.62) and (228.23,115.62) .. (229.9,117.29) .. controls (231.57,118.96) and (233.23,118.96) .. (234.9,117.29) .. controls (236.57,115.62) and (238.23,115.62) .. (239.9,117.29) .. controls (241.57,118.96) and (243.23,118.96) .. (244.9,117.29) .. controls (246.57,115.62) and (248.23,115.62) .. (249.9,117.29) .. controls (251.57,118.96) and (253.23,118.96) .. (254.9,117.29) .. controls (256.57,115.62) and (258.23,115.62) .. (259.9,117.29) .. controls (261.57,118.96) and (263.23,118.96) .. (264.9,117.29) .. controls (266.57,115.62) and (268.23,115.62) .. (269.9,117.29) -- (273.7,117.29) -- (273.7,117.29) ;
\draw [shift={(214.9,117.29)}, rotate = 0] [color={rgb, 255:red, 0; green, 0; blue, 0 }  ][line width=0.75]    (6.56,-2.94) .. controls (4.17,-1.38) and (1.99,-0.4) .. (0,0) .. controls (1.99,0.4) and (4.17,1.38) .. (6.56,2.94)   ;

\draw (126,111) node [anchor=north west][inner sep=0.75pt]    {$\times $};
\draw (59.36,94) node [anchor=north west][inner sep=0.75pt]    {$v^{-}\ \ \ \ \ \ \ \ v\ \ \ \ \ v^{+}$};
\draw (239,111) node [anchor=north west][inner sep=0.75pt]    {$\times $};
\draw (226.43,94) node [anchor=north west][inner sep=0.75pt]    {$v^{-}\ \ v'\ \ \ \ \ \ \ \ \ \ \ v^{+}$};
\draw (174.66,114) node [anchor=north west][inner sep=0.75pt]  [font=\small]  {$\underset{HW}{\longleftrightarrow}$};
\draw (346,108) node [anchor=north west][inner sep=0.75pt]    {$v+v'=v^{-} +v^{+} +1$};

\end{tikzpicture}

%% file: exlocalsusy2.tex
\begin{tikzpicture}[x=0.75pt,y=0.75pt,yscale=-1,xscale=1]

\draw    (100,160.67) .. controls (101.67,159) and (103.33,159) .. (105,160.67) .. controls (106.67,162.34) and (108.33,162.34) .. (110,160.67) .. controls (111.67,159) and (113.33,159) .. (115,160.67) .. controls (116.67,162.34) and (118.33,162.34) .. (120,160.67) .. controls (121.67,159) and (123.33,159) .. (125,160.67) .. controls (126.67,162.34) and (128.33,162.34) .. (130,160.67) .. controls (131.67,159) and (133.33,159) .. (135,160.67) .. controls (136.67,162.34) and (138.33,162.34) .. (140,160.67) .. controls (141.67,159) and (143.33,159) .. (145,160.67) .. controls (146.67,162.34) and (148.33,162.34) .. (150,160.67) .. controls (151.67,159) and (153.33,159) .. (155,160.67) .. controls (156.67,162.34) and (158.33,162.34) .. (160,160.67) .. controls (161.67,159) and (163.33,159) .. (165,160.67) .. controls (166.67,162.34) and (168.33,162.34) .. (170,160.67) .. controls (171.67,159) and (173.33,159) .. (175,160.67) .. controls (176.67,162.34) and (178.33,162.34) .. (180,160.67) .. controls (181.67,159) and (183.33,159) .. (185,160.67) -- (189.87,160.67) -- (197.87,160.67) ;
\draw [shift={(199.87,160.67)}, rotate = 180] [color={rgb, 255:red, 0; green, 0; blue, 0 }  ][line width=0.75]    (8.74,-3.92) .. controls (5.56,-1.84) and (2.65,-0.53) .. (0,0) .. controls (2.65,0.53) and (5.56,1.84) .. (8.74,3.92)   ;
\draw    (199.87,160.67) -- (247.73,160.67) ;
\draw [shift={(249.73,160.67)}, rotate = 180] [color={rgb, 255:red, 0; green, 0; blue, 0 }  ][line width=0.75]    (8.74,-3.92) .. controls (5.56,-1.84) and (2.65,-0.53) .. (0,0) .. controls (2.65,0.53) and (5.56,1.84) .. (8.74,3.92)   ;
\draw    (249.73,160.67) .. controls (251.4,159) and (253.06,159) .. (254.73,160.67) .. controls (256.4,162.34) and (258.06,162.34) .. (259.73,160.67) .. controls (261.4,159) and (263.06,159) .. (264.73,160.67) .. controls (266.4,162.34) and (268.06,162.34) .. (269.73,160.67) .. controls (271.4,159) and (273.06,159) .. (274.73,160.67) .. controls (276.4,162.34) and (278.06,162.34) .. (279.73,160.67) .. controls (281.4,159) and (283.06,159) .. (284.73,160.67) .. controls (286.4,162.34) and (288.06,162.34) .. (289.73,160.67) .. controls (291.4,159) and (293.06,159) .. (294.73,160.67) .. controls (296.4,162.34) and (298.06,162.34) .. (299.73,160.67) .. controls (301.4,159) and (303.06,159) .. (304.73,160.67) .. controls (306.4,162.34) and (308.06,162.34) .. (309.73,160.67) .. controls (311.4,159) and (313.06,159) .. (314.73,160.67) .. controls (316.4,162.34) and (318.06,162.34) .. (319.73,160.67) .. controls (321.4,159) and (323.06,159) .. (324.73,160.67) .. controls (326.4,162.34) and (328.06,162.34) .. (329.73,160.67) .. controls (331.4,159) and (333.06,159) .. (334.73,160.67) -- (339.6,160.67) -- (347.6,160.67) ;
\draw [shift={(349.6,160.67)}, rotate = 180] [color={rgb, 255:red, 0; green, 0; blue, 0 }  ][line width=0.75]    (8.74,-3.92) .. controls (5.56,-1.84) and (2.65,-0.53) .. (0,0) .. controls (2.65,0.53) and (5.56,1.84) .. (8.74,3.92)   ;

\draw (138,155) node [anchor=north west][inner sep=0.75pt]    {$\times $};
\draw (281.33,155) node [anchor=north west][inner sep=0.75pt]    {$\times $};
\draw (122,140.73) node [anchor=north west][inner sep=0.75pt]    {$0\ \ \ \ \ \ \ 0\ \ \ \ \ \ \ \ \ \ \ \ \ \ \ \ \ \ \ 2\ \ \ \ \ \ \ 0\ $};

\end{tikzpicture}

%% file: bowA1.tex
\begin{tikzpicture}[x=0.75pt,y=0.75pt,yscale=-1,xscale=1]

\draw    (110,59) .. controls (111.67,57.33) and (113.33,57.33) .. (115,59) .. controls (116.67,60.67) and (118.33,60.67) .. (120,59) .. controls (121.67,57.33) and (123.33,57.33) .. (125,59) .. controls (126.67,60.67) and (128.33,60.67) .. (130,59) .. controls (131.67,57.33) and (133.33,57.33) .. (135,59) .. controls (136.67,60.67) and (138.33,60.67) .. (140,59) .. controls (141.67,57.33) and (143.33,57.33) .. (145,59) .. controls (146.67,60.67) and (148.33,60.67) .. (150,59) .. controls (151.67,57.33) and (153.33,57.33) .. (155,59) .. controls (156.67,60.67) and (158.33,60.67) .. (160,59) .. controls (161.67,57.33) and (163.33,57.33) .. (165,59) .. controls (166.67,60.67) and (168.33,60.67) .. (170,59) .. controls (171.67,57.33) and (173.33,57.33) .. (175,59) .. controls (176.67,60.67) and (178.33,60.67) .. (180,59) .. controls (181.67,57.33) and (183.33,57.33) .. (185,59) .. controls (186.67,60.67) and (188.33,60.67) .. (190,59) .. controls (191.67,57.33) and (193.33,57.33) .. (195,59) .. controls (196.67,60.67) and (198.33,60.67) .. (200,59) .. controls (201.67,57.33) and (203.33,57.33) .. (205,59) .. controls (206.67,60.67) and (208.33,60.67) .. (210,59) .. controls (211.67,57.33) and (213.33,57.33) .. (215,59) .. controls (216.67,60.67) and (218.33,60.67) .. (220,59) .. controls (221.67,57.33) and (223.33,57.33) .. (225,59) .. controls (226.67,60.67) and (228.33,60.67) .. (230,59) .. controls (231.67,57.33) and (233.33,57.33) .. (235,59) .. controls (236.67,60.67) and (238.33,60.67) .. (240,59) .. controls (241.67,57.33) and (243.33,57.33) .. (245,59) .. controls (246.67,60.67) and (248.33,60.67) .. (250,59) .. controls (251.67,57.33) and (253.33,57.33) .. (255,59) .. controls (256.67,60.67) and (258.33,60.67) .. (260,59) .. controls (261.67,57.33) and (263.33,57.33) .. (265,59) -- (268,59) -- (276,59) ;
\draw [shift={(278,59)}, rotate = 180] [color={rgb, 255:red, 0; green, 0; blue, 0 }  ][line width=0.75]    (8.74,-3.92) .. controls (5.56,-1.84) and (2.65,-0.53) .. (0,0) .. controls (2.65,0.53) and (5.56,1.84) .. (8.74,3.92)   ;

\draw (35,49.4) node [anchor=north west][inner sep=0.75pt]    {$(B,\Lambda,v):$};
\draw (152,53.5) node [anchor=north west][inner sep=0.75pt]    {$\times \ \ \ \ \ \ \ \ \ \ \ \times$};
\draw (125,36.07) node [anchor=north west][inner sep=0.75pt]    {$1\ \ \ \ \ \ \ \ \ \ \ \ 1\ \ \ \ \ \ \ \ \ \ \ \ 1$};
\draw (300,56) node [anchor=north west][inner sep=0.75pt]    {.};

\end{tikzpicture}

%% file: BowEX-1.tex
\begin{tikzpicture}[x=0.75pt,y=0.75pt,yscale=-1,xscale=1]

\draw    (145.51,91.18) .. controls (125.78,81.88) and (149.71,63.63) .. (154.23,81.48) node[midway, above] {\small$B_{0}^{-}$};
\draw [shift={(154.59,83.24)}, rotate = 254.89] [color={rgb, 255:red, 0; green, 0; blue, 0 }  ][line width=0.75]    (6.56,-2.94) .. controls (4.17,-1.38) and (1.99,-0.4) .. (0,0) .. controls (1.99,0.4) and (4.17,1.38) .. (6.56,2.94)   ;
\draw    (222.51,91.18) .. controls (202.78,81.88) and (226.71,63.63) .. (231.23,81.48) node[midway, above] {\small$B_{0}^{+}$};
\draw [shift={(231.59,83.24)}, rotate = 254.89] [color={rgb, 255:red, 0; green, 0; blue, 0 }  ][line width=0.75]    (6.56,-2.94) .. controls (4.17,-1.38) and (1.99,-0.4) .. (0,0) .. controls (1.99,0.4) and (4.17,1.38) .. (6.56,2.94)   ;
\draw    (313.36,82.34) .. controls (318.46,63.75) and (341.39,82.76) .. (322.93,91.18) node[midway, above] {\small$B_{1}^{+}$};
\draw [shift={(312.88,84.51)}, rotate = 285.36] [color={rgb, 255:red, 0; green, 0; blue, 0 }  ][line width=0.75]    (6.56,-2.94) .. controls (4.17,-1.38) and (1.99,-0.4) .. (0,0) .. controls (1.99,0.4) and (4.17,1.38) .. (6.56,2.94)   ;
\draw    (238.36,82.34) .. controls (243.46,63.75) and (266.39,82.76) .. (247.93,91.18) node[midway, above] {\small$B_{1}^{-}$};
\draw [shift={(237.88,84.51)}, rotate = 285.36] [color={rgb, 255:red, 0; green, 0; blue, 0 }  ][line width=0.75]    (6.56,-2.94) .. controls (4.17,-1.38) and (1.99,-0.4) .. (0,0) .. controls (1.99,0.4) and (4.17,1.38) .. (6.56,2.94)   ;

\draw (149.99,90.54) node [anchor=north west][inner sep=0.75pt]  [font=\small] [align=left] {$\displaystyle \mathbb{C}$};
\draw (188.85,133.22) node [anchor=north west][inner sep=0.75pt]  [font=\small] [align=left] {$\displaystyle \mathbb{C}$};
\draw (227.06,90.54) node [anchor=north west][inner sep=0.75pt]  [font=\small] [align=left] {$\displaystyle \mathbb{C}$};
\draw (154.53,118.51) node [anchor=north west][inner sep=0.75pt]  [font=\tiny] [align=left] {$ $};
\draw (305.57,90.54) node [anchor=north west][inner sep=0.75pt]  [font=\small] [align=left] {$\displaystyle \mathbb{C}$};
\draw (268.35,132.72) node [anchor=north west][inner sep=0.75pt]  [font=\small] [align=left] {$\displaystyle \mathbb{C}$};
\draw    (242.06,98.54) -- (300.57,98.54) node[midway, above] {\small $A_{1}$};
\draw [shift={(302.57,98.54)}, rotate = 180] [color={rgb, 255:red, 0; green, 0; blue, 0 }  ][line width=0.75]    (6.56,-2.94) .. controls (4.17,-1.38) and (1.99,-0.4) .. (0,0) .. controls (1.99,0.4) and (4.17,1.38) .. (6.56,2.94)   ;
\draw    (164.99,98.54) -- (222.06,98.54) node[midway, above] {\small $A_0$};
\draw [shift={(224.06,98.54)}, rotate = 180] [color={rgb, 255:red, 0; green, 0; blue, 0 }  ][line width=0.75]    (6.56,-2.94) .. controls (4.17,-1.38) and (1.99,-0.4) .. (0,0) .. controls (1.99,0.4) and (4.17,1.38) .. (6.56,2.94)   ;
\draw    (203.85,131.17) -- (222.73,110.09) node[midway, below right] {\small$a_0$};
\draw [shift={(224.06,108.6)}, rotate = 131.84] [color={rgb, 255:red, 0; green, 0; blue, 0 }  ][line width=0.75]    (6.56,-2.94) .. controls (4.17,-1.38) and (1.99,-0.4) .. (0,0) .. controls (1.99,0.4) and (4.17,1.38) .. (6.56,2.94)   ;
\draw    (164.99,108.43) -- (184.51,129.86) node[midway, below left] {\small$b_0$};
\draw [shift={(185.85,131.34)}, rotate = 227.68] [color={rgb, 255:red, 0; green, 0; blue, 0 }  ][line width=0.75]    (6.56,-2.94) .. controls (4.17,-1.38) and (1.99,-0.4) .. (0,0) .. controls (1.99,0.4) and (4.17,1.38) .. (6.56,2.94)   ;
\draw    (242.06,107.74) -- (263.96,130.1) node[midway, below left] {\small$b_1$};
\draw [shift={(265.35,131.53)}, rotate = 225.61] [color={rgb, 255:red, 0; green, 0; blue, 0 }  ][line width=0.75]    (6.56,-2.94) .. controls (4.17,-1.38) and (1.99,-0.4) .. (0,0) .. controls (1.99,0.4) and (4.17,1.38) .. (6.56,2.94)   ;
\draw    (283.35,130.52) -- (301.25,110.24) node[midway, below right] {\small$a_1$};
\draw [shift={(302.57,108.74)}, rotate = 131.42] [color={rgb, 255:red, 0; green, 0; blue, 0 }  ][line width=0.75]    (6.56,-2.94) .. controls (4.17,-1.38) and (1.99,-0.4) .. (0,0) .. controls (1.99,0.4) and (4.17,1.38) .. (6.56,2.94)   ;

\end{tikzpicture}

%% file: BowEX-2.tex
\begin{tikzcd}
\mathbb{C} \arrow["B_0", loop, distance=2em, in=55, out=125] \arrow[rr, "A_0"] \arrow[rd, "b_0"'] &                               & \mathbb{C} \arrow["B_1", loop, distance=2em, in=55, out=125] \arrow[rr, "A_1"] \arrow[rd, "b_1"'] &                               & \mathbb{C} \arrow["B_2", loop, distance=2em, in=55, out=125] \\
                                                                                                  & \mathbb{C} \arrow[ru, "a_0"'] &                                                                                                   & \mathbb{C} \arrow[ru, "a_1"'] &                                                             
\end{tikzcd}

%% file: BowEX-3.tex
\begin{tikzcd}
\TM\HR H \arrow[rr, "\sim"] \arrow[rd, "\hat{\mu}"'] &            & {Rep(\overline{Q^F},1,2)} \arrow[ld, "{\mu_{1,2}}"] & {\Bigl(B_0,a_0,a_1,b_1,b_2\Bigr)} \arrow[rd, maps to] \arrow[rr, maps to] &        & {\Bigl(I,J\Bigr)} \arrow[ld, maps to] \\
                                                           & \mathbb{C} &                                                     &                                                                           & B_0=IJ &                                      
\end{tikzcd}

%% file: xptsorder.tex
\tikzset{every picture/.style={line width=0.75pt}} 

\begin{tikzpicture}[x=0.75pt,y=0.75pt,yscale=-1,xscale=1]

\draw    (100.1,120.05) .. controls (101.77,118.38) and (103.43,118.38) .. (105.1,120.05) .. controls (106.77,121.72) and (108.43,121.72) .. (110.1,120.05) .. controls (111.77,118.38) and (113.43,118.38) .. (115.1,120.05) .. controls (116.77,121.72) and (118.43,121.72) .. (120.1,120.05) .. controls (121.77,118.38) and (123.43,118.38) .. (125.1,120.05) .. controls (126.77,121.72) and (128.43,121.72) .. (130.1,120.05) .. controls (131.77,118.38) and (133.43,118.38) .. (135.1,120.05) .. controls (136.77,121.72) and (138.43,121.72) .. (140.1,120.05) .. controls (141.77,118.38) and (143.43,118.38) .. (145.1,120.05) .. controls (146.77,121.72) and (148.43,121.72) .. (150.1,120.05) -- (150.6,120.05) -- (150.6,120.05) ;
\draw [shift={(150.6,120.05)}, rotate = 45] [color={rgb, 255:red, 0; green, 0; blue, 0 }  ][line width=0.75]    (-4.47,0) -- (4.47,0)(0,4.47) -- (0,-4.47)   ;
\draw    (150.6,120.05) .. controls (152.27,118.38) and (153.93,118.38) .. (155.6,120.05) .. controls (157.27,121.72) and (158.93,121.72) .. (160.6,120.05) .. controls (162.27,118.38) and (163.93,118.38) .. (165.6,120.05) .. controls (167.27,121.72) and (168.93,121.72) .. (170.6,120.05) .. controls (172.27,118.38) and (173.93,118.38) .. (175.6,120.05) .. controls (177.27,121.72) and (178.93,121.72) .. (180.6,120.05) .. controls (182.27,118.38) and (183.93,118.38) .. (185.6,120.05) .. controls (187.27,121.72) and (188.93,121.72) .. (190.6,120.05) .. controls (192.27,118.38) and (193.93,118.38) .. (195.6,120.05) -- (200.1,120.05) -- (200.1,120.05) ;
\draw [shift={(200.1,120.05)}, rotate = 45] [color={rgb, 255:red, 0; green, 0; blue, 0 }  ][line width=0.75]    (-4.47,0) -- (4.47,0)(0,4.47) -- (0,-4.47)   ;
\draw    (200.1,120.05) .. controls (201.77,118.38) and (203.43,118.38) .. (205.1,120.05) .. controls (206.77,121.72) and (208.43,121.72) .. (210.1,120.05) .. controls (211.77,118.38) and (213.43,118.38) .. (215.1,120.05) .. controls (216.77,121.72) and (218.43,121.72) .. (220.1,120.05) .. controls (221.77,118.38) and (223.43,118.38) .. (225.1,120.05) -- (225.1,120.05) ;
\draw  [dash pattern={on 0.84pt off 2.51pt}]  (226,120.05) -- (275.2,120.05) ;
\draw    (275.2,120.05) .. controls (276.87,118.38) and (278.53,118.38) .. (280.2,120.05) .. controls (281.87,121.72) and (283.53,121.72) .. (285.2,120.05) .. controls (286.87,118.38) and (288.53,118.38) .. (290.2,120.05) .. controls (291.87,121.72) and (293.53,121.72) .. (295.2,120.05) .. controls (296.87,118.38) and (298.53,118.38) .. (300.2,120.05) .. controls (301.87,121.72) and (303.53,121.72) .. (305.2,120.05) .. controls (306.87,118.38) and (308.53,118.38) .. (310.2,120.05) .. controls (311.87,121.72) and (313.53,121.72) .. (315.2,120.05) .. controls (316.87,118.38) and (318.53,118.38) .. (320.2,120.05) .. controls (321.87,121.72) and (323.53,121.72) .. (325.2,120.05) -- (325.7,120.05) -- (325.7,120.05) ;
\draw [shift={(325.7,120.05)}, rotate = 45] [color={rgb, 255:red, 0; green, 0; blue, 0 }  ][line width=0.75]    (-4.47,0) -- (4.47,0)(0,4.47) -- (0,-4.47)   ;
\draw    (325.7,120.05) .. controls (327.37,118.38) and (329.03,118.38) .. (330.7,120.05) .. controls (332.37,121.72) and (334.03,121.72) .. (335.7,120.05) .. controls (337.37,118.38) and (339.03,118.38) .. (340.7,120.05) .. controls (342.37,121.72) and (344.03,121.72) .. (345.7,120.05) .. controls (347.37,118.38) and (349.03,118.38) .. (350.7,120.05) .. controls (352.37,121.72) and (354.03,121.72) .. (355.7,120.05) .. controls (357.37,118.38) and (359.03,118.38) .. (360.7,120.05) -- (365.2,120.05) -- (373.2,120.05) ;
\draw [shift={(375.2,120.05)}, rotate = 180] [color={rgb, 255:red, 0; green, 0; blue, 0 }  ][line width=0.75]    (8.74,-3.92) .. controls (5.56,-1.84) and (2.65,-0.53) .. (0,0) .. controls (2.65,0.53) and (5.56,1.84) .. (8.74,3.92)   ;

\draw (63.5,119) node [anchor=north west][inner sep=0.75pt]  [font=\normalsize]  {$\sigma :$};
\draw (100,95.05) node [anchor=north west][inner sep=0.75pt]    {$\zeta _{\sigma } =\zeta _{0} \ \ \ \ \ \zeta _{1} \ \ \ \ \ \ \ \ \ \ \ \ \ \ \ \ \ \ \ \ \ \ \ \zeta _{w_{\sigma } -1} \ \ \ \ \ \zeta _{w_{\sigma }}$};
\draw (146.5,129.9) node [anchor=north west][inner sep=0.75pt]    {$1\ \ \ \ \  \ \ \ \ 2 \ \ \ \ \ \ \ \ \ \ \ \ \ \ \ \ \ \ \ \ \ \ \ \ \ \ w_{\sigma }$};

\end{tikzpicture}

%% file: mu2H.tex
\tikzset{every picture/.style={line width=0.75pt}} 

\begin{tikzpicture}[x=0.75pt,y=0.75pt,yscale=-1,xscale=1]

\draw    (381.2,77.6) .. controls (382.87,75.93) and (384.53,75.93) .. (386.2,77.6) .. controls (387.87,79.27) and (389.53,79.27) .. (391.2,77.6) .. controls (392.87,75.93) and (394.53,75.93) .. (396.2,77.6) .. controls (397.87,79.27) and (399.53,79.27) .. (401.2,77.6) .. controls (402.87,75.93) and (404.53,75.93) .. (406.2,77.6) .. controls (407.87,79.27) and (409.53,79.27) .. (411.2,77.6) .. controls (412.87,75.93) and (414.53,75.93) .. (416.2,77.6) .. controls (417.87,79.27) and (419.53,79.27) .. (421.2,77.6) .. controls (422.87,75.93) and (424.53,75.93) .. (426.2,77.6) .. controls (427.87,79.27) and (429.53,79.27) .. (431.2,77.6) -- (431.2,77.6) .. controls (432.88,75.95) and (434.55,75.97) .. (436.2,77.65) .. controls (437.85,79.34) and (439.51,79.36) .. (441.2,77.71) .. controls (442.88,76.06) and (444.55,76.08) .. (446.2,77.76) .. controls (447.85,79.44) and (449.52,79.46) .. (451.2,77.81) .. controls (452.88,76.16) and (454.55,76.18) .. (456.2,77.86) .. controls (457.85,79.54) and (459.52,79.56) .. (461.2,77.91) .. controls (462.88,76.26) and (464.55,76.28) .. (466.2,77.96) -- (469.7,78) -- (477.7,78.08) ;
\draw [shift={(479.7,78.1)}, rotate = 180.59] [color={rgb, 255:red, 0; green, 0; blue, 0 }  ][line width=0.75]    (8.74,-3.92) .. controls (5.56,-1.84) and (2.65,-0.53) .. (0,0) .. controls (2.65,0.53) and (5.56,1.84) .. (8.74,3.92)   ;
\draw [shift={(406.2,77.6)}, rotate = 45] [color={rgb, 255:red, 0; green, 0; blue, 0 }  ][line width=0.75]    (-4.47,0) -- (4.47,0)(0,4.47) -- (0,-4.47)   ;
\draw [shift={(455.45,77.85)}, rotate = 45.59] [color={rgb, 255:red, 0; green, 0; blue, 0 }  ][line width=0.75]    (-4.47,0) -- (4.47,0)(0,4.47) -- (0,-4.47)   ;
\draw    (381.2,122.74) .. controls (382.87,121.07) and (384.53,121.07) .. (386.2,122.74) .. controls (387.87,124.41) and (389.53,124.41) .. (391.2,122.74) .. controls (392.87,121.07) and (394.53,121.07) .. (396.2,122.74) .. controls (397.87,124.41) and (399.53,124.41) .. (401.2,122.74) .. controls (402.87,121.07) and (404.53,121.07) .. (406.2,122.74) .. controls (407.87,124.41) and (409.53,124.41) .. (411.2,122.74) .. controls (412.87,121.07) and (414.53,121.07) .. (416.2,122.74) .. controls (417.87,124.41) and (419.53,124.41) .. (421.2,122.74) .. controls (422.87,121.07) and (424.53,121.07) .. (426.2,122.74) .. controls (427.87,124.41) and (429.53,124.41) .. (431.2,122.74) .. controls (432.87,121.07) and (434.53,121.07) .. (436.2,122.74) -- (440.2,122.74) -- (448.2,122.74) ;
\draw [shift={(451.2,122.74)}, rotate = 180] [fill={rgb, 255:red, 0; green, 0; blue, 0 }  ][line width=0.08]  [draw opacity=0] (7.14,-3.43) -- (0,0) -- (7.14,3.43) -- (4.74,0) -- cycle    ;
\draw [shift={(416.2,122.74)}, rotate = 45] [color={rgb, 255:red, 0; green, 0; blue, 0 }  ][line width=0.75]    (-3.35,0) -- (3.35,0)(0,3.35) -- (0,-3.35)   ;
\draw    (450.87,122.36) -- (476.82,139.26) ;
\draw [shift={(479.33,140.89)}, rotate = 213.07] [fill={rgb, 255:red, 0; green, 0; blue, 0 }  ][line width=0.08]  [draw opacity=0] (7.14,-3.43) -- (0,0) -- (7.14,3.43) -- (4.74,0) -- cycle    ;
\draw    (450.87,122.36) -- (475.56,104.64) ;
\draw [shift={(478,102.89)}, rotate = 144.34] [fill={rgb, 255:red, 0; green, 0; blue, 0 }  ][line width=0.08]  [draw opacity=0] (7.14,-3.43) -- (0,0) -- (7.14,3.43) -- (4.74,0) -- cycle    ;

\draw (76,65.6) node [anchor=north west][inner sep=0.75pt]  [font=\normalsize] [align=left] {$\displaystyle -( h_{x} -I_{x} J_{x}) +u_{x^{'}}^{-1} h_{x^{'}} u_{x^{'}} =0$};
\draw (335.33,69.1) node [anchor=north west][inner sep=0.75pt]   [align=left] {{if}};
\draw (399.89,50.08) node [anchor=north west][inner sep=0.75pt]  [font=\small] [align=left] {$\displaystyle x\ \ \ \ \zeta \ \ \ \ x'$};
\draw (336.68,114.24) node [anchor=north west][inner sep=0.75pt]   [align=left] {{if}};
\draw (56.38,103.24) node [anchor=north west][inner sep=0.75pt]  [font=\small] [align=left] {$\displaystyle -\sum _{s:t( s) =\zeta } D_{s} C_{s} \ -( h_{x} -I_{x} J_{x}) =0\ $};
\draw (473.1,106) node [anchor=north west][inner sep=0.75pt]  [font=\scriptsize] [align=left] {$\displaystyle \vdots $};
\draw (453.33,99) node [anchor=north west][inner sep=0.75pt]  [font=\small] [align=left] {$\displaystyle s$};
\draw (411,100.74) node [anchor=north west][inner sep=0.75pt]  [font=\small] [align=left] {$\displaystyle x\ \ \zeta $};

\end{tikzpicture}

%% file: Nahmatrix.tex
\setcounter{MaxMatrixCols}{10}
\begin{NiceMatrix}[columns-width=10mm]
\NotEmpty \Block{1-3}{} & & &  \NotEmpty\Block{1-3}{}& & & \\
   \NotEmpty \Block{3-3}{T_i^-(s)+O(s)} & \NotEmpty & \NotEmpty & \NotEmpty\Block{3-3}{O(s^{\frac{n-m-1}{2}})} & \NotEmpty & \NotEmpty & \NotEmpty \Block{3-1}{\quad m}\\
   \\
   \\
   \NotEmpty \Block{3-3}{O(s^{\frac{n-m-1}{2}})} & \NotEmpty & \NotEmpty & \NotEmpty \Block{3-3}{\frac{1}{2}\frac{\rho_i}{s}+O(s^0)} & \NotEmpty & \NotEmpty & \NotEmpty\Block{3-1}{\quad n-m}\\
   \\
   \\
    \CodeAfter
    \OverBrace[shorten,yshift=0.7mm]{2-1}{2-3}{m}
  \OverBrace[shorten,yshift=0.7mm]{2-4}{2-6}{n-m}
  \SubMatrix[{2-1}{7-6}]
    \SubMatrix{.}{2-1}{4-6}{\}}[xshift=5mm]
    \SubMatrix{.}{5-1}{7-6}{\}}[xshift=5mm]
    \SubMatrix{.}{2-1}{7-3}{\vert}[xshift=5mm]
    \tikz\draw[shorten > = 1.5em, shorten < = 1.5em](5-|1) -- (5-|7);
\end{NiceMatrix}

%% file: notation-original-description.tex
\begin{tikzpicture}[x=0.75pt,y=0.75pt,yscale=-1,xscale=1]

\draw    (100.73,144.17) .. controls (102.4,142.5) and (104.06,142.5) .. (105.73,144.17) .. controls (107.4,145.84) and (109.06,145.84) .. (110.73,144.17) .. controls (112.4,142.5) and (114.06,142.5) .. (115.73,144.17) .. controls (117.4,145.84) and (119.06,145.84) .. (120.73,144.17) .. controls (122.4,142.5) and (124.06,142.5) .. (125.73,144.17) .. controls (127.4,145.84) and (129.06,145.84) .. (130.73,144.17) .. controls (132.4,142.5) and (134.06,142.5) .. (135.73,144.17) .. controls (137.4,145.84) and (139.06,145.84) .. (140.73,144.17) .. controls (142.4,142.5) and (144.06,142.5) .. (145.73,144.17) .. controls (147.4,145.84) and (149.06,145.84) .. (150.73,144.17) .. controls (152.4,142.5) and (154.06,142.5) .. (155.73,144.17) .. controls (157.4,145.84) and (159.06,145.84) .. (160.73,144.17) .. controls (162.4,142.5) and (164.06,142.5) .. (165.73,144.17) .. controls (167.4,145.84) and (169.06,145.84) .. (170.73,144.17) .. controls (172.4,142.5) and (174.06,142.5) .. (175.73,144.17) .. controls (177.4,145.84) and (179.06,145.84) .. (180.73,144.17) .. controls (182.4,142.5) and (184.06,142.5) .. (185.73,144.17) .. controls (187.4,145.84) and (189.06,145.84) .. (190.73,144.17) .. controls (192.4,142.5) and (194.06,142.5) .. (195.73,144.17) .. controls (197.4,145.84) and (199.06,145.84) .. (200.73,144.17) .. controls (202.4,142.5) and (204.06,142.5) .. (205.73,144.17) .. controls (207.4,145.84) and (209.06,145.84) .. (210.73,144.17) .. controls (212.4,142.5) and (214.06,142.5) .. (215.73,144.17) .. controls (217.4,145.84) and (219.06,145.84) .. (220.73,144.17) .. controls (222.4,142.5) and (224.06,142.5) .. (225.73,144.17) .. controls (227.4,145.84) and (229.06,145.84) .. (230.73,144.17) .. controls (232.4,142.5) and (234.06,142.5) .. (235.73,144.17) .. controls (237.4,145.84) and (239.06,145.84) .. (240.73,144.17) -- (242.54,144.17) -- (242.54,144.17) ;
\draw [shift={(242.54,144.17)}, rotate = 0] [color={rgb, 255:red, 0; green, 0; blue, 0 }  ][fill={rgb, 255:red, 0; green, 0; blue, 0 }  ][line width=0.75]      (0, 0) circle [x radius= 2.68, y radius= 2.68]   ;
\draw [shift={(171.63,144.17)}, rotate = 45] [color={rgb, 255:red, 0; green, 0; blue, 0 }  ][line width=0.75]    (-4.47,0) -- (4.47,0)(0,4.47) -- (0,-4.47)   ;
\draw [shift={(100.73,144.17)}, rotate = 0] [color={rgb, 255:red, 0; green, 0; blue, 0 }  ][fill={rgb, 255:red, 0; green, 0; blue, 0 }  ][line width=0.75]      (0, 0) circle [x radius= 2.68, y radius= 2.68]   ;
\draw    (242.54,144.17) .. controls (244.2,142.5) and (245.87,142.49) .. (247.54,144.15) .. controls (249.21,145.81) and (250.87,145.81) .. (252.54,144.14) .. controls (254.21,142.47) and (255.87,142.47) .. (257.54,144.13) .. controls (259.21,145.79) and (260.87,145.79) .. (262.54,144.12) .. controls (264.21,142.45) and (265.87,142.45) .. (267.54,144.11) .. controls (269.21,145.77) and (270.87,145.77) .. (272.54,144.1) .. controls (274.21,142.43) and (275.87,142.43) .. (277.54,144.09) .. controls (279.21,145.75) and (280.87,145.75) .. (282.54,144.08) .. controls (284.21,142.41) and (285.87,142.41) .. (287.54,144.07) .. controls (289.21,145.73) and (290.87,145.73) .. (292.54,144.06) .. controls (294.2,142.39) and (295.87,142.38) .. (297.54,144.04) .. controls (299.21,145.7) and (300.87,145.7) .. (302.54,144.03) .. controls (304.21,142.36) and (305.87,142.36) .. (307.54,144.02) -- (312.39,144.01) -- (312.39,144.01) ;
\draw [shift={(312.39,144.01)}, rotate = 44.87] [color={rgb, 255:red, 0; green, 0; blue, 0 }  ][line width=0.75]    (-4.47,0) -- (4.47,0)(0,4.47) -- (0,-4.47)   ;
\draw    (312.39,144.01) .. controls (314.04,142.33) and (315.71,142.32) .. (317.39,143.97) .. controls (319.07,145.62) and (320.74,145.61) .. (322.39,143.93) .. controls (324.04,142.25) and (325.71,142.24) .. (327.39,143.89) .. controls (329.07,145.54) and (330.74,145.53) .. (332.39,143.85) .. controls (334.04,142.17) and (335.71,142.16) .. (337.39,143.81) .. controls (339.07,145.46) and (340.74,145.44) .. (342.39,143.76) .. controls (344.04,142.08) and (345.71,142.07) .. (347.39,143.72) .. controls (349.07,145.37) and (350.74,145.36) .. (352.39,143.68) .. controls (354.04,142) and (355.71,141.99) .. (357.39,143.64) .. controls (359.07,145.29) and (360.74,145.28) .. (362.39,143.6) .. controls (364.04,141.92) and (365.71,141.91) .. (367.39,143.56) -- (368.27,143.55) -- (368.27,143.55) ;
\draw  [dash pattern={on 0.84pt off 2.51pt}]  (368.27,143.55) -- (436.73,143.55) ;
\draw    (436.73,143.55) .. controls (438.4,141.88) and (440.06,141.88) .. (441.73,143.55) .. controls (443.4,145.22) and (445.06,145.22) .. (446.73,143.55) .. controls (448.4,141.88) and (450.06,141.88) .. (451.73,143.55) .. controls (453.4,145.22) and (455.06,145.22) .. (456.73,143.55) .. controls (458.4,141.88) and (460.06,141.88) .. (461.73,143.55) .. controls (463.4,145.22) and (465.06,145.22) .. (466.73,143.55) .. controls (468.4,141.88) and (470.06,141.88) .. (471.73,143.55) .. controls (473.4,145.22) and (475.06,145.22) .. (476.73,143.55) .. controls (478.4,141.88) and (480.06,141.88) .. (481.73,143.55) .. controls (483.4,145.22) and (485.06,145.22) .. (486.73,143.55) .. controls (488.4,141.88) and (490.06,141.88) .. (491.73,143.55) .. controls (493.4,145.22) and (495.06,145.22) .. (496.73,143.55) .. controls (498.4,141.88) and (500.06,141.88) .. (501.73,143.55) .. controls (503.4,145.22) and (505.06,145.22) .. (506.73,143.55) .. controls (508.4,141.88) and (510.06,141.88) .. (511.73,143.55) .. controls (513.4,145.22) and (515.06,145.22) .. (516.73,143.55) .. controls (518.4,141.88) and (520.06,141.88) .. (521.73,143.55) .. controls (523.4,145.22) and (525.06,145.22) .. (526.73,143.55) .. controls (528.4,141.88) and (530.06,141.88) .. (531.73,143.55) .. controls (533.4,145.22) and (535.06,145.22) .. (536.73,143.55) .. controls (538.4,141.88) and (540.06,141.88) .. (541.73,143.55) .. controls (543.4,145.22) and (545.06,145.22) .. (546.73,143.55) .. controls (548.4,141.88) and (550.06,141.88) .. (551.73,143.55) .. controls (553.4,145.22) and (555.06,145.22) .. (556.73,143.55) .. controls (558.4,141.88) and (560.06,141.88) .. (561.73,143.55) .. controls (563.4,145.22) and (565.06,145.22) .. (566.73,143.55) .. controls (568.4,141.88) and (570.06,141.88) .. (571.73,143.55) .. controls (573.4,145.22) and (575.06,145.22) .. (576.73,143.55) -- (578.53,143.55) -- (578.53,143.55) ;
\draw [shift={(578.53,143.55)}, rotate = 0] [color={rgb, 255:red, 0; green, 0; blue, 0 }  ][fill={rgb, 255:red, 0; green, 0; blue, 0 }  ][line width=0.75]      (0, 0) circle [x radius= 2.68, y radius= 2.68]   ;
\draw [shift={(507.63,143.55)}, rotate = 45] [color={rgb, 255:red, 0; green, 0; blue, 0 }  ][line width=0.75]    (-4.47,0) -- (4.47,0)(0,4.47) -- (0,-4.47)   ;
\draw [shift={(436.73,143.55)}, rotate = 0] [color={rgb, 255:red, 0; green, 0; blue, 0 }  ][fill={rgb, 255:red, 0; green, 0; blue, 0 }  ][line width=0.75]      (0, 0) circle [x radius= 2.68, y radius= 2.68]   ;
\draw    (102.73,160) -- (169.63,160) ;
\draw [shift={(171.63,160)}, rotate = 180] [color={rgb, 255:red, 0; green, 0; blue, 0 }  ][line width=0.75]    (8.74,-3.92) .. controls (5.56,-1.84) and (2.65,-0.53) .. (0,0) .. controls (2.65,0.53) and (5.56,1.84) .. (8.74,3.92)   ;
\draw [shift={(100.73,160)}, rotate = 0] [color={rgb, 255:red, 0; green, 0; blue, 0 }  ][line width=0.75]    (8.74,-3.92) .. controls (5.56,-1.84) and (2.65,-0.53) .. (0,0) .. controls (2.65,0.53) and (5.56,1.84) .. (8.74,3.92)   ;
\draw    (315.44,160) -- (367.57,160) ;
\draw [shift={(313.44,160)}, rotate = 0] [color={rgb, 255:red, 0; green, 0; blue, 0 }  ][line width=0.75]    (8.74,-3.92) .. controls (5.56,-1.84) and (2.65,-0.53) .. (0,0) .. controls (2.65,0.53) and (5.56,1.84) .. (8.74,3.92)   ;
\draw  [dash pattern={on 0.84pt off 2.51pt}]  (366.18,160) -- (436.03,160) ;
\draw    (173.63,160) -- (240.54,160) ;
\draw [shift={(242.54,160)}, rotate = 180] [color={rgb, 255:red, 0; green, 0; blue, 0 }  ][line width=0.75]    (8.74,-3.92) .. controls (5.56,-1.84) and (2.65,-0.53) .. (0,0) .. controls (2.65,0.53) and (5.56,1.84) .. (8.74,3.92)   ;
\draw [shift={(171.63,160)}, rotate = 0] [color={rgb, 255:red, 0; green, 0; blue, 0 }  ][line width=0.75]    (8.74,-3.92) .. controls (5.56,-1.84) and (2.65,-0.53) .. (0,0) .. controls (2.65,0.53) and (5.56,1.84) .. (8.74,3.92)   ;
\draw    (244.54,160) -- (311.44,160) ;
\draw [shift={(313.44,160)}, rotate = 180] [color={rgb, 255:red, 0; green, 0; blue, 0 }  ][line width=0.75]    (8.74,-3.92) .. controls (5.56,-1.84) and (2.65,-0.53) .. (0,0) .. controls (2.65,0.53) and (5.56,1.84) .. (8.74,3.92)   ;
\draw [shift={(242.54,160)}, rotate = 0] [color={rgb, 255:red, 0; green, 0; blue, 0 }  ][line width=0.75]    (8.74,-3.92) .. controls (5.56,-1.84) and (2.65,-0.53) .. (0,0) .. controls (2.65,0.53) and (5.56,1.84) .. (8.74,3.92)   ;
\draw    (438.03,160) -- (504.93,160) ;
\draw [shift={(506.93,160)}, rotate = 180] [color={rgb, 255:red, 0; green, 0; blue, 0 }  ][line width=0.75]    (8.74,-3.92) .. controls (5.56,-1.84) and (2.65,-0.53) .. (0,0) .. controls (2.65,0.53) and (5.56,1.84) .. (8.74,3.92)   ;
\draw [shift={(436.03,160)}, rotate = 0] [color={rgb, 255:red, 0; green, 0; blue, 0 }  ][line width=0.75]    (8.74,-3.92) .. controls (5.56,-1.84) and (2.65,-0.53) .. (0,0) .. controls (2.65,0.53) and (5.56,1.84) .. (8.74,3.92)   ;
\draw    (508.93,160) -- (575.83,160) ;
\draw [shift={(577.83,160)}, rotate = 180] [color={rgb, 255:red, 0; green, 0; blue, 0 }  ][line width=0.75]    (8.74,-3.92) .. controls (5.56,-1.84) and (2.65,-0.53) .. (0,0) .. controls (2.65,0.53) and (5.56,1.84) .. (8.74,3.92)   ;
\draw [shift={(506.93,160)}, rotate = 0] [color={rgb, 255:red, 0; green, 0; blue, 0 }  ][line width=0.75]    (8.74,-3.92) .. controls (5.56,-1.84) and (2.65,-0.53) .. (0,0) .. controls (2.65,0.53) and (5.56,1.84) .. (8.74,3.92)   ;

\draw (96.37,116) node [anchor=north west][inner sep=0.75pt]  [font=\small]  {$p_{\sigma ,0} =\sigma _{L} \ \ \ \ x_{\sigma ,1} \ \ \ \ \ \ \ \ \ \ \ p_{\sigma ,1} \ \ \ \ \ \ \ \ \ \ \ x_{\sigma ,2} \ \ \ \ \ \ \ \ \  \dotsc \ \ \ \ \ \ \ \ \ p_{\sigma ,w_{\sigma } -1} \ \ \ \ \ \ \ \ x_{\sigma ,w_{\sigma }} \ \ p_{\sigma ,w_{\sigma }} =\sigma _{R}$};
\draw (124.67,170.07) node [anchor=north west][inner sep=0.75pt]  [font=\small]  {$c_{\sigma ,1}^{-} \ \ \ \ \ \ \ \ \ \ \ \ c_{\sigma ,1}^{+} \ \ \ \ \ \ \ \ \ \ \ c_{\sigma ,2}^{-} \ \ \ \ \ \ \ \ \ \ \ \ c_{\sigma ,2}^{+} \ \ \ \ \ \ \ \ \ \ \ \dotsc \ \ \ \ \ \ \ \ \ \ \ c_{\sigma ,w_{\sigma }}^{-} \ \ \ \ \ \ \ \ \ \ c_{\sigma ,w_{\sigma }}^{+}$};
\draw (390.67,82.4) node [anchor=north west][inner sep=0.75pt]    {$.$};

\end{tikzpicture}

%% file: Main.bbl
\begin{thebibliography}{HKLR87}

\bibitem[BFN16]{braverman2016towards}
A.~Braverman, M.~Finkelberg, and H.~Nakajima.
\newblock {Towards a mathematical definition of Coulomb branches of $3 $-dimensional $\mathcal{N}= 4$ gauge theories, II}.
\newblock {\em arXiv preprint arXiv:1601.03586}, 2016.

\bibitem[Bie97]{Bie97}
R.~Bielawski.
\newblock {Hyperk{\"a}hler structures and group actions}.
\newblock {\em Journal of the London Mathematical Society}, 55(2):400--414, 1997.

\bibitem[Bie98]{bielawski1998asymptotic}
R.~Bielawski.
\newblock {Asymptotic metrics for SU(N)-monopoles with maximal symmetry breaking}.
\newblock {\em Communications in mathematical physics}, 199:297--325, 1998.

\bibitem[Bie07]{bielawski2007lie}
R.~Bielawski.
\newblock {Lie groups, Nahm's equations and hyperk{\"a}hler manifolds}.
\newblock {\em Algebraic groups}, pages 1--17, 2007.

\bibitem[Che11]{cherkis2011instantons}
S.A. Cherkis.
\newblock Instantons on gravitons.
\newblock {\em Communications in mathematical physics}, 306:449--483, 2011.

\bibitem[Don84]{Don84}
S.K. Donaldson.
\newblock {Nahm's equations and the classification of monopoles}.
\newblock {\em Communications in Mathematical Physics}, 96(3):387 -- 407, 1984.

\bibitem[DS97]{DS97-geometryofsingular}
A.~Dancer and A.~Swann.
\newblock {The geometry of singular quaternionic K{\"a}hler quotients}.
\newblock {\em International Journal of Mathematics}, 8(05):595--610, 1997.

\bibitem[Gin09]{ginzburg2009lectures}
V.~Ginzburg.
\newblock {Lectures on Nakajima's quiver varieties}.
\newblock {\em arXiv preprint arXiv:0905.0686}, 2009.

\bibitem[HKLR87]{hitchin1987hyperkahler}
N.J. Hitchin, A.~Karlhede, U.~Lindstr{\"o}m, and M.~Ro{\v{c}}ek.
\newblock {Hyperk{\"a}hler metrics and supersymmetry}.
\newblock {\em Communications in Mathematical Physics}, 108(4):535--589, 1987.

\bibitem[Hur89]{hurtubise1989classification}
J.~Hurtubise.
\newblock {The classification of monopoles for the classical groups}.
\newblock {\em Communications in mathematical physics}, 120(4):613--641, 1989.

\bibitem[Kin94]{king1994moduli}
A.D. King.
\newblock {Moduli of representations of finite dimensional algebras}.
\newblock {\em The Quarterly Journal of Mathematics}, 45(4):515--530, 1994.

\bibitem[Kir84]{kirwan1984cohomology}
F.~Kirwan.
\newblock {\em {Cohomology of quotients in symplectic and algebraic geometry}}, volume~31.
\newblock Princeton university press, 1984.

\bibitem[Kir16]{Ki16}
A.~Kirillov.
\newblock {\em Quiver representations and quiver varieties}, volume 174.
\newblock American Mathematical Society, 2016.

\bibitem[KN79]{kempf1979length}
G.~Kempf and L.~Ness.
\newblock The length of vectors in representation spaces.
\newblock In {\em Algebraic Geometry: Summer Meeting, Copenhagen, August 7--12, 1978}, pages 233--243. Springer, 1979.

\bibitem[Kro04]{kronheimer2004hyperkahler}
P.B. Kronheimer.
\newblock {A hyperkahler structure on the cotangent bundle of a complex Lie group}.
\newblock {\em arXiv preprint}, 2004.

\bibitem[May19]{May19}
M.~Mayrand.
\newblock {Kempf--Ness type theorems and Nahm equations}.
\newblock {\em Journal of Geometry and Physics}, 136:138--155, 2019.

\bibitem[May20]{May20}
M.~Mayrand.
\newblock {Nahm's equations in hyperk{\"a}hler geometry}.
\newblock \url{https://maxencemayrand.github.io/documents/nahm-lectures.pdf}, 2020.

\bibitem[May22]{May22}
M.~Mayrand.
\newblock {Stratification of singular hyperk{\"a}hler quotients}.
\newblock {\em Complex Manifolds}, 9(1):261--284, 2022.

\bibitem[MFK94]{mumford1994geometric}
D.~Mumford, J.~Fogarty, and F.~Kirwan.
\newblock {\em {Geometric invariant theory}}, volume~34.
\newblock Springer Science \& Business Media, 1994.

\bibitem[Muk03]{mukai2003introduction}
S.~Mukai.
\newblock {\em {An introduction to invariants and moduli}}, volume~81.
\newblock Cambridge University Press, 2003.

\bibitem[Nak94]{Nak94}
H.~Nakajima.
\newblock {Instantons on ALE spaces, quiver varieties, and Kac-Moody algebras}.
\newblock {\em Duke Mathematical Journal}, 76(2):365 -- 416, 1994.

\bibitem[Nak98]{nakajima1998quiver}
H.~Nakajima.
\newblock {Quiver varieties and Kac-Moody algebras}.
\newblock {\em Duke Mathematical Journal}, 91, 02 1998.

\bibitem[Nak99]{nakajima1999lectures}
H.~Nakajima.
\newblock {\em {Lectures on Hilbert schemes of points on surfaces}}.
\newblock American Mathematical Society, 1999.

\bibitem[NT17]{NT17}
H.~Nakajima and Y.~Takayama.
\newblock {Cherkis bow varieties and Coulomb branches of quiver gauge theories of affine type A}.
\newblock {\em Selecta Mathematica}, 23:2553--2633, 2017.

\bibitem[SW23]{SW23}
C.~Stroppel and T.~Wehrhan.
\newblock Existence and orthogonality of stable envelopes for bow varieties.
\newblock {\em arXiv preprint arXiv:2312.03144}, 2023.

\bibitem[Tak15]{T15}
Y.~Takayama.
\newblock {Bow varieties and ALF spaces}.
\newblock {\em Mathematical Proceedings of the Cambridge Philosophical Society}, 158(1):37--82, 2015.

\bibitem[Tak16]{Tak16}
Y.~Takayama.
\newblock {Nahm's equations, quiver varieties and parabolic sheaves}.
\newblock {\em Publications of the Research Institute for Mathematical Sciences}, 52(1):1--41, 2016.

\bibitem[Tho05]{thomas2005notes}
R.P. Thomas.
\newblock {Notes on GIT and symplectic reduction for bundles and varieties}.
\newblock {\em arXiv preprint math/0512411}, 2005.

\end{thebibliography}
